\documentclass[10pt,twoside,reqno]{amsart}
 \usepackage[top=1.4in, bottom=1.4in, left=1in, right=1in]{geometry}
\usepackage{graphicx}
\usepackage{amsmath,latexsym}
\usepackage{amsfonts,amssymb}
\usepackage{mathtools}
\usepackage{amsmath}
\usepackage{amscd, amsthm}
\usepackage{stmaryrd}
\usepackage{fancyhdr}
\usepackage{commath,enumerate}
\usepackage{indentfirst, hyperref}
\usepackage{caption, subcaption}
\usepackage{dsfont}

\usepackage{pgf,tikz}
\usepackage{mathrsfs}
\usetikzlibrary{arrows,calc,positioning}
\usepackage{float}
\numberwithin{figure}{section}
\numberwithin{table}{section}
\usepackage[bottom]{footmisc}

\numberwithin{equation}{section}
\newtheorem{theorem}{Theorem}[section]
\newtheorem{lemma}[theorem]{Lemma}
\newtheorem{proposition}[theorem]{Proposition}

\newtheorem{definition}[theorem]{Definition}
\newtheorem{remark}[theorem]{Remark}

\newcommand{\Rm}[1]{
  \textup{\uppercase\expandafter{\romannumeral#1}}
}
\let\olddefinition\definition
\renewcommand{\definition}{\olddefinition\normalfont}
\let\oldremark\remark
\renewcommand{\remark}{\oldremark\normalfont}

\allowdisplaybreaks[3]

\newcommand{\Bf}{\mathfrak{B}^{(\alpha)}}
\newcommand{\Bff}{\mathfrak{B}^{(1)}}
\newcommand{\C}{\mathbb{C}}

\newcommand{\etab}{\boldsymbol{\eta}}
\newcommand{\F}{\mathcal{F}}
\newcommand{\fp}{\mathrm{p.f.}}

\renewcommand{\L}{\mathbf{L}}

\newcommand{\n}{\mathbf{n}}
\newcommand{\N}{\mathbb{N}}

\newcommand{\M}{\mathcal{M}}
\renewcommand{\O}{\mathcal{O}}

\newcommand{\R}{\mathbb{R}}
\newcommand{\Rc}{\mathcal{R}}
\newcommand{\Rf}{\mathfrak{R}}
\newcommand{\T}{\mathbb{T}}
\newcommand{\Tb}{\mathbf{T}}
\renewcommand{\u}{\mathbf{u}}
\newcommand{\x}{\mathbf{x}}
\newcommand{\X}{\mathbf{X}}
\newcommand{\Z}{\mathbb{Z}}
\newcommand{\Time}{T}

\def\vp{\varphi}
\def\ve{\varepsilon}
\def\px{\partial_x}
\def\pt{\partial_t}

\def\moddy{\left|\partial_y\right|}

\newcommand{\diff}{\,\mathrm{d}}

\DeclareMathOperator{\hilbert}{\mathbf{H}}
\newcommand{\s}{\mathbf{s}}
\DeclareMathOperator{\supp}{\text{supp}}
\DeclareMathOperator{\sgn}{\mathrm{sgn}}

\usepackage{scalerel}
\usepackage{stackengine,wasysym}
\stackMath
\newcommand\reallywidehat[1]{%
\savestack{\tmpbox}{\stretchto{%
  \scaleto{%
    \scalerel*[\widthof{\ensuremath{#1}}]{\kern-.6pt\bigwedge\kern-.6pt}%
    {\rule[-\textheight/2]{1ex}{\textheight}}
  }{\textheight}%
}{0.5ex}}%
\stackon[1pt]{#1}{\tmpbox}%
}

\newcommand\reallywidetilde[1]{%
\savestack{\tmpbox}{\stretchto{%
  \scaleto{%
    \scalerel*[\widthof{\ensuremath{#1}}]{\kern-.4pt\AC\kern-.4pt}%
    {\rule[-\textheight/2]{1ex}{\textheight}}
  }{\textheight}%
}{0.5ex}}%
\stackon[1pt]{#1}{\tmpbox}%
}

\def\bel{\begin{equation}\label}
\def\beq{\begin{equation}}
\def\eeq{\end{equation}}
\def\bega{\begin{array}}
\def\enda{\end{array}}

\renewcommand{\vec}[1]{\mathbf{#1}}

\author{John K. Hunter}
\address{Department of Mathematics, University of California at Davis}
\email{jkhunter@ucdavis.edu}
\thanks{JKH was supported by the NSF under grant numbers DMS-1616988 and DMS-1908947}
\author{Jingyang Shu}
\address{Department of Mathematics, University of California at Davis}
\email{jyshu@ucdavis.edu}
\author{Qingtian Zhang}
\address{Department of Mathematics, West Virginia University}
\email{qingtian.zhang@mail.wvu.edu}
\title[Two-Front GSQG Equations]{
Two-Front Solutions of the SQG Equation and its Generalizations}
\date{\today}

\begin{document}

\begin{abstract}
The  generalized surface quasi-geostrophic (GSQG) equations are transport equations for an active scalar that depend on a parameter $0<\alpha \le 2$. Special cases are the two-dimensional incompressible Euler equations ($\alpha = 2$) and the surface quasi-geostrophic (SQG) equations ($\alpha = 1$).
We derive contour-dynamics equations for a class of two-front solutions of the GSQG equations when the fronts are a graph. Scalar reductions of these equations include ones that describe a single front in the presence of a rigid, flat boundary. We use the contour dynamics equations to determine the linearized stability of the GSQG shear flows that correspond to two flat fronts. We also prove local-in-time existence and uniqueness for large, smooth solutions of the two-front equations in the parameter regime $1<\alpha\le 2$, and  small, smooth solutions in the parameter regime $0<\alpha\le 1$.
\end{abstract}

\maketitle
\tableofcontents

\section{Introduction}

In this paper, we derive contour dynamics equations for the motion of two fronts in a class of piecewise constant solutions of the incompressible Euler, surface quasi-geostrophic (SQG), and generalized surface quasi-geostropic (GSQG) equations; these two-front solutions are described in more detail in Section~\ref{sec:twofront} below. We also prove local existence and uniqueness theorems for the resulting front-equations.

The GSQG equations are a family of active scalar equations in two spatial dimensions, depending on a parameter $0 < \alpha \leq 2$,
which arise naturally from fluid dynamics. They consist of
a transport equation for a scalar function $\theta \colon \R^2 \times \R \to \R$ that is transported by a divergence-free velocity field $\u \colon \R^2 \times \R \to \R^2$ which depends non-locally on $\theta$:
\begin{align}
\label{gsqg}
\begin{split}
& \theta_t + \u \cdot \nabla \theta = 0,\qquad
(- \Delta)^{\alpha / 2} \u = \nabla^\perp \theta.
\end{split}
\end{align}
Here, $\x = (x, y)$ is the spatial variable, $\nabla^\perp = (- \partial_y, \partial_x)$ is the perpendicular gradient,
and $(- \Delta)^{\alpha / 2}$ is the Fourier multiplier with symbol $(\xi^2+\eta^2)^{\alpha/2}$.
Alternatively, one can introduce a stream function $\psi : \R^2\times \R \to \R$, and write
\[
\u = \nabla^\perp \psi,\qquad (- \Delta)^{\alpha / 2}\psi = \theta.
\]

When $\alpha = 2$, equation \eqref{gsqg} is the vorticity-stream function formulation of the two-dimensional, incompressible Euler equation
for an inviscid fluid, and the scalar $\theta$ is the negative of the vorticity \cite{MB02}. It has long been established that the 2D Euler equation has global smooth solutions \cite{Hol33, Wol33}. Further results on the 2D Euler equation can be found in \cite{MB02, MP94} and the references therein.

When $\alpha = 1$, equation \eqref{gsqg} is the (inviscid) SQG equation. This equation describes the motion of quasi-geostrophic flows confined near a surface \cite{HPGS95, Lap17, Maj03, Ped87}, and $\theta$ is usually referred to as the potential temperature or the surface buoyancy. From an analytical point of view, the SQG equation has many similar features to the 3D incompressible Euler equation \cite{CMT94a, CMT94b}. In particular, the scalar $\theta$ has the same dimensions as the velocity field $\u$ that transports it.

The SQG equation has global weak solutions in $L^p$-spaces ($p > 4 / 3$) \cite{Mar08, Res95}, and convex integration shows that low-regularity weak solutions need not be unique \cite{BSV}. A class of nontrivial global smooth solutions is constructed in \cite{CCG}, but --- as for the 3D incompressible Euler equation --- the question of whether general smooth solutions of the SQG equation remain smooth for all time or form singularities in finite time is open.

The other cases in the family, with $0 < \alpha < 1$ or $1<\alpha < 2$, correspond to a natural generalization of the Euler and SQG equations. Local existence of smooth solutions of these equations is proved in \cite{CCCGW12}, but the global existence of smooth solutions with general initial data is not known for any $0<\alpha <2$.

\subsection{Patch and front solutions}
\label{patch_front}
Equation \eqref{gsqg} has  a class of piecewise constant solutions of the form
\begin{align}
\label{piecewisetheta}
\theta(\x, t) = \sum_{k = 1}^N \theta_k \mathds{1}_{\Omega_k(t)}(\x),
\end{align}
where $N \geq 2$ is a positive integer, $\theta_1, \dotsc, \theta_N\in \R$ are constants, and $\Omega_1(t),\dotsc, \Omega_N(t) \subset \R^2$ are disjoint domains such that
\[
\bigcup_{k = 1}^N \overline{\Omega_k(t)} = \R^2,
\]
and their boundaries $\partial \Omega_1(t), \dots, \partial \Omega_N(t)$ are smooth curves, whose components either coincide or are a positive distance apart.
In \eqref{piecewisetheta}, $\mathds{1}_{\Omega_k(t)}$ denotes the indicator function of $\Omega_k(t)$.
The transport equation \eqref{gsqg} preserves the form of these weak solutions, at least locally in time, and to study their evolution, we only need to understand the dynamics of the boundaries $\partial \Omega_k(t)$.

Depending on the number of regions and the boundedness of each region, we distinguish the following three different types of solutions (see Figure \ref{fig:patch+fronts}). In this paper, we will be concerned with the third type, which we call two-front solutions.

\begin{figure}[h]
\centering
\begin{subfigure}[]{0.45\textwidth}
\includegraphics[width=\textwidth]{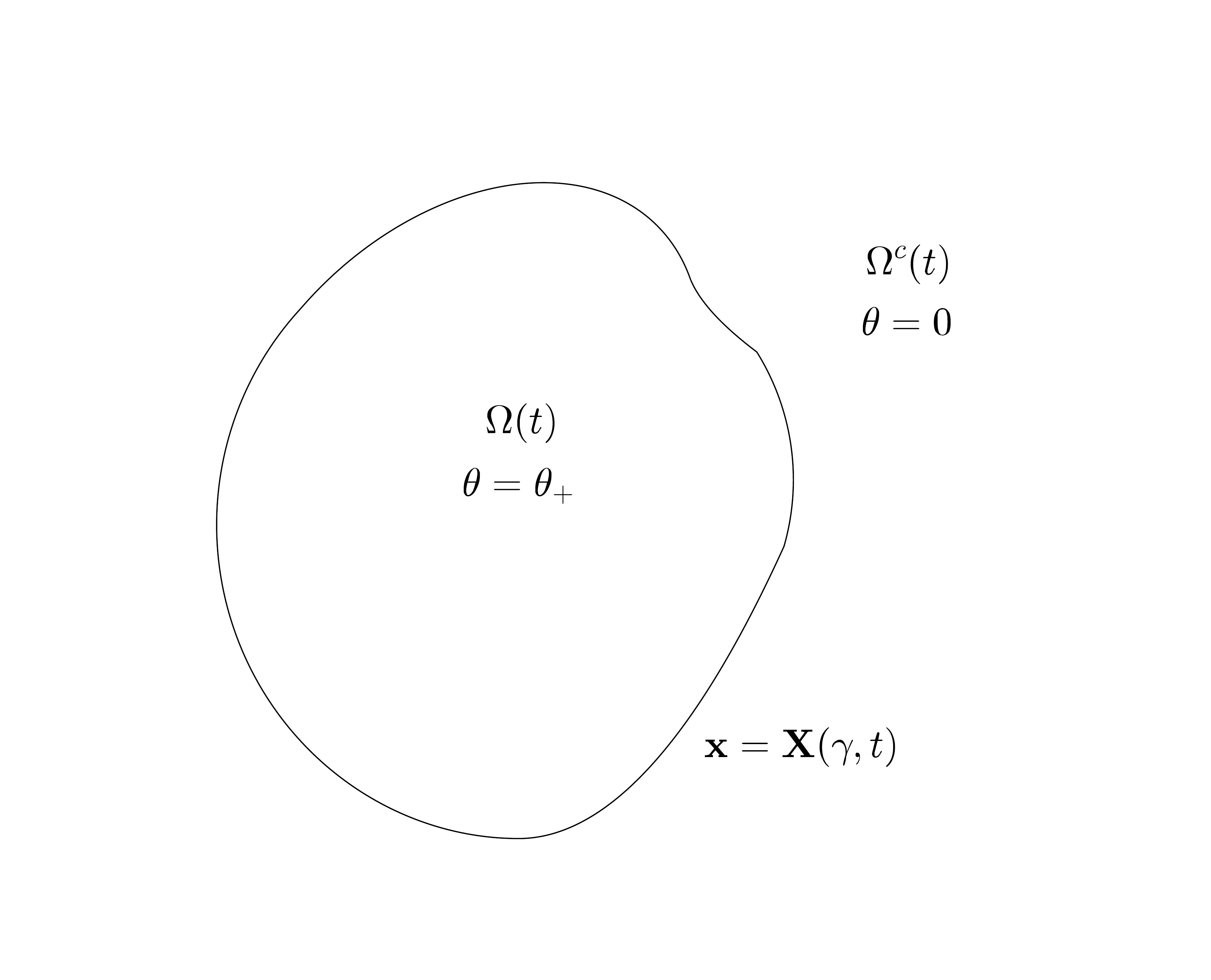}
\caption{Patch problem with $N = 2$.}
\end{subfigure}~
\begin{subfigure}[]{0.45\textwidth}
\includegraphics[width=\textwidth]{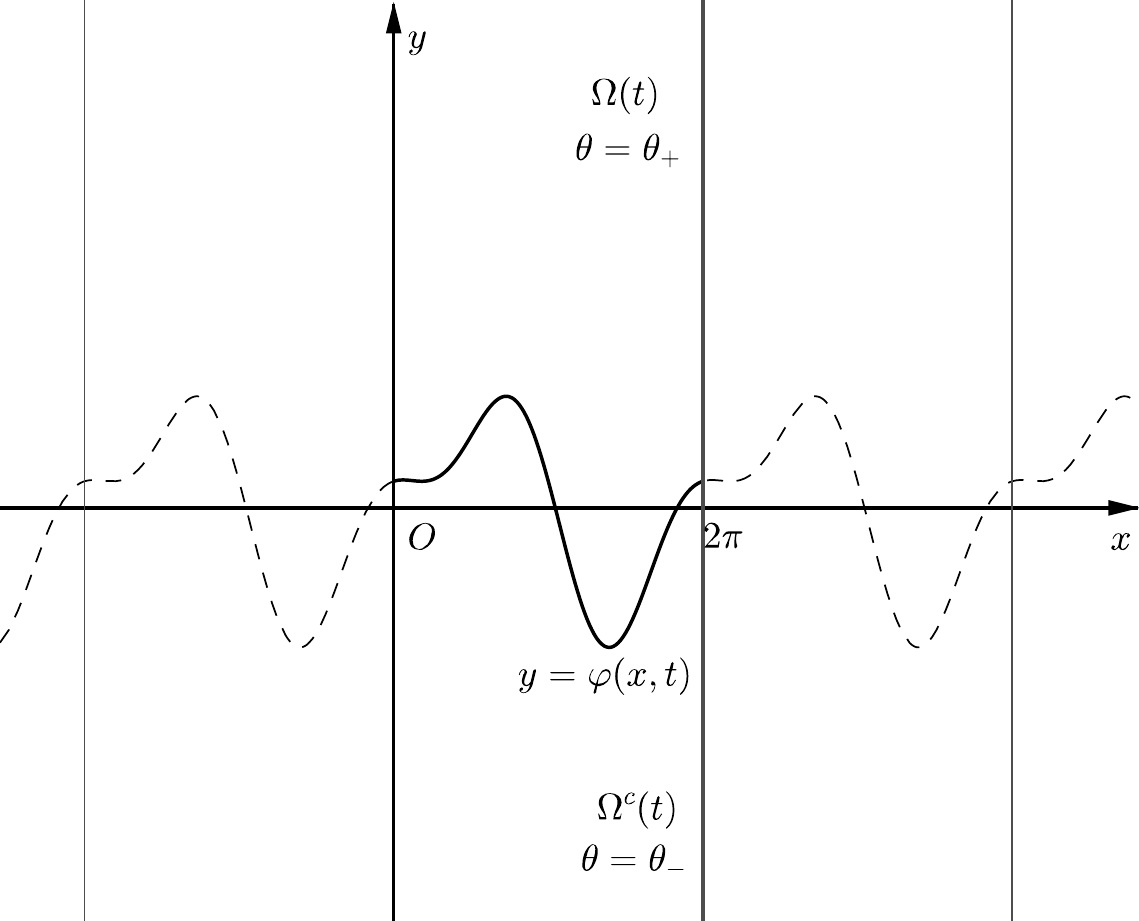}
\caption{Spatially periodic front problem.}
\end{subfigure}\\
\begin{subfigure}[]{0.45\textwidth}
\includegraphics[width=\textwidth]{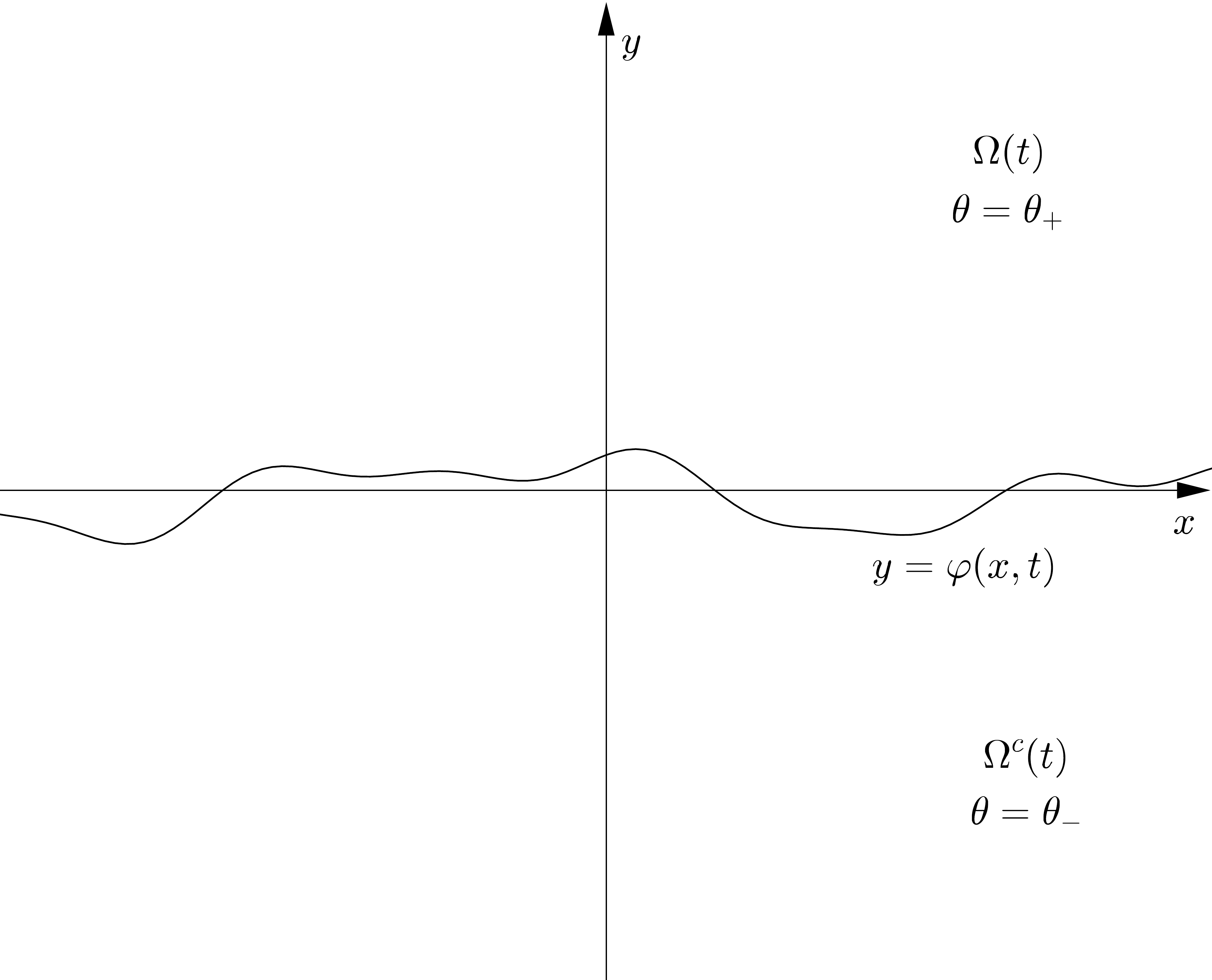}
\caption{Non-periodic front problem.}
\end{subfigure}~
\begin{subfigure}[]{0.45\textwidth}
\includegraphics[width=\textwidth]{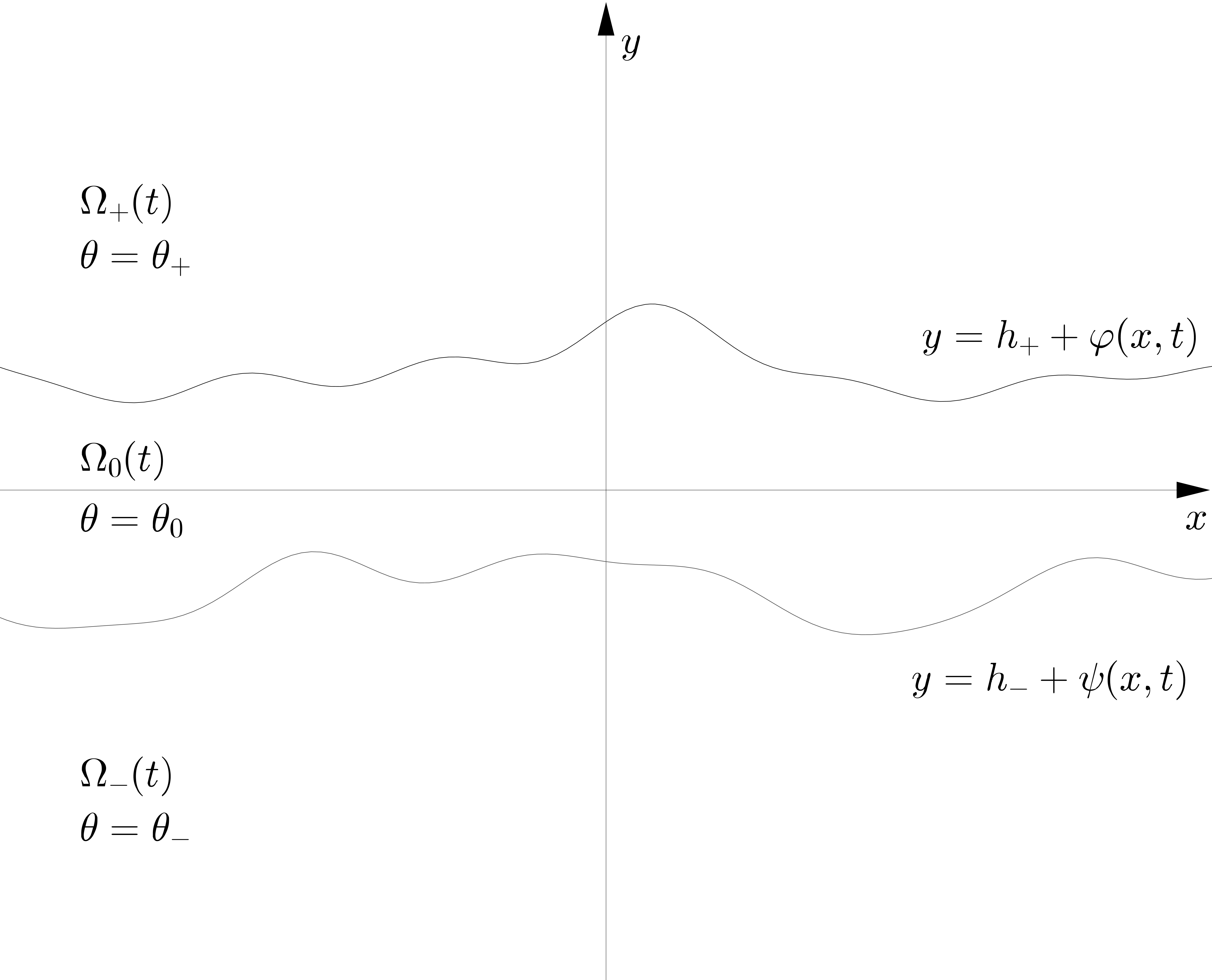}
\caption{Two-front problem.}
\end{subfigure}
\caption{Different types of patch and front problems.}\label{fig:patch+fronts}
\end{figure}

\subsubsection{Patches}
Equation \eqref{piecewisetheta} is a patch solution if it satisfies the following assumptions:
\begin{enumerate}
\item $N \geq 2$;
\item $\theta_N = 0$, but $\theta_k \in \R \setminus \{0\}$ for each $1 \leq k \leq N - 1$;
\item for each $1 \leq k \leq N - 1$, the region $\Omega_k(t)$ is bounded, and its boundary $\partial \Omega_k(t)$
is a smooth, simple, closed curve that is diffeomorphic to the circle $\T$;
\item the region $\Omega_N(t)$ is unbounded.
\end{enumerate}

Under these assumptions, $\theta$ has compact support and contour dynamics equations for the motion of the patches are straightforward to derive, as was first done by Zabusky et.~al.~\cite{Zabusky} for vortex-patch solutions of the Euler equation. The 2D Euler equation has global weak solutions with vorticity in $L^1(\R^2)\cap L^\infty(\R^2)$ \cite{ MB02,Yud63}, and smooth vortex patch boundaries remain smooth and non-self-intersecting for all times \cite{BC93, Che93, Che98}. Some special types of nontrivial global-in-time smooth vortex patch solutions are constructed in \cite{Bur82, CCG16b, dlHHMV16, HM16, HM17, HMV13}.

Local well-posedness of the contour dynamics equations for SQG and GSQG patches is proved in \cite{CCCGW12, CCG18, Gan08, GP18p}. The question of whether finite-time singularities can form in smooth boundaries of SQG or GSQG patches remains open, but it is proved in \cite{GS14} that splash singularities cannot form, and some particular classes of nontrivial global solutions for SQG and GSQG patches have been shown to exist \cite{CCG16a, dlHHH16, Gom18p, HH15, HM17}.

The local existence of smooth GSQG patches in the presence of a rigid boundary is shown in \cite{GP18p, KYZ17} for a range of $\alpha$, and the formation of finite-time singularities is proved for a range of $\alpha$ close to $2$. By contrast, vortex patches in this setting (with $\alpha=2$) have global regularity \cite{KRYZ16}.

Numerical solutions for vortex patches show that, although their boundaries remain smooth globally in time, they form extraordinarily thin, high-curvature filaments \cite{dritschel1, dritschel2}. On the other hand,
numerical solutions for SQG patches suggest that complex, self-similar singularities can form in the boundary of a single patch \cite{SD14} and provide evidence that two separated SQG patches can touch in finite time \cite{CFMR05}.

\subsubsection{Fronts}
Equation \eqref{piecewisetheta} is a front solution if it satisfies the following assumptions:
\begin{enumerate}
\item $N = 2$;
\item $\theta_1, \theta_2 \in \R$ are distinct constants;
\item both $\Omega_1(t)$ and $\Omega_2(t)$ are unbounded and they share a boundary which is a simple, smooth curve diffeomorphic to $\R$.
\end{enumerate}

When $1 \leq \alpha \leq 2$, the kernel of the (generalized) Biot-Savart law that recovers the velocity field $\u$ from the scalar $\theta$ decays too slowly at infinity for the standard potential representation of $\u$ to converge. This differentiates the patch problems and the front problems, since there are no convergence issues at infinity in the case of patches with compactly supported $\theta$. A procedure to derive regularized equations for a single front that is a graph located at $y = \vp(x,t)$ was introduced in \cite{HS18}. As shown in \cite{HSZ19} for the SQG equation, the regularized one-front equation also follows by decomposing the velocity field into an unbounded shear flow and a velocity perturbation that has a standard potential representation, and applying contour dynamics to the velocity perturbation.

The front problem for vorticity discontinuities in the Euler equation is studied in \cite{BH10, Ray1895}. Local existence and uniqueness for spatially periodic SQG fronts is proved for $C^\infty$ solutions in \cite{Rod05} and analytic solutions in \cite{FR11}, while local well-posedness in Sobolev spaces for spatially periodic solutions of a cubically nonlinear approximation of the SQG front equation is proved in \cite{HSZ18}. Almost sharp SQG fronts are studied in \cite{CFR04, FLR12, FR12, FR15}, and smooth $C^\infty$ solutions for spatially periodic GSQG fronts with $1 < \alpha < 2$ also exist locally in time \cite{CFMR05}.

In the non-periodic setting, smooth  solutions to the GSQG front equations with $0 < \alpha < 1$ on $\R$ are shown to exist globally in time for small initial data in \cite{CGI19},
and an analogous result for the  SQG front equation with $\alpha = 1$ is proved in \cite{HSZ18p}.

\subsubsection{Two-fronts}
\label{sec:twofront}
Equation \eqref{piecewisetheta} is a two-front solution if it satisfies the following assumptions:
\begin{enumerate}
\item $N = 3$;
\item $\theta_1, \theta_2, \theta_3 \in \R$ with $\theta_1 \neq \theta_2$ and $\theta_2 \neq \theta_3$;
\item there is a diffeomorphism $\Psi_t \colon \R^2 \to \R^2$, satisfying $\Psi_t\left(\Omega_1(t)\right) = \R \times (1, \infty)$, $\Psi_t\left(\Omega_2(t)\right) = \R \times (-1, 1)$, and $\Psi_t\left(\Omega_3(t)\right) = \R \times (-\infty, -1)$.
\end{enumerate}

This case is the one we study here. We derive equations for the locations of the two fronts and prove well-posedness results for the resulting systems.
From now on, we write $\Omega_+(t) = \Omega_1(t)$, $\Omega_0(t) = \Omega_2(t)$, $\Omega_-(t) = \Omega_3(t)$, with the same subscript changes applying to $\theta_+$, $\theta_0$, $\theta_-$. We also define the jumps in $\theta$ across the fronts, scaled by a convenient factor $g_\alpha$ given in \eqref{GreenF}, by
\begin{equation}
\Theta_+ = g_\alpha\left(\theta_+ - \theta_0\right),\qquad\Theta_- = g_\alpha\left(\theta_0 - \theta_-\right).
\label{defTheta}
\end{equation}
Numerical solutions of the contour dynamics equations for spatially-periodic two-front solutions of the Euler equation and a study of the approximation of vortex sheets by
a thin vortex layer are given in \cite{baker}.

\subsection{Main results}

We consider two-front solutions whose fronts are graphs located at
\[
y = h_+ + \vp(x, t), \qquad y = h_- + \psi(x, t),
\]
where $\vp, \psi \colon \R \times \R_+ \to \R$ denote the perturbations from the flat fronts
$y = h_+$, $y = h_-$, and $h_+ > h_-$. We also write
\begin{equation}
h = \frac{h_+ - h_-}{ 2}.
\label{defh}
\end{equation}

Since the fronts are graphs, they cannot self-intersect, but we also need to require that
the fronts do not intersect each other, which is the case if $\vp(\cdot, t), \psi(\cdot, t) \in L^2(\R)$ satisfy the pointwise condition
\begin{align}
\label{ptwise}
2 h - \psi(x, t) + \vp(x, t) > 0\ \text{for all}\ x \in \R.
\end{align}
This condition corresponds to the chord-arc condition for patches \cite{Gan08}.

As we will see, there are different features for $0 < \alpha < 1$, $\alpha =1$, and $1<\alpha \le 2$, which are a consequence of a loss of local integrability in the restriction of the Riesz potential \cite{Ste70, Ste93} for $(-\Delta)^{-\alpha/2}$ to the front for $0<\alpha \le 1$, leading to an infinite tangential velocity on the front, and a loss of global integrability for $1\le \alpha \le 2$, leading to an unbounded velocity far from the front. The nonlinear terms in the front equations also behave differently, losing derivatives if $0<\alpha \le 1$, and having good, hyperbolic-type energy estimates if $1<\alpha \le 2$  (see Table~\ref{tab:alpha}).

\begin{table}[!htbp]
\caption{Behavior of front solutions in different $\alpha$-regimes}
\begin{tabular}{|c|c|c|c|}

\hline

$\alpha$&Far-Field Velocity& Tangential Velocity&Derivative Loss\\

\hline

(0,1)&Bounded&Unbounded&Yes\\

\hline

1&Unbounded\footnotemark[1]&Unbounded&Yes\\

\hline

(1,2]&Unbounded\footnotemark[1]&Bounded&No\\

\hline

\end{tabular}
\label{tab:alpha}
\end{table}
\footnotetext[1]{The far-field velocity of the two-front solutions is bounded if $\Theta_+ = - \Theta_-$.}

The equations describing the dynamics of the fronts are given by \eqref{eulersys} for Euler, \eqref{sqgsys} for SQG,
and \eqref{gsqgsys12} for GSQG.
Symmetric (with $\Theta_+ = \Theta_-$) and anti-symmetric (with $\Theta_+ = -\Theta_-$) scalar reductions of these equations are given in \eqref{symm-gsqg} and \eqref{anti-symm-gsqg}, respectively.

\subsubsection{Local well-posedness}
We briefly summarize our local well-posedness results for the front equations. As explained further in Section~\ref{sec:prelim}, we use $T_b$ to denote a Weyl para-product with symbol $b$. In the following, we assume $\Theta_+$ and $\Theta_-$ are two nonzero numbers fixed beforehand,
and we denote by $H^s$ and $W^{k,p}$ the standard Sobolev spaces of functions with $s$ weak-$L^2$ and $k$ weak-$L^p$ derivatives, respectively.

Our results for $0<\alpha \le 1$ are restricted to small data.

\begin{theorem}[$\alpha\in(0,1)$]
\label{existence01}
Let $s \ge 5$ be an integer, and suppose that $\vp_0,\psi_0 \in H^s(\R)$ satisfy: (i)
\begin{align*}
& \left\|\vartheta - T_{B^{1 - \alpha}[\vp_0]}\right\|_{L^2 \to L^2} \ge m_0,\qquad \left\|\vartheta - T_{B^{1 - \alpha}[ \psi_0]}\right\|_{L^2 \to L^2} \ge m_0,\\
& \left\|T_{\beta[\vp_0]}\right\|_{L^2 \to L^2} \geq m_0', \qquad \left\|T_{\beta[\psi_0]}\right\|_{L^2 \to L^2} \geq m_0',
\end{align*}
for some constants $m_0, m_0' > 0$, where the constant $\vartheta$ is defined in \eqref{def_vartheta}, the symbol $\beta[f]$ is defined in \eqref{betadef01}, and the symbol $B^{1 - \alpha}[f]$ is defined in \eqref{defB01}; (ii)
\begin{align*}
\begin{split}
\sum_{n = 0}^\infty \sum_{\ell = 0}^n \sum_{m = 0}^{2n - \ell} \tilde{C}^n \left(1 + h^{\ell - 2n + \alpha - \frac{5}{2}}\right) \|\vp_0\|_{W^{4, \infty}}^{2n - \ell - m} \|\psi_0\|_{W^{4, \infty}}^m < \infty,
\end{split}
\end{align*}
where $h$ is defined in \eqref{defh}, and  $\tilde C > 1$ is the constant depending only on $s$ and $h$ in Proposition~\ref{apriori01}.
Then there exists $\Time > 0$, depending only on $\|\vp_0\|_{H^s}$, $\|\psi_0\|_{H^s}$,  $m_0$, $m_0'$, and $\tilde{C}$, such that the initial value problem for \eqref{gsqgsys12} with $0 < \alpha < 1$, $\vp(x,0) = \vp(x)$, $\psi(x,0) = \psi_0(x)$ has a unique solution with $\vp, \psi \in C([0, \Time); H^s(\R))$.
\end{theorem}
\begin{theorem}[$\alpha=1$]
\label{existence1}
Let $s \ge 5$ be an integer, and suppose that $\vp_0,\psi_0 \in H^s(\R)$ satisfy: (i)
\[
\|T_{B^{\log}[\vp_0]}\|_{L^2 \to L^2} \leq C,\quad \|T_{B^{\log}[\psi_0]}\|_{L^2 \to L^2} \leq C
\]
for some constant $0 < C < 2$, where  the symbol $B^{\log}[f]$ is defined in \eqref{defBsqg}; (ii)
\begin{align*}
\sum_{n = 1}^\infty \tilde{C}^n |c_n| \Big(\|\varphi_0\|_{W^{4, \infty}}^{2n} + \|L \varphi_0\|_{W^{4, \infty}}^{2n}\Big) < \infty,\qquad
\sum_{n = 1}^\infty \tilde{C}^n |c_n| \Big(\|\psi_0\|_{W^{4, \infty}}^{2n} + \|L \psi_0\|_{W^{4, \infty}}^{2n}\Big) < \infty,
\end{align*}
where $L=\log|\partial_x|$ is the Fourier multiplier with symbol $\log|\xi|$, $c_n$ is given by \eqref{alpha-taylor}, and  $\tilde C > 1$ is the constant depending only on $s$ and $h$ in Proposition~\ref{apriori}.
Then there exists $T > 0$, depending only on $\|\vp_0\|_{H^s}$, $\|\psi_0\|_{H^s}$,  $C$, and $\tilde{C}$, such that the initial value problem  for \eqref{sqgsys} with $\vp(x,0) = \vp(x)$, $\psi(x,0) = \psi_0(x)$ has a unique solution with $\vp, \psi \in C([0, \Time); H^s(\R))$.
\end{theorem}

\begin{remark}
The smallness conditions in Theorems~\ref{existence01}--\ref{existence1} arise from the fact that the nonlinear terms in the front equations lose derivatives, and we use a multilinear expansion of the nonlinearity to extract the terms responsible for the loss of derivatives. This expansion can only be done when the solutions are sufficiently small and requires Condition (ii). We then use the linear terms to control these nonlinear terms in a weighted energy space, but our weight may degenerate if Condition (i) fails.

Condition (ii) also implies that the initial data satisfies the non-intersection condition \eqref{ptwise}, since it  guarantees that $|\psi_0(x) - \vp_0(x)| < 2 h$ for all $x \in \R$.
\end{remark}

The case $1< \alpha \le 2$ is simpler than $0< \alpha \le 1$, since the nonlinear terms do not lose derivatives, and we have the following large data result.

\begin{theorem}[\mbox{$\alpha\in(1,2]$}]
\label{existence12}
Let $s \ge 3$ be an integer, and suppose that $\vp_0,\psi_0 \in H^s(\R)$  satisfy the non-intersection condition \eqref{ptwise}.
Then there exists $\Time>0$, depending only on $\|\vp_0\|_{H^s}$, $\|\psi_0\|_{H^s}$, and $\|2 h - \psi_0 + \vp_0\|_{L^\infty(\R)}$, such that the initial value problem for the system \eqref{gsqgsys12} with $1 < \alpha \leq 2$, $\vp(x,0) = \vp(x)$, $\psi(x,0) = \psi_0(x)$ has a unique solution with $\vp, \psi \in C([0, \Time); H^s(\R))$.
\end{theorem}


Theorems \ref{existence01}--\ref{existence12} follows from \emph{a priori} estimates and standard local existence theory for quasilinear equations (see e.g. \cite{Kat75}). Therefore, in the following we only derive the \emph{a priori} estimates for the corresponding Cauchy problems. The same results also hold backwards in time.

\subsection{Outline of the paper}
In Section~\ref{sec:prelim}, we provide the definitions and  some properties of fractional Laplacians, the Weyl paradifferential calculus, and modified Bessel functions of the second kind. In Section~\ref{sec:reg}, we derive the two-front equations. In Section~\ref{sec:lin_stab}, we analyze the linearized stability of the unperturbed, flat two-front solutions, which is a particular example of a GSQG shear flow, and in Sections~\ref{sec:lwp01}--\ref{sec:lwp12}, we prove the \emph{a priori} estimates for the front equations. Section~\ref{sec:lwp01} treats the GSQG equation with $0<\alpha<1$, Section~\ref{sec:lwp02} treats the SQG equation, and Section~\ref{sec:lwp12} treats the GSQG equation with $1<\alpha\le 2$.

\subsection{Acknowledgement}
JS would like to thank Javier G\'{o}mez-Serrano for discussions in the ``MathFluids'' Workshop held in Mathematical Institute of University of Seville, Seville, Spain, June 12--15, 2018.

\section{Preliminaries}\label{sec:prelim}

\subsection{Fractional Laplacians}

We interpret the fractional Laplacian $(-\Delta)^{\alpha/2}$ in \eqref{gsqg} in a distributional sense, and we summarize its definition here.

Let $0<\alpha < 2$. We denote by $L^1_\alpha(\R^n)$ the space of measurable functions $f: \R^n \to \R$ such that
\[
\int_{\R^n} \frac{|f(\x)|}{1 + |\x|^{n+\alpha}}\, \diff{\x} < \infty.
\]
Then $(-\Delta)^{\alpha/2} : L^1_\alpha(\R^n) \to \mathcal{D}'(\R^n)$ can be defined by \cite{BB99}
\[
\left\langle (-\Delta)^{\alpha/2} f ,\phi\right\rangle = \int_{\R^n} f(\x)\cdot (-\Delta)^{\alpha/2}\phi(\x)\, \diff{\x}
\qquad \text{for all $\phi\in C_c^\infty(\R^n)$},
\]
where $(-\Delta)^{\alpha/2}$ acts on test functions $\phi$
as, for example, a Fourier multiplier or a singular integral \cite{KW17}.
For every compact set $K \subset \R^n$, there exists a constant $C(K,\alpha)$ such that
\[
\sup_{\x\in \R^n}\left|\left(1 + |\x|^{n+\alpha}\right) (-\Delta)^{\alpha/2}\phi(\x)\right|\le C(K,\alpha) \|\phi\|_{C^2(\R^n)}
\qquad \text{for all $\phi\in C_c^\infty(\R^n)$ with $\supp \phi \subset K$},
\]
so $(-\Delta)^{\alpha/2} f$ is a distribution of order at most $2$ for $f\in L^1_\alpha(\R^n)$.

As can be seen for the shear-flow solutions \eqref{shear_flow}, the front velocity-fields $\u$ belong to $L^1_\alpha(\R^2)$ for $0<\alpha < 2$, so $(-\Delta)^{\alpha/2}\u$ in \eqref{gsqg} is well-defined as a distribution. Moreover, the only $\alpha$-harmonic solutions
$f\in L^1_\alpha(\R^n)$ of $(-\Delta)^{\alpha/2}f = 0$ are constant functions for $0<\alpha \le 1$ or affine functions for $1<\alpha <2$ \cite{CDL15,Fall16}. Thus,
if we require that $\u$ has sublinear growth in $\x$, then $\u$ is determined from $\theta$ up to a spatially uniform constant (which may depend upon $t$), and velocity fields that differ by $\mathbf{C}(t)$ give equivalent dynamics by transforming into a reference frame
moving with velocity $\mathbf{C}(t)$.

\subsection{Para-differential calculus}

In this section, we recall the definition of Weyl para-products and state two lemmas.
Further discussion of the Weyl calculus and para-products can be found in \cite{BCD11, Hor85, Tay00}.

We denote the Fourier transform of $f \colon \R\to \C$ by $\hat f \colon \R\to \C$, where $\hat f= \F f$ is given by
\[
f(x)=\int_{\R} \hat f(\xi) e^{i\xi x} \diff\xi,  \qquad \hat f(\xi)=\frac1{2\pi} \int_{\R}f(x) e^{-i\xi x}\diff{x}.
\]

For $s\in \R$, we denote by $H^s(\R)$ the space of Schwartz distributions $f$ with $\|f\|_{H^s} < \infty$, where
\[
\|f\|_{H^s} = \left[\int_\R (1+\xi^2)^s |\hat{f}(\xi)|^2\, d\xi\right]^{1/2}.
\]

Throughout this paper, we use $A\lesssim B$ to mean there is a constant $C$ such that $A\leq C B$, and $A\gtrsim B$ to mean there is a constant $C$ such that $A\geq C B$. We use $A\approx B$ to mean that $A\lesssim B$ and $B\lesssim A$.
The notation $\O(f)$ denotes a term satisfying $\|\O(f)\|_{{H}^s}\lesssim \|f\|_{{H}^s}$
whenever there exists  $s\in \R$ such that $f\in {H}^s$, and $O(f)$ denote a term satisfying $|O(f)|\lesssim|f|$ pointwise.

Let $\chi \colon \R \to \R$ be a smooth function supported in the interval $\{\xi\in \R \mid |\xi|\leq 1/10\}$
and equal to $1$  on $\{\xi\in \R \mid |\xi|\leq 3/40\}$.
If
$a \colon \R \times \R \to \C$ is a symbol, then
we define the Weyl para-product operator $T_a$ by
\begin{equation}
\label{weyldef}
\F \left[T_a f\right](\xi)=\int_{\R} \chi\left(\frac{|\xi-\eta|}{|\xi+\eta|}\right) \tilde{a}\Big(\xi-\eta, \frac{\xi + \eta}{2}\Big)\hat f(\eta)\diff\eta,
\end{equation}
where
\begin{align}
\label{atilde}
\tilde{a}(\xi,\eta) = \frac{1}{2 \pi} \int_\R a(x, \eta) e^{- i \xi x} \diff{x}
\end{align}
denotes the partial Fourier transform of $a(x, \eta) $ with respect to $x$. For $r_1, r_2 \in \N_0$, we define a normed symbol space by
\begin{align*}
\M_{(r_1, r_2)} &= \{a \colon \R \times \R \to \C : \|a\|_{\M_{(r_1, r_2)}} < \infty\},
\\
\|a\|_{\M_{(r_1, r_2)}} &= \sup_{(x, \eta) \in \R^2} \left\{\sum_{\alpha=0}^ {r_1} \sum_{\beta=0}^{r_2} (1 + |\eta|^\beta)  \big|\partial_\eta^\beta \px^\alpha a(x, \eta)\big|\right\}.
\end{align*}
If $a \in \M_{(1, 1)}$ and $f \in {L}^2$, then $T_a f \in L^2$ and \cite{Bou99}
\[
\|T_a f\|_{L^2} \lesssim \|a\|_{\M_{(1, 1)}} \|f\|_{L^2}.
\]
In particular, if $a\in\M_{(1, 1)}$ is real-valued, then $T_a$ is a self-adjoint, bounded linear operator on $L^2$.

Next, we state a lemma on composition for Weyl para-products proved in \cite{HSZ18p} (see also \cite{CGI19, DIPP17}).
\begin{lemma}
\label{weyl-comm}
If $a, b \in \M_{(3, 5)}$ and $f \in H^s(\R)$, then
\[
T_a T_b f = T_{a b} f + \frac{1}{2 i} T_{\{a, b\}} f + \Rf',
\]
where $\{a, b\} = \partial_\eta a\cdot \px b - \partial_\eta b\cdot \px a$ is the Poisson bracket of $a$ and $b$, and
the remainder term $\Rf'$ satisfies the estimate
\begin{align}
\label{comm-remainer}
\|\Rf'\|_{H^{s + 2}} \lesssim \|a\|_{\M_{(3, 5)}} \|b\|_{\M_{(3, 5)}} \|f\|_{H^s}.
\end{align}
As a consequence,
\[
[T_a, T_b] f = - i T_{\{a, b\}} f + \Rf,
\]
where $\Rf$ also satisfies \eqref{comm-remainer}.
\end{lemma}

Finally, we state an expansion for the action the Fourier multiplier $|D|^s$
with symbol $|\xi|^s$ on para-products, whose proof can be found in \cite{HSZ18p}.

\begin{lemma}\label{lem-DsT}
If $a \in \M_{(3, 1)}$ and $f\in H^s(\R)$, then
\[
|D|^sT_a f=T_a |D|^sf+sT_{Da}|D|^{s-2}Df+\O(T_{|D|^2a}|D|^{s-2}f),
\]
where $Da$ means that the differential operator $D$ acts on the function $x\mapsto a(x, \xi)$ for fixed $\xi$, and similarly
for  $|D|^2a$.
\end{lemma}

\subsection{Modified Bessel function of the second kind}

In this section, we summarize some definitions and properties of modified Bessel functions, which can be found in \cite{OLBC, Wat95}.
The modified Bessel function $I_\nu$ of the first kind is defined for $\nu \in \R$ by
\[
I_\nu(x) = \sum_{m = 0}^\infty \frac{1}{m! \Gamma(m + \nu + 1)} \bigg(\frac{x}{2}\bigg)^{2m + \nu}.
\]
The modified Bessel function $K_\nu$ of the second kind is defined for $\nu \notin \Z$ by
\[
K_\nu(x) = \frac{\pi}{2} \frac{I_{-\nu}(x) - I_\nu(x)}{\sin{\nu \pi}},
\]
and $K_n(x) = \lim_{\nu \to n} K_\nu(x)$ for $n \in \Z$.
When $\nu > - 1/2$ and $x > 0$, we can also write $K_\nu$  as
\begin{equation}
\label{BesselK}
K_\nu(x) = \frac{\Gamma(\nu + \frac 12) (2x)^\nu}{\sqrt{\pi}} \int_0^\infty \frac{\cos y}{(y^2 + x^2)^{\nu + 1/2}} \diff{y}.
\end{equation}
In \eqref{BesselK}, and throughout this paper, $\Gamma(z)$ denotes the Gamma function.

The following lemma collects the properties of modified Bessel functions of the second kind that we need. Properties (i)--(iv) can be found in \cite{OLBC}.

\begin{lemma}
\label{BesselProp}
The modified Bessel functions of the second kind have following properties:
\begin{enumerate}[(i)]
\item For each $\nu \geq 0$, $K_\nu(x)$ is a real-valued, analytic, strictly decreasing function on $(0, \infty)$.

\item For each fixed $x, \nu > 0$, $K_\nu(x) = K_{-\nu}(x)$.

\item If $\nu > 0$, then
\begin{align*}
K_\nu(x) \sim \frac{1}{2} \Gamma(\nu) \left(\frac{x}{2}\right)^{- \nu}, \quad K_0(x) \sim - \log(x)
 \qquad \text{as $x \to 0^+$}.
\end{align*}

\item If $\nu \geq 0$, then
\begin{align*}
K_\nu(x) \sim \sqrt{\frac{\pi}{2x}} e^{-x}\qquad \text{as $x \to \infty$}.
\end{align*}

\item Let $m \geq 0$ be an integer, and define $f_m \colon \R \times (\frac{1}{2}, \infty) \to \R$ by
\[
f_m(x, \nu) = |x|^{\nu + m} K_\nu(|x|).
\]
Then $f_m(\cdot,\nu)$ attains its maximum, and if the maximum is attained at some $x_0\in \R$, then
\begin{align}
\label{fm-max}
|x_0| \leq \sqrt{m^2 + (2 \nu - 1) m}, \qquad 0 \leq f_m(x_0, \nu) \leq \frac{\left(m^2 + (2 \nu - 1) m\right)^{m / 2} \Gamma(\nu)}{2^{\nu + 1}}.
\end{align}
\end{enumerate}
\end{lemma}
\begin{proof}[Proof of (v)]
It follows from (ii) that we only need to consider $x \geq 0$.
We use the identities (see \cite{OLBC})
\begin{align*}
& K_\nu'(x) = - \frac{K_{\nu - 1}(x) + K_{\nu + 1}(x)}{2},\\
& x K_{\nu + 1}(x) - x K_{\nu - 1}(x) = 2 \nu K_\nu(x),
\end{align*}
to obtain
\begin{align*}
\frac{\partial}{\partial{x}} f_m(x, \nu) & = (\nu + m) x^{\nu + m - 1} K_\nu(x) - \frac{1}{2} x^{\nu + m} \left(K_{\nu - 1}(x) + K_{\nu + 1}(x)\right)\\
& = x^{\nu + m - 1} \left(m K_\nu(x) - x K_{\nu - 1}(x)\right).
\end{align*}

When $m = 0$ and $\nu > 1 / 2$, we have $\partial_xf_0(x,\nu)\leq 0$. Thus $f_0$ is decreasing in $x$, and its maximum is attained at $x_0 = 0$ with
\[
f_0(0, \nu) = \frac{\Gamma(\nu)}{2^{\nu + 1}}.
\]

When $m > 0$, it is clear that $f_m$ is smooth in $x \in (0, \infty)$ with
\[
f_m(0, \nu) = \lim_{x \to \infty} f_m(x, \nu) = 0,
\]
so the maximum is attained at its critical points.
Therefore, $x_0$ must satisfy
\[
\frac{m}{x_0} = \frac{K_{\nu - 1}(x_0)}{K_\nu(x_0)}.
\]

For $\nu > \frac{1}{2}$ and $x > 0$, we have that \cite{Seg11}
\[
\frac{K_{\nu - 1}(x)}{K_\nu(x)} > \frac{x}{\sqrt{x^2 + (\nu - 1 / 2)^2} + \nu - 1 / 2},
\]
which leads to the estimate of $|x_0|$ in \eqref{fm-max}.
Then, using
\[
f_m(x_0,\nu)=|x_0|^{\nu+m}K_\nu(x_0)
=|x_0|^{m}f_0(x_0,\nu)\leq |x_0|^{m}f_0(0,\nu)=|x_0|^{m}\frac{\Gamma(\nu)}{2^{\nu+1}},
\]
we obtain the upper bound for $f_m$.
\end{proof}

\section{Two-front GSQG systems}\label{sec:reg}

\subsection{Contour dynamics}
In this section, we derive contour dynamics equations for two-front solutions of the Euler, SQG, and GSQG equations. For $1\le \alpha\le 2$,
the formal contour dynamics equations diverge at infinity, and we use the regularization procedure developed in \cite{HS18} to obtain convergent front equations. Equivalent results could be obtained by decomposing the velocity field into a two-front shear flow of the type discussed in Section~\ref{sec:lin_stab} and a velocity perturbation due to the motion of the fronts, as is done in \cite{HSZ19} for one-front SQG solutions, but we find it more convenient to compute the front equations by use of the regularization procedure.

Using the GSQG equation \eqref{gsqg} and Green's theorem, we find that the velocity field of the two front solution illustrated in Figure~\ref{fig:patch+fronts}(d) is given formally by
\begin{align}
\label{cde}
\begin{split}
\u(\x, t) &= \nabla^\perp G * \theta(\x, t)\\
&= \Theta_+ \int_{\partial \Omega_+(t)} G\left(|\x - \x'|\right) \n^\perp(\x', t) \diff{s_+(\x')} + \Theta_- \int_{\partial \Omega_-(t)} G\left(|\x - \x'|\right) \n^\perp(\x', t) \diff{s_-(\x')},
\end{split}
\end{align}
where the jumps $\Theta_\pm$ are defined in \eqref{defTheta}, the Green's function for the operator $(-\Delta)^{\alpha / 2}$ on $\R^2$ is given by $g_\alpha G(|\x|)$ with
\begin{align}
\label{GreenF}
G(x) = \left\{\begin{array}{ll}- \frac{1}{2 \pi} \log|x| & \quad \text{if}\ \alpha = 2,\\[1.5ex] |x|^{-(2 - \alpha)} & \quad \text{if}\ 0 < \alpha < 2,\end{array}\right. \qquad g_\alpha = \left\{\begin{array}{ll}1 & \quad \text{if}\ \alpha = 2,\\[1ex] \frac{\Gamma(1 - \alpha / 2)}{2^\alpha \pi \Gamma(\alpha / 2)} & \quad \text{if}\ 0 < \alpha < 2,\end{array}\right.
\end{align}
$\n = (m, n)$ is the upward unit normal to $\partial \Omega_\pm(t)$, $\n^\perp = (-n, m)$, and $s_\pm(\x')$ is arc-length on $\partial \Omega_\pm(t)$.

The integrals in \eqref{cde} converge at infinity when $0 < \alpha < 1$, but diverge when $1 \leq \alpha \leq 2$. To obtain the front equations, we first cut-off the integration region to a $\lambda$-interval about some point $x\in \R$ and consider the limit $\lambda \to \infty$. If  $1 \leq \alpha \leq 2$ and $\Theta_+ +\Theta_- \ne 0$, we also make a  Galilean transformation $x\mapsto x- v(\lambda)t$, where $v(\lambda)$ is chosen to give well-defined limiting front equations and $|v(\lambda)|\to \infty$ as $\lambda \to \infty$ \cite{HS18}. We assume that the top and bottom fronts are smooth, approach $y = h_+$ and $y = h_-$ sufficiently rapidly as $|s_+(\x')| \to \infty$ and  $|s_-(\x')| \to \infty$, respectively, and do not self-intersect or intersect each other.

 Let the top and bottom fronts have parametric equations $\x = \X_1(\zeta, t)$ and $\x = \X_2(\zeta, t)$, where
 \[
 \X_1(\cdot, t), \X_2(\cdot, t) \colon \R \to \R^2.
 \]
 Since $\theta$ is transported by the velocity field, the fronts move with normal velocity
\[\begin{aligned}
\pt \X_1 \cdot \n = \u \cdot \n,\qquad
\pt \X_2 \cdot \n = \u \cdot \n,
\end{aligned}\]
so the cut-off equations for $\X_1$ and $\X_2$ are
\[\begin{aligned}
\pt \X_1(\zeta, t) &= c_1(\zeta, t) \partial_\zeta \X_1(\zeta, t) - \Theta_+ \int_{\zeta -\lambda}^{\zeta + \lambda} G\left(|\X_1(\zeta', t) - \X_1(\zeta, t)|\right) \partial_{\zeta'} \X_1(\zeta', t) \diff{\zeta'}\\
&\hspace{2in} - \Theta_- \int_{\zeta - \lambda}^{\zeta + \lambda} G\left(|\X_2(\zeta', t) - \X_1(\zeta, t)|\right) \partial_{\zeta'} \X_2(\zeta', t) \diff{\zeta'},\\
\pt \X_2(\zeta, t) &= c_2(\zeta, t) \partial_\zeta \X_2(\zeta, t) - \Theta_+ \int_{\zeta - \lambda}^{\zeta + \lambda} G\left(|\X_1(\zeta', t) - \X_2(\zeta, t)|\right) \partial_{\zeta'} \X_1(\zeta', t) \diff{\zeta'}\\
&\hspace{2in} - \Theta_- \int_{\zeta - \lambda}^{\zeta + \lambda} G\left(|\X_2(\zeta', t) - \X_2(\zeta, t)|\right) \partial_{\zeta'} \X_2(\zeta', t) \diff{\zeta'},
\end{aligned}\]
where $c_1(\zeta, t)$ and $c_2(\zeta, t)$ are arbitrary functions corresponding to time-dependent reparametrizations of the fronts.

If the fronts are given by graphs that are perturbations of $y = h_+$ and $y = h_-$, then the top front is located at $y = h_+ + \varphi(x, t)$ and the bottom front at $y = h_- + \psi(x, t)$,
and we can solve for $c_1$ and $c_2$ to get
\[\begin{aligned}
c_1(x, t) &= \Theta_+ \int_{-\lambda}^\lambda G\left(\sqrt{\zeta^2 + \big(\varphi(x + \zeta, t) - \varphi(x, t)\big)^2}\right) \diff{\zeta}
\\
&+ \Theta_- \int_{-\lambda}^\lambda G\left(\sqrt{\zeta^2 + \big(-2h + \psi(x + \zeta, t) - \varphi(x, t)\big)^2}\right) \diff{\zeta},\\
c_2(x, t) &= \Theta_+ \int_{-\lambda}^\lambda G\left(\sqrt{\zeta^2 + \big(2h + \varphi(x + \zeta, t) - \psi(x, t)\big)^2}\right) \diff{\zeta}
\\
&+ \Theta_- \int_{-\lambda}^\lambda G\left(\sqrt{\zeta^2 + \big(\psi(x + \zeta, t) - \psi(x, t)\big)^2}\right) \diff{\zeta}.
\end{aligned}\]
We then obtain a coupled system for $\varphi$ and $\psi$
\begin{equation}
\label{cutoffgsqg}
\begin{aligned}
\varphi_t(x, t) &+ \Theta_+ \int_{-\lambda}^\lambda \left[\varphi_x(x + \zeta, t) - \varphi_x(x, t)\right] G\left(\sqrt{\zeta^2 + \big(\varphi(x + \zeta, t) - \varphi(x, t)\big)^2}\right) \diff{\zeta}\\
&\quad + \Theta_- \int_{-\lambda}^\lambda \left[\psi_x(x + \zeta, t) - \varphi_x(x, t)\right] G\left(\sqrt{\zeta^2 + \big(-2h + \psi(x + \zeta, t) - \varphi(x, t)\big)^2}\right) \diff{\zeta} = 0,\\[1ex]
\psi_t(x, t) &+ \Theta_+ \int_{-\lambda}^\lambda \left[\varphi_x(x + \zeta, t) - \psi_x(x, t)\right] G\left(\sqrt{\zeta^2 + \big(2h + \varphi(x + \zeta, t) - \psi(x, t)\big)^2}\right) \diff{\zeta}\\
&\quad + \Theta_- \int_{-\lambda}^\lambda \left[\psi_x(x + \zeta, t) - \psi_x(x, t)\right] G\left(\sqrt{\zeta^2 + \big(\psi(x + \zeta, t) - \psi(x, t)\big)^2}\right) \diff{\zeta} = 0.
\end{aligned}
\end{equation}

\subsection{Cut-off regularization}

As is in \cite{HS18}, we consider separately the cases $0 < \alpha < 1$, $\alpha = 1$, and $1 < \alpha \leq 2$, since the Green's functions in \eqref{GreenF} have different rates of growth or decay as $x \to 0$ and $|x| \to \infty$. We rewrite the system \eqref{cutoffgsqg} as
\begin{align}
\label{lambda-K-G}
\begin{split}
\varphi_t(x, t) &+ \Theta_+ \px \int_{-\lambda}^\lambda H_1\big(\zeta, \varphi(x + \zeta, t) - \varphi(x, t)\big) \diff{\zeta} + \Theta_+ \px \int_{-\lambda}^\lambda G(\zeta) \big[\varphi(x + \zeta, t) - \varphi(x, t)\big] \diff{\zeta}\\
&+ \Theta_- \px \int_{-\lambda}^\lambda H_2\big(\zeta, -2h + \psi(x + \zeta, t) - \varphi(x, t)\big) \diff{\zeta}\\
& \hspace{1.5in} + \Theta_- \px \int_{-\lambda}^\lambda G\left(\sqrt{\zeta^2 + (2h)^2}\right) \big[\psi(x + \zeta, t) - \varphi(x, t)\big] \diff{\zeta} = 0,\\
\psi_t(x, t) &+ \Theta_- \px \int_{-\lambda}^\lambda H_1\big(\zeta, \psi(x + \zeta, t) - \psi(x, t)\big) \diff{\zeta} + \Theta_- \px \int_{-\lambda}^\lambda G(\zeta) \big[\psi(x + \zeta, t) - \psi(x, t)\big] \diff{\zeta}\\
&+ \Theta_+ \px \int_{-\lambda}^\lambda H_2\big(\zeta, 2h + \varphi(x + \zeta, t) - \psi(x, t)\big) \diff{\zeta}\\
& \hspace{1.5in} + \Theta_+ \px \int_{-\lambda}^{\lambda} G\left(\sqrt{\zeta^2 + (2h)^2}\right) \big[\varphi(x + \zeta, t) - \psi(x, t)\big] \diff{\zeta} = 0
\end{split}
\end{align}
where
\begin{align}
\label{H12}
\begin{split}
H_1(x, y) &= - G(x) y + \int_0^y G\Big(\sqrt{x^2 + s^2}\Big) \diff{s},\\
H_2(x, y) &= - G\left(\sqrt{x^2 + (2h)^2}\right) y + \int_0^y G\Big(\sqrt{x^2 + s^2}\Big) \diff{s}.
\end{split}
\end{align}
When $G$ is given by \eqref{GreenF} we have for $j = 1, 2$ and fixed $y$ that
\[
H_j(x, y) = O\bigg(\frac{1}{|x|^{4 - \alpha}}\bigg) \quad \text{as}\ |x| \to \infty.
\]
It follows that the nonlinear terms in \eqref{lambda-K-G} converge as $\lambda \to \infty$, so it suffices to consider linear terms in \eqref{lambda-K-G}.

We only write out the computation for the first equation; the computation for the second equation is similar. The linear term
\[
\L_{1, \lambda}\varphi(x, t) := \int_{-\lambda}^\lambda G(\zeta) \left[\vp(x + \zeta, t) - \varphi(x, t)\right] \diff{\zeta}
\]
can be written as \cite{HS18}
\[
\L_{1, \lambda} \vp(x, t) = v_1(\lambda) \vp(x, t) + \L_{1, \lambda}^* \vp(x, t),
\]
where
\begin{equation}
\label{defv1}
v_1(\lambda) =
\begin{dcases}
~ 0 & \qquad \text{if}\ 0 < \alpha < 1,\\ - 2 \int_1^\lambda G(\zeta) \diff{\zeta} & \qquad \text{if}\ \alpha = 1,\\ - 2 \int_0^\lambda G(\zeta) \diff{\zeta} & \qquad \text{if}\ 1 < \alpha \leq 2,
\end{dcases}
\end{equation}
and $\L_{1, \lambda}^* \vp \to {\L}_1 \vp$ as $\lambda \to \infty$, where ${\L}_1$ is the Fourier multiplier with symbol
\begin{align}
\label{L1Symb}
\begin{split}
b_1(\xi) &=
\begin{dcases}
- \int_\R G(\zeta) \left(1 - e^{i \xi \zeta}\right) \diff{\zeta} & \qquad \text{if}\ 0 < \alpha < 1,\\
\int_{|\zeta| > 1} G(\zeta) e^{i \xi \zeta} \diff{\zeta} - \int_{|\zeta| < 1} G(\zeta) \left(1 - e^{i \xi \zeta}\right) \diff{\zeta} & \qquad \text{if}\ \alpha = 1,\\
\int_\R G(\zeta) e^{i \xi \zeta} \diff{\zeta} & \qquad \text{if}\ 1 < \alpha \leq 2,
\end{dcases}\\
&=
\begin{dcases}
2 \sin\bigg(\frac{\pi \alpha}{2}\bigg) \Gamma(\alpha - 1) |\xi|^{1 - \alpha} & \qquad \text{if}\ \alpha \in (0, 2) \setminus \{1\},\\
- 2 \gamma - 2 \log|\xi| & \qquad \text{if}\ \alpha = 1,\\
\gamma \delta(\xi) + \frac{1}{2} \fp \frac{1}{|\xi|} & \qquad \text{if}\ \alpha = 2.
\end{dcases}
\end{split}
\end{align}
Here, $\gamma$ is the Euler--Mascheroni constant \cite{Vla71}.

As for the second linear term, we have
\begin{align*}
\L_{2, \lambda} \left[\vp, \psi\right](x, t) &:= \int_{-\lambda}^\lambda G\left(\sqrt{\zeta^2 + (2h)^2}\right) \big[\psi(x + \zeta, t) - \varphi(x, t)\big] \diff{\zeta}\\
&= v_2(\lambda) \vp(x, t) + v_3(\lambda) \vp(x, t) + \L_{2, \lambda}^* \psi(x, t),
\end{align*}
where $v_2(\lambda)$ is a divergent part (or zero if $\L_{2, \lambda}$ converges) and $v_3(\lambda)$ is a convergent part
\begin{align}
\label{defv2}
v_2(\lambda) &=
\begin{dcases}
0 & \quad \text{if}\ 0 < \alpha < 1,\\
- 2 \int_1^\lambda G\left(\sqrt{\zeta^2 + (2h)^2}\right) \diff{\zeta} & \quad \text{if}\ \alpha = 1,\\
- 2 \int_0^\lambda G\left(\sqrt{\zeta^2 + (2h)^2}\right) \diff{\zeta} & \quad \text{if}\ 1 < \alpha \leq 2,
\end{dcases}\\[1ex]
v_3(\lambda) &=
\begin{dcases}
- 2 \int_0^\lambda G\left(\sqrt{\zeta^2 + (2h)^2}\right) \diff{\zeta} & \quad \text{if}\ 0 < \alpha < 1,\\
- 2 \int_0^1 G\left(\sqrt{\zeta^2 + (2h)^2}\right) \diff{\zeta} & \quad \text{if}\ \alpha = 1,\\
0 & \quad \text{if}\ 1 < \alpha \leq 2,
\end{dcases}
\label{defv3}
\end{align}
and $\L_{2, \lambda}^* \psi \to \L_2 \psi$ as $\lambda \to \infty$, where $\L_2$ is the Fourier multiplier with symbol
\begin{align}
\label{L2Symb}
\begin{split}
b_2(\xi) &= \int_\R G\left(\sqrt{\zeta^2 + (2h)^2}\right) e^{i \xi \zeta} \diff{\zeta}\\
& =
\begin{dcases}
\frac{2 \sqrt{\pi}}{\Gamma\left(1 - \frac{\alpha}{2}\right) (4h)^{\frac{1 - \alpha}{2}}} |\xi|^{\frac{1 - \alpha}{2}} K_{\frac{1 - \alpha}{2}}\left(2 h |\xi|\right) & \quad \text{if}\ 0 < \alpha < 2,\\
\frac{e^{- 2 h |\xi|}}{2 |\xi|} & \quad \text{if}\ \alpha = 2.
\end{dcases}
\end{split}
\end{align}
In \eqref{L2Symb}, we use the definition of $K_\nu$ in \eqref{BesselK} for $0<\alpha <2$, and for $\alpha = 2$, we use the fact that
\[
\F\bigg[\frac{1}{|\cdot|^2 + c}\bigg](\xi) = \frac{\pi}{\sqrt{c}} e^{-\sqrt{c} |\xi|}
\]
for any $c > 0$, which gives
\[
-\frac{1}{4\pi} \F[\log(|\cdot|^2 + (2h)^2)](\xi) = -\frac{1}{4} \int \frac{e^{-\sqrt{c} |\xi|}}{\sqrt{c}} \diff{c} \bigg|_{c = (2h)^2} = \frac{e^{-2h |\xi|}}{2 |\xi|}.
\]

We denote by $v_4$ the limit
\begin{align*}
v_4 = \lim_{\lambda \to \infty} v_3(\lambda) = \begin{dcases}
- B\big(1 / 2, (1 - \alpha) / 2\big) (2h)^{\alpha - 1} & \quad \text{if}\ 0 < \alpha < 1,\\
2 \log(2 h) - 2 \log\left(1 + \sqrt{1 + (2h)^2}\right) & \quad \text{if}\ \alpha = 1,\\
0 & \quad \text{if}\ 1 < \alpha \leq 2,
\end{dcases}
\end{align*}
where $v_3(\lambda)$ is given in \eqref{defv3}, with $G$ given by \eqref{GreenF}, and $B$ is the Beta function
\begin{equation}
\label{betafn}
B(a, b) = \frac{\Gamma(a) \Gamma(b)}{\Gamma(a + b)}.
\end{equation}

The cut-off system \eqref{lambda-K-G} can then be written as
\begin{align*}
\begin{split}
& \varphi_t(x, t) + \left[\Theta_+ v_1(\lambda) + \Theta_- v_2(\lambda) + \Theta_- v_3(\lambda)\right] \varphi_x(x, t) + \Theta_+ \L_{1, \lambda}^* \varphi_x(x, t) + \Theta_- \L_{2, \lambda}^* \psi_x(x, t)\\
& \quad + \Theta_+ \px \int_{-\lambda}^\lambda H_1\big(\zeta, \varphi(x + \zeta, t) - \varphi(x, t)\big) \diff{\zeta} + \Theta_- \px \int_{-\lambda}^\lambda H_2\big(\zeta, -2h + \psi(x + \zeta, t) - \varphi(x, t)\big) \diff{\zeta} = 0,\\
& \psi_t(x, t) + \left[\Theta_- v_1(\lambda) + \Theta_+ v_2(\lambda) + \Theta_+ v_3(\lambda)\right] \psi_x(x, t) + \Theta_- \L_{1, \lambda}^* \psi_x(x, t) + \Theta_+ \L_{2, \lambda}^* \varphi_x(x, t)\\
& \quad + \Theta_- \px \int_{-\lambda}^\lambda H_1\big(\zeta, \psi(x + \zeta, t) - \psi(x, t)\big) \diff{\zeta} + \Theta_+ \px \int_{-\lambda}^\lambda H_2\big(\zeta, 2h + \varphi(x + \zeta, t) - \psi(x, t)\big) \diff{\zeta} = 0.
\end{split}
\end{align*}

In the limit $\lambda \to \infty$, the possibly problematic terms in these equations are $\left[\Theta_+ v_1(\lambda) + \Theta_- v_2(\lambda)\right] \vp_x(x, t)$ and $\left[\Theta_- v_1(\lambda) + \Theta_+ v_2(\lambda)\right] \psi_x(x, t)$. The only case when these two terms converge to finite limits are when $\Theta_+ = - \Theta_-$ or $0 < \alpha < 1$. Otherwise, we regularize the equations by choosing a suitable Galilean transformation. Indeed, if we choose
\[
v(\lambda) = \frac{\Theta_+ + \Theta_-}{2} \left(v_1(\lambda) + v_2(\lambda) + v_3(\lambda)\right),
\]
and make a Galilean transformation $x \mapsto x - v(\lambda) t$, then the system becomes
\begin{align*}
\begin{split}
& \varphi_t(x, t) + \frac{\Theta_+ - \Theta_-}{2} \left(v_1(\lambda) -  v_2(\lambda) - v_3(\lambda)\right) \varphi_x(x, t) + \Theta_+ \L_{1, \lambda}^* \varphi_x(x, t) + \Theta_- \L_{2, \lambda}^* \psi_x(x, t)\\
& \qquad + \Theta_+ \px \int_{-\lambda}^\lambda H_1\big(\zeta, \varphi(x + \zeta, t) - \varphi(x, t)\big) \diff{\zeta} + \Theta_- \px \int_{-\lambda}^\lambda H_2\big(\zeta, -2h + \psi(x + \zeta, t) - \varphi(x, t)\big) \diff{\zeta} = 0,\\
& \psi_t(x, t) - \frac{\Theta_+ - \Theta_-}{2} \left(v_1(\lambda) - v_2(\lambda) - v_3(\lambda)\right) \psi_x(x, t) + \Theta_- \L_{1, \lambda}^* \psi_x(x, t) + \Theta_+ \L_{2, \lambda}^* \varphi_x(x, t)\\
& \qquad + \Theta_- \px \int_{-\lambda}^\lambda H_1\big(\zeta, \psi(x + \zeta, t) - \psi(x, t)\big) \diff{\zeta} + \Theta_+ \px \int_{-\lambda}^\lambda H_2\big(\zeta, 2h + \varphi(x + \zeta, t) - \psi(x, t)\big) \diff{\zeta} = 0.
\end{split}
\end{align*}

The asymptotic behavior of $G(\zeta)$ and $G\left(\sqrt{\zeta^2 + (2h)^2}\right)$
as $\zeta \to \infty$ is given by
\begin{align*}
G(\zeta) & \sim
\begin{dcases}
\frac{1}{\zeta^{2 - \alpha}} & \quad \text{if}\ 1 \leq \alpha < 2,\\[1ex]
- \frac{1}{2 \pi} \log \zeta & \quad \text{if}\ \alpha = 2,
\end{dcases}\\[1ex]
G\left(\sqrt{\zeta^2 + (2h)^2}\right) & \sim
\begin{dcases}
\frac{1}{\zeta^{2 - \alpha}} + O\bigg(\frac{h}{\zeta^{4 - \zeta}}\bigg) & \quad \text{if}\ 1 \leq \alpha < 2,\\[1ex]
- \frac{1}{2 \pi} \log \zeta + O\bigg(\frac{h}{\zeta^2}\bigg) & \quad \text{if}\ \alpha = 2.
\end{dcases}
\end{align*}

Therefore, from \eqref{defv1} and \eqref{defv2}, we see that $v_1(\lambda) - v_2(\lambda)$ converges as $\lambda \to \infty$, and we define
\begin{align*}
v_5 = \lim_{\lambda \to \infty} \left[v_1(\lambda) - v_2(\lambda)\right] =
\begin{dcases}
0 & \quad \text{if}\ 0 < \alpha < 1,\\
2 \log{2} - 2 \log\left(1 + \sqrt{1 + (2h)^2}\right) & \quad \text{if}\ \alpha = 1,\\
B\big(1 / 2, (1 - \alpha) / 2\big) (2h)^{\alpha - 1} & \quad \text{if}\ 1 < \alpha < 2,\\
- h & \quad \text{if}\ \alpha = 2.
\end{dcases}
\end{align*}

Putting everything together and letting $\lambda \to \infty$, we get the regularized system in conservative form
\begin{align}
\label{reg-GSQG}
\begin{split}
& \varphi_t(x, t) + v \varphi_x(x, t) + \Theta_+ {\L}_1 \varphi_x(x, t) + \Theta_- \L_2 \psi_x(x, t)\\
&+ \Theta_+ \px \int_\R H_1\big(\zeta, \varphi(x + \zeta, t) - \varphi(x, t)\big) \diff{\zeta} + \Theta_- \px \int_\R H_2\big(\zeta, -2h + \psi(x + \zeta, t) - \varphi(x, t)\big) \diff{\zeta} = 0,\\
& \psi_t(x, t) - v \psi_x(x, t) + \Theta_- {\L}_1 \psi_x(x, t) + \Theta_+ \L_2 \varphi_x(x, t)\\
&+ \Theta_- \px \int_\R H_1\big(\zeta, \psi(x + \zeta, t) - \psi(x, t)\big) \diff{\zeta} + \Theta_+ \px \int_\R H_2\big(\zeta, 2h + \varphi(x + \zeta, t) - \psi(x, t)\big) \diff{\zeta} = 0,
\end{split}
\end{align}
where $H_1$, $H_2$ are given in \eqref{H12}, the symbols of $\L_1$, $\L_2$ are given in \eqref{L1Symb}--\eqref{L2Symb}, and
\begin{equation}
\label{defv}
v= \frac{\Theta_+ - \Theta_-}{2}\left(v_5 - v_4\right),\qquad
v_5-v_4 = \begin{cases}
B(1/2,(1-\alpha)/2) (2h)^{\alpha-1} & \text{if $\alpha\in (0,1)\cup(1,2)$},
\\
-2\log h & \text{if $\alpha =1$},
\\
-h & \text{if $\alpha = 2$}.\end{cases}
\end{equation}
One can also take the derivatives inside the integrals to obtain the non-conservative form
\begin{eqnarray}
\label{reg-GSQG-nc}
\begin{split}
& \varphi_t(x, t) + v \varphi_x(x, t) + \Theta_+ {\L}_1 \varphi_x(x, t) + \Theta_- \L_2 \psi_x(x, t)\\
&+ \Theta_+ \int_\R \left[\varphi_x(x + \zeta, t) - \varphi_x(x, t)\right] \left\{G\left(\sqrt{\zeta^2 + [\vp(x + \zeta, t) - \vp(x, t)]^2}\right) - G(\zeta)\right\} \diff{\zeta}\\
&+ \Theta_- \int_\R \left[\psi_x(x + \zeta, t) - \vp_x(x, t)\right] \left\{G\left(\sqrt{\zeta^2 + [-2h + \psi(x + \zeta, t) - \varphi(x, t)]^2}\right) - G\left(\sqrt{\zeta^2 + (2h)^2}\right)\right\} \diff{\zeta} = 0,\\
& \psi_t(x, t) - v \psi_x(x, t) + \Theta_- {\L}_1 \psi_x(x, t) + \Theta_+ \L_2 \varphi_x(x, t)\\
&+ \Theta_- \int_\R \left[\psi_x(x + \zeta, t) - \psi_x(x, t)\right] \left\{G\left(\sqrt{\zeta^2 + [\psi(x + \zeta, t) - \psi(x, t)]^2}\right) - G(\zeta)\right\} \diff{\zeta}\\
& + \Theta_+ \int_\R \left[\vp_x(x + \zeta, t) - \psi_x(x, t)\right] \left\{G\left(\sqrt{\zeta^2 + [2h + \vp(x + \zeta, t) - \psi(x, t)]^2}\right) - G\left(\sqrt{\zeta^2 + (2h)^2}\right)\right\} \diff{\zeta} = 0.
\end{split}
\end{eqnarray}

The system \eqref{reg-GSQG} has the Hamiltonian form
\begin{align*}
\varphi_t + J_+ \frac{\delta \mathcal{H}}{\delta \varphi}=0,
\quad
\psi_t + J_-\frac{\delta \mathcal{H}}{\delta \psi} = 0,\qquad J_+ = \frac{1}{\Theta_+} \px,\quad J_- = \frac{1}{\Theta_-} \px,
\end{align*}
with the Hamiltonian
\begin{align*}
&\mathcal{H}(\varphi, \psi) = \frac{1}{2} \int_\R \bigg\{v \Theta_+\varphi^2 - v \Theta_- \psi^2 + \Theta_+^2 \varphi {\L}_1 \varphi
+ 2\Theta_+ \Theta_- \varphi \L_2 \psi + \Theta_- ^2\psi {\L}_1 \psi\bigg\} \diff{x}\\
& \qquad +\frac{1}{2} \int_{\R^2} \bigg\{\Theta_+^2 F_1(x - x', \varphi - \varphi')
+
 2\Theta_+\Theta_- F_2(x-x', 2h+\varphi-\psi')
+ \Theta_- ^2 F_1(x - x',\psi - \psi') \bigg\} \diff{x} \diff{x'},
\end{align*}
where $\varphi=\varphi(x,t)$, $\varphi' = \varphi(x',t)$, $\psi=\psi(x,t)$, $\psi' = \psi(x',t)$,
and the functions $F_1$, $F_2$ satisfy
\[
F_{1y}(x,y) = H_1(x,y),\qquad F_{2y}(x,y) = H_2(x,y).
\]

\subsection{Regularized systems}

We write out  specific expressions for the non-conservative two-front systems \eqref{reg-GSQG-nc}  in the cases
$\alpha = 2$ (Euler), $\alpha = 1$ (SQG), and $0<\alpha<1$ or $1< \alpha <2$ (GSQG).

\subsubsection{Euler equations ($\alpha = 2$)}
In the case of Euler equations, the Green's function is
$G(x) = - \log|x|/2\pi$, and
the two-front Euler system is
\begin{eqnarray}
\label{eulersys}
\begin{split}
& \varphi_t(x, t) - \frac{\Theta_+ - \Theta_-}{2} h \varphi_x(x, t) - \frac{\Theta_+}{2} \hilbert \varphi(x, t) - \frac{\Theta_-}{2} e^{-2 h |\px|} \hilbert \psi(x, t)\\
&- \frac{\Theta_+}{2 \pi} \int_\R \big[\varphi_x(x + \zeta, t) - \varphi_x(x, t)\big] \log\left(\sqrt{1 + \bigg[\frac{\varphi(x + \zeta, t) - \varphi(x, t)}{\zeta}\bigg]^2}\right) \diff{\zeta}\\
&- \frac{\Theta_-}{2 \pi} \int_\R \big[\psi_x(x + \zeta, t) - \varphi_x(x, t)\big] \log\left(\sqrt{1 - \frac{4h (\psi(x + \zeta, t) - \varphi(x, t))}{\zeta^2 + (2h)^2} + \frac{(\psi(x + \zeta, t) - \varphi(x, t))^2}{\zeta^2 + (2h)^2}}\right) \diff{\zeta} = 0,\\
& \psi_t(x, t) + \frac{\Theta_+ - \Theta_-}{2} h \psi_x(x, t) - \frac{\Theta_-}{2} \hilbert \psi(x, t) - \frac{\Theta_+}{2} e^{- 2 h |\px|} \hilbert \varphi(x, t)\\
&- \frac{\Theta_-}{2 \pi} \int_\R \big[\psi_x(x + \zeta, t) - \psi_x(x, t)\big] \log\left(\sqrt{1 + \bigg[\frac{\psi(x + \zeta, t) - \psi(x, t)}{\zeta}\bigg]^2}\right) \diff{\zeta}\\
&- \frac{\Theta_+}{2 \pi} \int_\R \big[\varphi_x(x + \zeta, t) - \psi_x(x, t)\big] \log\left(\sqrt{1 + \frac{4h (\varphi(x + \zeta, t) - \psi(x, t))}{\zeta^2 + (2h)^2} + \frac{(\varphi(x + \zeta, t) - \psi(x, t))^2}{\zeta^2 + (2h)^2}}\right) \diff{\zeta} = 0.
\end{split}
\end{eqnarray}
Here, $\hilbert$ is the Hilbert transform with symbol $- i \sgn \xi$.

\subsubsection{SQG equations ($\alpha = 1$)}

In the case of SQG equation, the Green's function is
$G(x) = {1}/{|x|}$,
and (with an additional Galilean transformation $x \mapsto x + \gamma (\Theta_+ + \Theta_-) t$) the two-front SQG system is
\begin{eqnarray}
\label{sqgsys}
\begin{split}
& \varphi_t(x, t) - (\Theta_+ - \Theta_-) (\gamma + \log{h}) \varphi_x(x, t) - 2 \Theta_+ \log|\px| \varphi_x(x, t) + 2 \Theta_- K_0(2 h |\px|) \psi_x(x, t)\\
&+ \Theta_+ \int_\R \big[\varphi_x(x + \zeta, t) - \varphi_x(x, t)\big] \bigg\{\frac{1}{\sqrt{\zeta^2 + (\varphi(x + \zeta, t) - \varphi(x, t))^2}} - \frac{1}{|\zeta|}\bigg\} \diff{\zeta}\\
&+ \Theta_- \int_\R \big[\psi_x(x + \zeta, t) - \varphi_x(x, t)\big] \bigg\{\frac{1}{\sqrt{\zeta^2 + (-2h + \psi(x + \zeta, t) - \varphi(x, t))^2}} - \frac{1}{\sqrt{\zeta^2 + (2h)^2}}\bigg\} \diff{\zeta} = 0,\\
& \psi_t(x, t) + (\Theta_+ - \Theta_-) (\gamma + \log{h}) \psi_x(x, t) - 2 \Theta_- \log|\px| \psi_x(x, t) + 2 \Theta_+ K_0(2 h |\px|) \varphi_x(x, t)\\
&+ \Theta_- \int_\R \big[\psi_x(x + \zeta, t) - \psi_x(x, t)\big] \bigg\{\frac{1}{\sqrt{\zeta^2 + (\psi(x + \zeta, t) - \psi(x, t))^2}} - \frac{1}{|\zeta|}\bigg\} \diff{\zeta}\\
&+ \Theta_+ \int_\R \big[\varphi_x(x + \zeta, t) - \psi_x(x, t)\big] \bigg\{\frac{1}{\sqrt{\zeta^2 + (2h + \varphi(x + \zeta, t) - \psi(x, t))^2}} - \frac{1}{\sqrt{\zeta^2 + (2h)^2}}\bigg\} \diff{\zeta} = 0.
\end{split}
\end{eqnarray}

\subsubsection{GSQG equations}
In this case, the Green's function is
$G(x) = {1}/{|x|^{2 - \alpha}}$, $\alpha \in (0, 1) \cup (1, 2)$, the two-front GSQG system is
\begin{eqnarray}
\label{gsqgsys12}
\begin{split}
& \varphi_t(x, t) + \frac{\Theta_+ - \Theta_-}{2} B\bigg(\frac{1}{2}, \frac{1 - \alpha}{2}\bigg) (2h)^{\alpha - 1} \varphi_x(x, t)\\
&- 2 \Theta_+ \sin\bigg(\frac{\pi \alpha}{2}\bigg) \Gamma(\alpha - 1) |\px|^{2 - \alpha} \hilbert \varphi(x, t) + \Theta_- \frac{2 \sqrt{\pi}}{\Gamma\left(1 - \frac{\alpha}{2}\right) (4h)^{\frac{1 - \alpha}{2}}} |\px|^{\frac{1 - \alpha}{2}} K_{\frac{1 - \alpha}{2}}(2 h |\px|) \psi_x(x, t)\\
&+ \Theta_+ \int_\R \big[\varphi_x(x + \zeta, t) - \varphi_x(x, t)\big] \bigg\{\frac{1}{(\zeta^2 + (\varphi(x + \zeta, t) - \varphi(x, t))^2)^{1 - \frac{\alpha}{2}}} - \frac{1}{|\zeta|^{2 - \alpha}}\bigg\} \diff{\zeta}\\
&+ \Theta_- \int_\R \big[\psi_x(x + \zeta, t) - \varphi_x(x, t)\big] \bigg\{\frac{1}{(\zeta^2 + (- 2 h + \psi(x + \zeta, t) - \varphi(x, t))^2)^{1 - \frac{\alpha}{2}}} - \frac{1}{(\zeta^2 + (2 h)^2)^{1 - \frac{\alpha}{2}}}\bigg\} \diff{\zeta} = 0,\\
& \psi_t(x, t) - \frac{\Theta_+ - \Theta_-}{2} B\bigg(\frac{1}{2}, \frac{1 - \alpha}{2}\bigg) (2h)^{\alpha - 1} \psi_x(x, t)\\
&- 2 \Theta_- \sin\bigg(\frac{\pi \alpha}{2}\bigg) \Gamma(\alpha - 1) |\px|^{2 - \alpha} \hilbert \psi(x, t) + \Theta_+ \frac{2 \sqrt{\pi}}{\Gamma\left(1 - \frac{\alpha}{2}\right) (4h)^{\frac{1 - \alpha}{2}}} |\px|^{\frac{1 - \alpha}{2}} K_{\frac{1 - \alpha}{2}}(2 h |\px|) \varphi_x(x, t)\\
&+ \Theta_- \int_\R \big[\psi_x(x + \zeta, t) - \psi_x(x, t)\big] \bigg\{\frac{1}{(\zeta^2 + (\psi(x + \zeta, t) - \psi(x, t))^2)^{1 - \frac{\alpha}{2}}} - \frac{1}{|\zeta|^{2 - \alpha}}\bigg\} \diff{\zeta}\\
&+ \Theta_+ \int_\R \big[\varphi_x(x + \zeta, t) - \psi_x(x, t)\big] \bigg\{\frac{1}{(\zeta^2 + (2 h + \varphi(x + \zeta, t) - \psi(x, t))^2)^{1 - \frac{\alpha}{2}}} - \frac{1}{(\zeta^2 + (2 h)^2)^{1 - \frac{\alpha}{2}}}\bigg\} \diff{\zeta} = 0.
\end{split}
\end{eqnarray}

\subsection{Scalar reductions of the equations}

In this subsection, we write out two scalar equations that arise as reductions of the system \eqref{reg-GSQG-nc} when
the jumps are symmetric or anti-symmetric.

\subsubsection{Symmetric reduction}

If $\Theta_+ = \Theta_-$, then $v=0$ from \eqref{defv}, and the system \eqref{reg-GSQG-nc} is compatible with solutions of the form
$\psi(x, t) = - \vp(x, t)$,
when it reduces to a scalar equation for $\vp$.  Writing $\Theta = \Theta_+ = \Theta_-$, we find that the equation becomes
\begin{eqnarray}
\label{symm-gsqg}
\begin{split}
& \varphi_t(x, t)
 + \Theta \left(\L_1 - \L_2\right) \varphi_x(x, t)
 \\
&+ \Theta \int_\R \left[\varphi_x(x + \zeta, t) - \varphi_x(x, t)\right] \left\{G\left(\sqrt{\zeta^2 + [\vp(x + \zeta, t) - \vp(x, t)]^2}\right) - G(\zeta)\right\} \diff{\zeta}\\
&+ \Theta \int_\R \left[\vp_x(x + \zeta, t) + \vp_x(x, t)\right] \left\{G\left(\sqrt{\zeta^2 + [2h + \vp(x + \zeta, t) + \varphi(x, t)]^2}\right) - G\left(\sqrt{\zeta^2 + (2h)^2}\right)\right\} \diff{\zeta} = 0.
\end{split}
\end{eqnarray}

\begin{figure}[h]
\centering
\begin{subfigure}[]{0.475\textwidth}
\includegraphics[width=\textwidth]{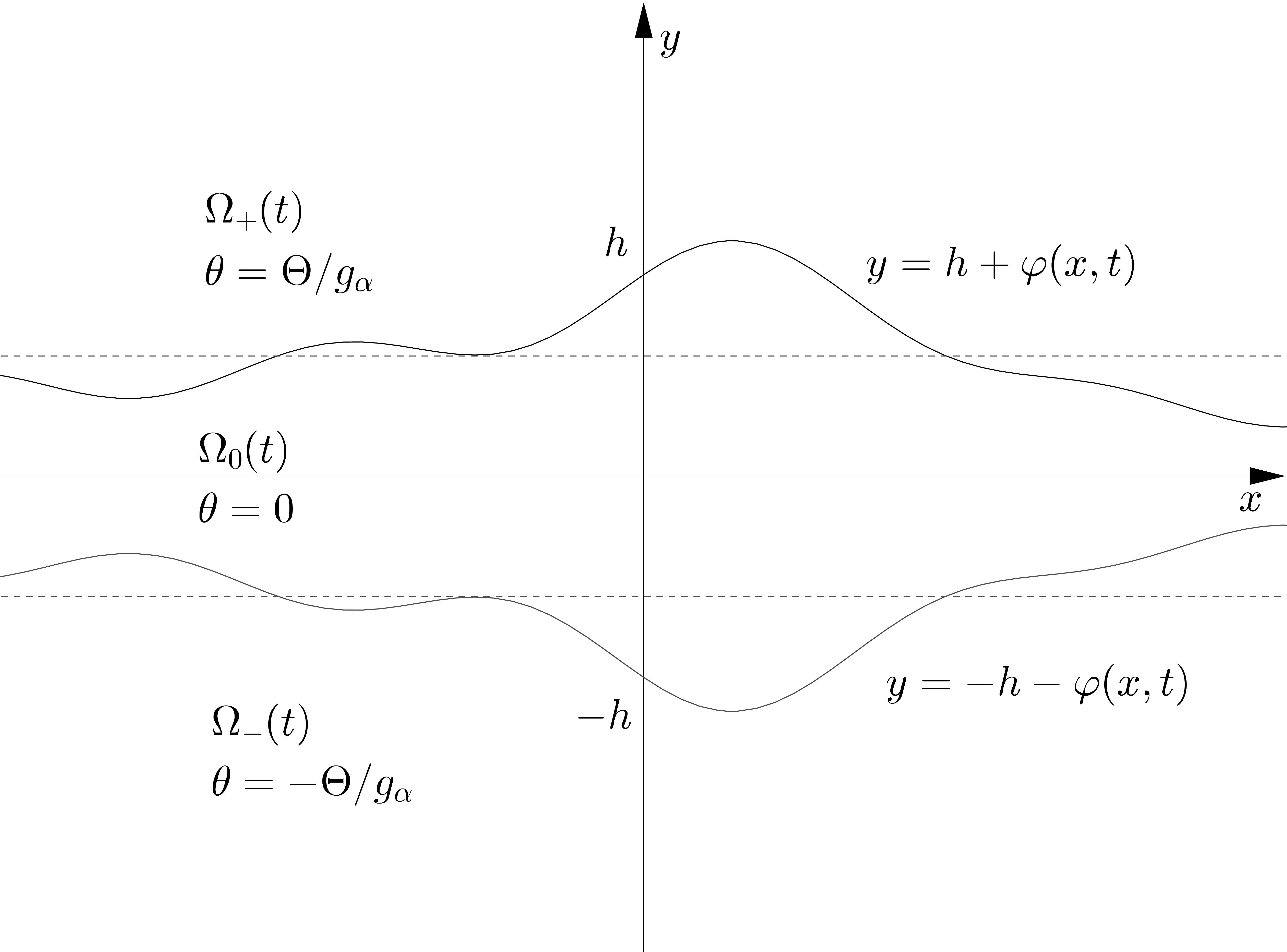}
\caption{Symmetric GSQG fronts.}
\end{subfigure}~
\begin{subfigure}[]{0.475\textwidth}
\includegraphics[width=\textwidth]{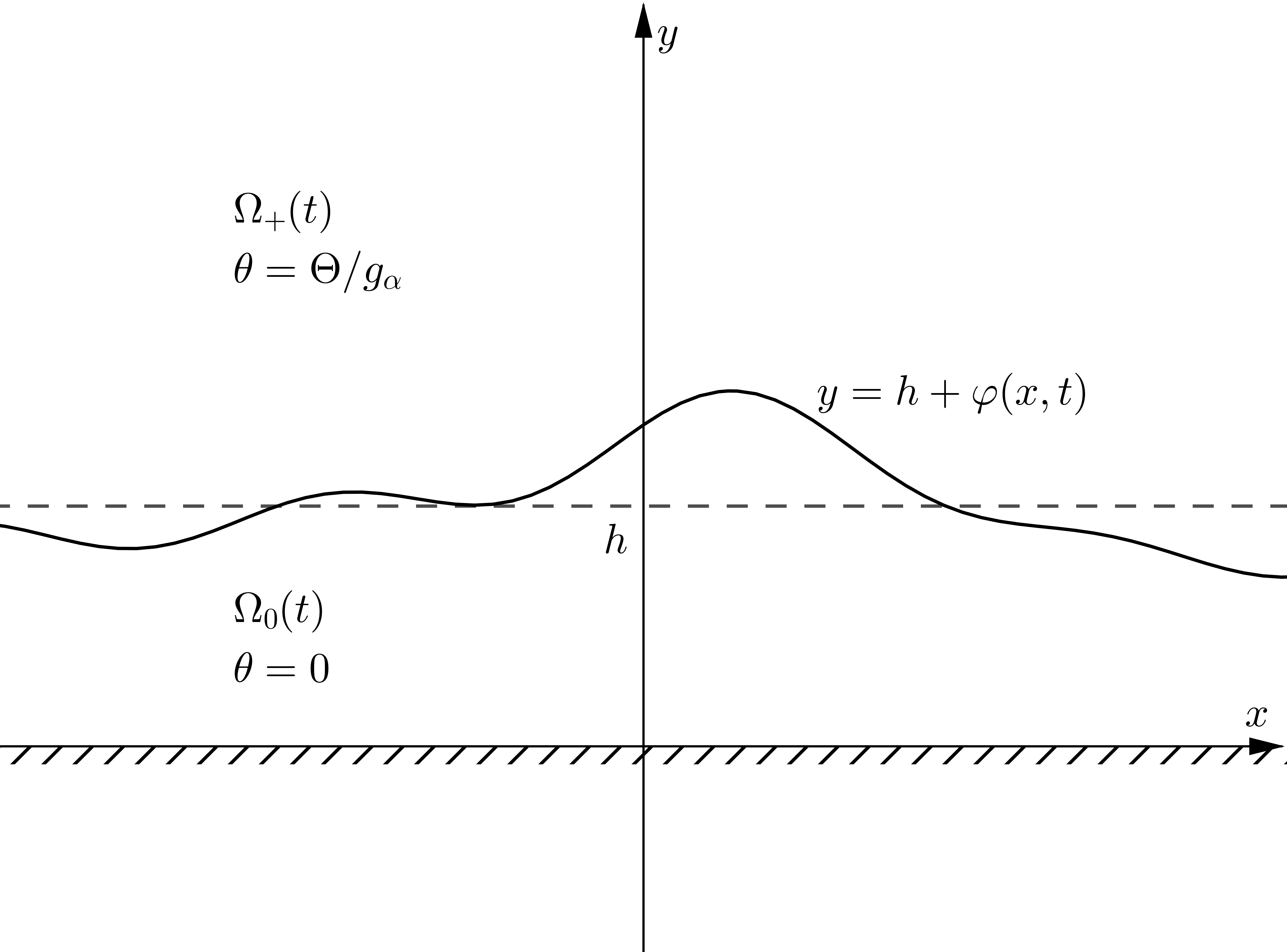}
\caption{GSQG front with a rigid flat bottom.}
\end{subfigure}
\caption{Symmetric reduction of GSQG system.}\label{fig:symm-red}
\end{figure}

For the GSQG equations \eqref{gsqg} in the spatial upper half-plane $\R \times \R_+$ with no-flow boundary conditions on a rigid boundary $y = 0$ (see Figure \ref{fig:symm-red} and \cite{GP18p, KRYZ16, KYZ17}), we find by the method of images that
\[
\u(\x, t) =
g_\alpha \int_{\R \times \R_+} \biggl\{\nabla_{\x}^\perp G(|\x - \x'|)-\nabla_{\x}^\perp G(|\x - \bar{\x}'|)\biggl\}  \theta(\x, t) \diff{\x'},
\]
where $\bar{\x}' = (x',-y')$ if $\x' = (x',y')$. In this setting, if a front is located at $y = h + \vp(x, t) > 0$, and
\[
\theta(x, y, t) = \begin{dcases}
\Theta/g_\alpha & \qquad \text{if $y > h + \vp(x, t)$},\\
0 & \qquad \text{if $0<y < h + \vp(x, t)$},
\end{dcases}
\]
then the regularized contour dynamics equation for a front in the half-plane coincides with \eqref{symm-gsqg}.

\subsubsection{Anti-symmetric reduction}

If $\Theta_+ = - \Theta_-$, then \eqref{reg-GSQG-nc} is compatible with solutions of the form
\[
\vp(x, t) = \vp(x, t), \qquad \psi(x, t) = - \vp(-x, t),
\]
and it reduces to a scalar equation for $\vp$ (see Figure \ref{fig:antisymm-red}). Writing $\Theta = \Theta_+ = -\Theta_-$ and making a Galilean transformation $x \mapsto x - v t$, we find that the equation becomes
\begin{align}
\label{anti-symm-gsqg}
\begin{split}
& \varphi_t(x, t)
 + \Theta \L_1 \varphi_x(x, t) - \Theta \L_2 \vp_x(-x, t)\\
&+ \Theta \int_\R \left[\varphi_x(x + \zeta, t) - \varphi_x(x, t)\right] \left\{G\left(\sqrt{\zeta^2 + [\vp(x + \zeta, t) - \vp(x, t)]^2}\right) - G(\zeta)\right\} \diff{\zeta}\\
&- \Theta \int_\R \left[\vp_x(-x - \zeta, t) - \vp_x(x, t)\right] \left\{G\left(\sqrt{\zeta^2 + [2h + \vp(-x - \zeta, t) + \varphi(x, t)]^2}\right) - G\left(\sqrt{\zeta^2 + (2h)^2}\right)\right\} \diff{\zeta} = 0.
\end{split}
\end{align}
\begin{figure}[h]
\centering
\includegraphics[width=0.6\textwidth]{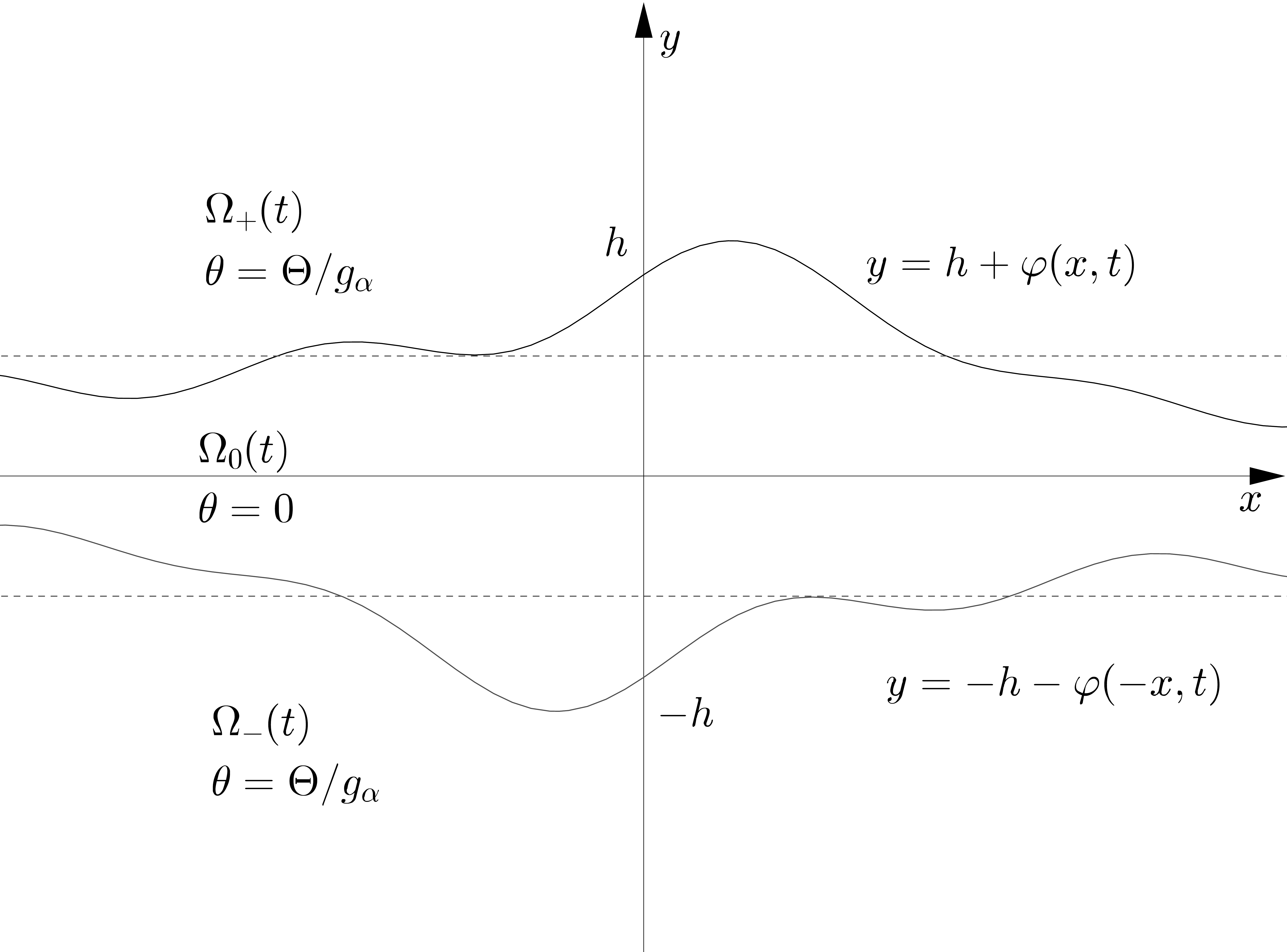}
\caption{Anti-symmetric reduction of GSQG system.}\label{fig:antisymm-red}
\end{figure}

\subsection{Expanded systems}

We consider fronts with small amplitude and small slope, i.e., $|\varphi|, |\psi| \ll h$ and $|\varphi_x|, |\psi_x| \ll 1$, and carry out a multilinear expansion of the nonlinearities in the systems derived in the previous subsection. We will use the expanded system in the local existence proof for $0<\alpha \le 1$, and the smallness condition (ii) in Theorems \ref{existence01}--\ref{existence1} is sufficient
to justify the expansion.

\noindent {\bf (1)} When $0 < \alpha < 2$, we have the Taylor expansion
\begin{align}
\label{alpha-taylor}
(1 + x)^{-1 + \alpha/2} = 1 + \sum_{n = 1}^\infty c_n x^n,
\qquad
c_n = \frac{\Gamma\left(\frac{\alpha}{2}\right)}{\Gamma(n + 1) \Gamma\left(\frac{\alpha}{2} - n\right)}.
\end{align}
Taking Fourier transforms and letting $\etab_n = (\eta_1, \eta_2, \dotsc, \eta_{2n + 1})$, we find that the first nonlinear term in the first equation of the systems \eqref{sqgsys}--\eqref{gsqgsys12} can be written as
\begin{align*}
& \int_\R \big[\varphi_x(x + \zeta, t) - \varphi_x(x, t)\big] \bigg\{\frac{1}{\big(\zeta^2 + (\varphi(x + \zeta) - \varphi(x, t))^2\big)^{1 - \frac{\alpha}{2}}} - \frac{1}{|\zeta|^{2 - \alpha}}\bigg\} \diff{\zeta}\\
&=- \sum_{n = 1}^\infty \frac{c_n}{2n + 1} \px \int_\R \bigg[\frac{\varphi(x, t) - \varphi(x + \zeta, t)}{\zeta}\bigg]^{2n + 1} |\zeta|^{\alpha - 1} \sgn{\zeta} \diff{\zeta}\\
&= - \sum_{n = 1}^\infty \frac{c_n}{2n + 1} \px \int_{\R^{2n + 1}} \Tb_n(\etab_n) \hat{\varphi}(\eta_1, t) \hat{\varphi}(\eta_2, t) \dotsm \hat{\varphi}(\eta_{2n + 1}, t) e^{i (\eta_1 + \eta_2 + \dotsb + \eta_{2n + 1}) x} \diff{\etab_n},
\end{align*}
where
\begin{align}
\label{defTn}
\Tb_n(\etab_n) = \int_{\R} \frac{\prod_{j = 1}^{2n + 1} \left(1 - e^{i \eta_j \zeta}\right)}{\zeta^{2n + 1}} |\zeta|^{\alpha - 1} \sgn{\zeta} \diff{\zeta}.
\end{align}
Replacing $\varphi$ by $\psi$ gives the first nonlinear term in the second equation of the systems \eqref{sqgsys}--\eqref{gsqgsys12}.

For the second nonlinear term of the first equation of these systems, we take Fourier transforms and use \eqref{BesselK} to get
\[
\int_\R \frac{\big(\psi(x + \zeta, t)\big)^m}{(\zeta^2 + (2h)^2)^{n + 1 - \frac{\alpha}{2}}} \diff{\zeta} = \begin{dcases}
B\bigg(\frac{1}{2}, n + \frac{1 - \alpha}{2}\bigg) (2h)^{\alpha - 2n - 1}& \quad \text{if}\ m = 0,\\
\frac{2 \sqrt{\pi}}{\Gamma\left(n + 1 - \frac{\alpha}{2}\right) (4h)^{n + \frac{1 - \alpha}{2}}} |\px|^{n + \frac{1 - \alpha}{2}} K_{n + \frac{1 - \alpha}{2}}(2 h |\px|) \big(\psi(x, t)\big)^m & \quad \text{if}\ m \geq 1.
\end{dcases}
\]
Then, using \eqref{alpha-taylor}, we get
\begin{align*}
& \int_\R \big[\psi_x(x + \zeta, t) - \varphi_x(x, t)\big] \bigg\{\frac{1}{(\zeta^2 + (- 2 h + \psi(x + \zeta, t) - \varphi(x, t))^2)^{1 - \frac{\alpha}{2}}} - \frac{1}{(\zeta^2 + (2 h)^2)^{1 - \frac{\alpha}{2}}}\bigg\} \diff{\zeta}\\
&= \sum_{n = 1}^\infty c_n \int_\R \frac{\psi_x(x + \zeta, t) - \varphi_x(x, t)}{(\zeta^2 + (2h)^2)^{n + 1 - \frac{\alpha}{2}}} \Big[- 4h \big(\psi(x + \zeta, t) - \varphi(x, t)\big) + \big(\psi(x + \zeta, t) - \varphi(x, t)\big)^2\Big]^n \diff{\zeta}\\
&= \sum_{n = 1}^\infty \sum_{\ell = 0}^n d_{n, \ell} \px \int_\R \frac{\big(\psi(x + \zeta, t) - \varphi(x, t)\big)^{2n - \ell + 1}}{(\zeta^2 + (2h)^2)^{n + 1 - \frac{\alpha}{2}}} \diff{\zeta}\\
&= \sum_{n = 1}^\infty \sum_{\ell = 0}^n \sum_{m = 0}^{2n - \ell + 1} d_{n, \ell, m} \px \bigg\{\big(\varphi(x, t)\big)^{2n - \ell + 1 - m} \int_\R \frac{\big(\psi(x + \zeta, t)\big)^m}{(\zeta^2 + (2h)^2)^{n + 1 - \frac{\alpha}{2}}} \diff{\zeta}\bigg\}\\
&= \sum_{n = 1}^\infty \sum_{\ell = 0}^n d_{n, \ell, 0, 1} \px \left\{\big(\varphi(x, t)\big)^{2n - \ell + 1}\right\}\\
& \qquad + \sum_{n = 1}^\infty \sum_{\ell = 0}^n \sum_{m = 1}^{2n - \ell + 1} d_{n, \ell, m, 1} \px\Big\{\big(\varphi(x, t)\big)^{2n - \ell + 1 - m} |\px|^{n + \frac{1 - \alpha}{2}} K_{n + \frac{1 - \alpha}{2}}(2 h |\px|) \big(\psi(x, t)\big)^m\Big\},
\end{align*}
where
\begin{align*}
\begin{split}
d_{n, \ell} &= \frac{\Gamma\left(\frac{\alpha}{2}\right) (-4 h)^\ell}{(2n - \ell + 1) \Gamma(\ell + 1) \Gamma(n + 1 - \ell) \Gamma\left(\frac{\alpha}{2} - n\right)},\\
d_{n, \ell, m} &=  \frac{(-1)^{2n + 1 - m} \Gamma\left(\frac{\alpha}{2}\right) \Gamma(2n + 2 - \ell) (4 h)^\ell}{(2n - \ell + 1) \Gamma(\ell + 1) \Gamma(n + 1 - \ell) \Gamma\left(\frac{\alpha}{2} - n\right) \Gamma(m + 1) \Gamma(2n + 2 - m - \ell)},\\
d_{n, \ell, m, 1} &= \begin{dcases}
d_{n, \ell, 0} \cdot \frac{2 \sqrt{\pi} \Gamma\left(n + \frac{1 - \alpha}{2}\right)}{\Gamma\left(n + 1 - \frac{\alpha}{2}\right) (4h)^{n + 1 - \frac{\alpha}{2}}} & \quad \text{if}\ m = 0,\\
d_{n, \ell, m} \cdot \frac{2 \sqrt{\pi}}{\Gamma\left(n + 1 - \frac{\alpha}{2}\right) (4h)^{n + \frac{1 - \alpha}{2}}} & \quad \text{if}\ m \geq 1.
\end{dcases}
\end{split}
\end{align*}
The computation for the second nonlinear term in the second equation of the systems \eqref{sqgsys}--\eqref{gsqgsys12} is similar. We only need to replace $\varphi$ by $\psi$, multiply $d_{n, \ell}$ and $d_{n, \ell, m}$ by $(-1)^\ell$, and replace $d_{n, \ell, m, 1}$ by $d_{n, \ell, m, 2}$ where
\[
d_{n, \ell, m, 2} = (-1)^\ell d_{n, \ell, m, 1}.
\]
All of these constants grow at most exponentially in $n$, so the series in the expanded equations converge when $\vp$, $\psi$ are sufficiently small.

\noindent {\bf (2)} When $\alpha = 2$, we use Taylor's expansion and Fourier transform to find that the first nonlinear term in the first equation of the system \eqref{eulersys} can be written as
\begin{align*}
& -\frac{1}{2 \pi} \int_\R \big[\varphi_x(x + \zeta, t) - \varphi_x(x, t)\big] \log\left(\sqrt{1 + \bigg[\frac{\varphi(x + \zeta, t) - \varphi(x, t)}{\zeta}\bigg]^2}\right)\\
= ~& - \frac{1}{2 \pi} \cdot \frac{1}{2} \sum_{n = 1}^\infty \px \int_\R \frac{(-1)^n}{n (2n + 1)} \bigg[\frac{\varphi(x, t) - \varphi(x + \zeta, t)}{\zeta}\bigg]^{2 n + 1} \zeta \diff{\zeta}\\
= ~& - \frac{1}{2 \pi} \sum_{n = 1}^\infty \tilde{c}_n \px \int_{\R^{2n + 1}} \Tb_n(\etab_n) \hat{\varphi}(\eta_1, t) \hat{\varphi}(\eta_2, t) \dotsm \hat{\varphi}(\eta_{2n + 1}, t) e^{i (\eta_1 + \eta_2 + \dotsb + \eta_{2n + 1}) x} \diff{\etab_n},
\end{align*}
where $\Tb_n$ also lies in the family \eqref{defTn} for $\alpha = 2$ and
\[
\tilde{c}_n = \frac{(-1)^n}{2 n (2n + 1)}.
\]
Substituting $\varphi$ by $\psi$ gives the first nonlinear term in the second equation of the system \eqref{eulersys}.

For the second nonlinear term in the first equation of this system, by Taylor expansion and Fourier transform, we have
\begin{align*}
& - \frac{1}{2 \pi} \int_\R \big[\psi_x(x + \zeta, t) - \varphi_x(x, t)\big] \log\left(\sqrt{1 - \frac{4h (\psi(x + \zeta, t) - \varphi(x, t))}{\zeta^2 + (2h)^2} + \frac{(\psi(x + \zeta, t) - \varphi(x, t))^2}{\zeta^2 + (2h)^2}}\right) \diff{\zeta}\\
&= \frac{1}{2 \pi} \sum_{n = 1}^\infty \frac{(-1)^n}{2 n} \int_\R \frac{\psi_x(x + \zeta, t) - \varphi_x(x, t)}{(\zeta^2 + (2h)^2)^n}\Big[-4h \big(\psi(x + \zeta, t) - \varphi(x, t)\big) + \big(\psi(x + \zeta, t) - \varphi(x, t)\big)^2\Big]^n \diff{\zeta}\\
&= \frac{1}{2 \pi} \sum_{n = 1}^\infty \sum_{\ell = 0}^n \tilde{d}_{n, \ell} \px \int_\R \frac{\big(\psi(x + \zeta, t) - \varphi(x, t)\big)^{2n - \ell + 1}}{(\zeta^2 + (2h)^2)^n} \diff{\zeta}\\
&= \frac{1}{2 \pi} \sum_{n = 1}^\infty \sum_{\ell = 0}^n \sum_{m = 0}^{2n - \ell + 1} \tilde{d}_{n, \ell, m} \px \bigg\{\big(\varphi(x, t)\big)^{2n - \ell + 1 - m} \int_\R \frac{\big(\psi(x + \zeta, t)\big)^m}{(\zeta^2 + (2 h)^2)^n} \diff{\zeta}\bigg\}\\
&= \frac{1}{2 \pi} \sum_{n = 0}^\infty \sum_{\ell = 0}^n \tilde{d}_{n, \ell, 0, 1} \px \Big\{\big(\varphi(x, t)\big)^{2n - \ell + 1}\Big\}\\
&+ \frac{1}{2\pi} \sum_{n = 1}^\infty \sum_{\ell = 0}^n \sum_{m = 1}^\infty \tilde{d}_{n, \ell, m, 1} \px \Big\{\big(\varphi(x, t)\big)^{2n - \ell + 1 - m} |\px|^{n - \frac{1}{2}} K_{n - \frac{1}{2}}(2 h |\px|) \big(\psi(x, t)\big)^m\Big\},
\end{align*}
where
\begin{align*}
\tilde{d}_{n, \ell} &= \frac{(-1)^n \Gamma(n) (-4h)^{\ell}}{2 (2n - \ell + 1) \Gamma(\ell + 1) \Gamma(n + 1 - \ell)},\\
\tilde{d}_{n, \ell, m} &= \frac{(-1)^{n + \ell - m} \Gamma(n) (4h)^{\ell} \Gamma(2n + 2 - \ell)}{2 (2n - \ell + 1) \Gamma(\ell + 1) \Gamma(n + 1 - \ell) \Gamma(m + 1) \Gamma(2n + 2 - m - \ell)},\\
\tilde{d}_{n, \ell, m, 1} &= \begin{dcases}
\tilde{d}_{n, \ell, 0} \cdot \frac{\sqrt{\pi} \Gamma\left(n - \frac{1}{2}\right) (2h)^{1 - 2n}}{\Gamma(n)} & \quad \text{if}\ m = 0,\\
\tilde{d}_{n, \ell, m} \cdot \frac{2 \sqrt{\pi}}{\Gamma(n) (4h)^{n - \frac{1}{2}}} & \quad \text{if}\ m \geq 1.
\end{dcases}
\end{align*}
The calculation for the second nonlinear term in the second equation of the system \eqref{eulersys} is similar; we only need to exchange $\varphi$ and $\psi$, multiply $\tilde{d}_{n, \ell}$ and $\tilde{d}_{n, \ell, m}$ by $(-1)^\ell$, and replace $\tilde{d}_{n, \ell, m, 1}$ by $\tilde{d}_{n, \ell, m, 2}$ where
\[
\tilde{d}_{n, \ell, m, 2} = (-1)^\ell \tilde{d}_{n, \ell, m, 1}.
\]

Finally, we summarize the expanded systems for $0 < \alpha \le 1$. Similar expansions apply for $1<\alpha \le 2$, but, since we do not need them, we will not write them out explicitly here.

When $0 < \alpha < 1$, the expanded two-front GSQG system is
\begin{align*}
& \varphi_t(x, t) + \frac{\Theta_+ - \Theta_-}{2} B\bigg(\frac{1}{2}, \frac{1 - \alpha}{2}\bigg) \frac{1}{(2h)^{1 - \alpha}} \varphi_x(x, t)\\
& \quad + 2 \Theta_+ \sin\bigg(\frac{\pi \alpha}{2}\bigg) \Gamma(\alpha - 1) |\px|^{1 - \alpha} \varphi_x(x, t) + \Theta_- \frac{2 \sqrt{\pi}}{\Gamma\left(1 - \frac{\alpha}{2}\right) (4h)^{\frac{1 - \alpha}{2}}} |\px|^{\frac{1 - \alpha}{2}} K_{\frac{1 - \alpha}{2}}(2 h |\px|) \psi_x(x, t)\\
& \quad - \Theta_+ \sum_{n = 1}^\infty \frac{c_n}{2n + 1} \px \int_{\R^{2n + 1}} \Tb_n(\etab_n) \hat{\varphi}(\eta_1, t) \hat{\varphi}(\eta_2, t) \dotsm \hat{\varphi}(\eta_{2n + 1}, t) e^{i (\eta_1 + \eta_2 + \dotsb + \eta_{2n + 1}) x} \diff{\etab_n}\\
& \quad + \Theta_- \sum_{n = 1}^\infty \sum_{\ell = 0}^n d_{n, \ell, 0, 1} \px \left\{\big(\varphi(x, t)\big)^{2n - \ell + 1}\right\}\\
& \quad + \Theta_- \sum_{n = 1}^\infty \sum_{\ell = 0}^n \sum_{m = 1}^{2n - \ell + 1} d_{n, \ell, m, 1} \px\Big\{\big(\varphi(x, t)\big)^{2n - \ell + 1 - m} |\px|^{n + \frac{1 - \alpha}{2}} K_{n + \frac{1 - \alpha}{2}}(2 h |\px|) \big(\psi(x, t)\big)^m\Big\} = 0,\\
& \psi_t(x, t) - \frac{\Theta_+ - \Theta_-}{2} B\bigg(\frac{1}{2}, \frac{1 - \alpha}{2}\bigg) \frac{1}{(2h)^{1 - \alpha}} \psi_x(x, t)\stepcounter{equation}\tag{\theequation}\label{gsqgsys01-expd}\\
& \quad + 2 \Theta_- \sin\bigg(\frac{\pi \alpha}{2}\bigg) \Gamma(\alpha - 1) |\px|^{1 - \alpha} \psi_x(x, t) + \Theta_+ \frac{2 \sqrt{\pi}}{\Gamma\left(1 - \frac{\alpha}{2}\right) (4h)^{\frac{1 - \alpha}{2}}} |\px|^{\frac{1 - \alpha}{2}} K_{\frac{1 - \alpha}{2}}(2 h |\px|) \varphi_x(x, t)\\
& \quad - \Theta_- \sum_{n = 1}^\infty \frac{c_n}{2n + 1} \px \int_{\R^{2n + 1}} \Tb_n(\etab_n) \hat{\psi}(\eta_1, t) \hat{\psi}(\eta_2, t) \dotsm \hat{\psi}(\eta_{2n + 1}, t) e^{i (\eta_1 + \eta_2 + \dotsb + \eta_{2n + 1}) x} \diff{\etab_n}\\
& \quad + \Theta_+ \sum_{n = 1}^\infty \sum_{\ell = 0}^n d_{n, \ell, 0, 2} \px \left\{\big(\psi(x, t)\big)^{2n - \ell + 1}\right\}\\
& \quad + \Theta_+ \sum_{n = 1}^\infty \sum_{\ell = 0}^n \sum_{m = 1}^{2n - \ell + 1} d_{n, \ell, m, 2} \px\Big\{\big(\psi(x, t)\big)^{2n - \ell + 1 - m} |\px|^{n + \frac{1 - \alpha}{2}} K_{n + \frac{1 - \alpha}{2}}(2 h |\px|) \big(\varphi(x, t)\big)^m\Big\} = 0.
\end{align*}
When $\alpha = 1$, the expanded regularized two-front SQG system is
\begin{align*}
& \varphi_t(x, t) - (\Theta_+ - \Theta_-) (\gamma + \log{h}) \varphi_x(x, t) - 2 \Theta_+ \log|\px| \varphi_x(x, t) + 2 \Theta_- K_0(2 h |\px|) \psi_x(x, t)\\
& \quad - \Theta_+ \sum_{n = 1}^\infty \frac{c_n}{2n + 1} \px \int_{\R^{2n + 1}} \Tb_n(\etab_n) \hat{\varphi}(\eta_1, t) \hat{\varphi}(\eta_2, t) \dotsm \hat{\varphi}(\eta_{2n + 1}, t) e^{i (\eta_1 + \eta_2 + \dotsb + \eta_{2n + 1}) x} \diff{\etab_n}\\
& \quad + \Theta_- \sum_{n = 1}^\infty \sum_{\ell = 0}^n d_{n, \ell, 0, 1} \px \left\{\big(\varphi(x, t)\big)^{2n - \ell + 1}\right\}\\
& \qquad + \Theta_- \sum_{n = 1}^\infty \sum_{\ell = 0}^n \sum_{m = 1}^{2n - \ell + 1} d_{n, \ell, m, 1} \px\Big\{\big(\varphi(x, t)\big)^{2n - \ell + 1 - m} |\px|^n K_n(2 h |\px|) \big(\psi(x, t)\big)^m\Big\} = 0,
\\
& \psi_t(x, t) + (\Theta_+ - \Theta_-) (\gamma + \log{h}) \psi_x(x, t) \stepcounter{equation}\tag{\theequation}\label{sqgsys-expd1} - 2 \Theta_- \log|\px| \psi_x(x, t) + 2 \Theta_+ K_0(2 h |\px|) \varphi_x(x, t)\\
& \quad - \Theta_- \sum_{n = 1}^\infty \frac{c_n}{2n + 1} \px \int_{\R^{2n + 1}} \Tb_n(\etab_n) \hat{\psi}(\eta_1, t) \hat{\psi}(\eta_2, t) \dotsm \hat{\psi}(\eta_{2n + 1}, t) e^{i (\eta_1 + \eta_2 + \dotsb + \eta_{2n + 1}) x} \diff{\etab_n}\\
& \quad + \Theta_+ \sum_{n = 1}^\infty \sum_{\ell = 0}^n d_{n, \ell, 0, 2} \px \left\{\big(\psi(x, t)\big)^{2n - \ell + 1}\right\}\\
& \qquad + \Theta_+ \sum_{n = 1}^\infty \sum_{\ell = 0}^n \sum_{m = 1}^{2n - \ell + 1} d_{n, \ell, m, 2} \px\Big\{\big(\psi(x, t)\big)^{2n - \ell + 1 - m} |\px|^n K_n(2 h |\px|) \big(\varphi(x, t)\big)^m\Big\} = 0.
\end{align*}

\section{Linearized stability}
\label{sec:lin_stab}

The GSQG equation \eqref{gsqg} has steady shear-flow solutions in which
\[
\theta = \bar\theta(y), \quad \vec{u} = (U(y),0),\qquad
\moddy^{\alpha}U = - \bar{\theta}_y,\quad\moddy = \left(-\partial_y^2\right)^{1/2}.
\]
Functions $\bar{\theta}$ that differ by a constant give the same solutions for $U$, and
distributional solutions for $U\in L^1_{\alpha}(\R) $ are unique up to an additive constant $C$ for $0<\alpha\le 1$, or an additive linear function $Ay + C$ for $1<\alpha \le 2$. We set these homogeneous solutions to zero for definiteness.

A particular example of a shear flow is the
unperturbed two-front solution given by $\varphi = \psi = 0$ and
\[
\bar{\theta}(y) = \begin{cases} \theta_+ & \text{if $y>h_+$,}
\\
\theta_0 & \text{if $h_- < y < h_+$,}
\\
\theta_- & \text{$y < h_-$.}
\end{cases}
\]
Then
\[
\moddy^{\alpha}U = - \frac{\Theta_+}{g_\alpha} \delta(y-h_+) -  \frac{\Theta_-}{g_\alpha}\delta(y-h_-),
\]
where the jumps $\Theta_\pm$ are defined in \eqref{defTheta}. The solution is
\begin{equation}
U(y) = \begin{dcases} - B\big(1 / 2, (1 - \alpha) / 2\big) \left[\Theta_+ |y - h_+|^{\alpha - 1} + \Theta_- |y - h_-|^{\alpha - 1}\right] & \text{if $\alpha \in (0, 1) \cup (1, 2)$},
\\
2 \Theta_+ \log|y - h_+| + 2 \Theta_- \log|y - h_-| & \text{if $\alpha = 1$},
\\
\frac{1}{2}\Theta_+ |y - h_+| + \frac{1}{2}\Theta_- |y - h_-| & \text{if $\alpha = 2$},
\end{dcases}
\label{shear_flow}
\end{equation}
where $B$ is the Beta-function \eqref{betafn}.

For $0<\alpha<2$, this shear-flow solution is the SQG or GSQG analog of the piecewise linear shear flow
that is often considered for the Euler equation with $\alpha =2$ (see Figure~\ref{fig:SQGshear}). The tangential velocity of the
shear flow on the fronts is finite if $1<\alpha \le 2$, but diverges to infinity if
$0<\alpha \le 1$. In addition, $U(y) \to 0$ as $|y|\to \infty$ if $0<\alpha < 1$ or $\Theta_+ + \Theta_- = 0$; otherwise
$|U(y)| \to \infty$ as $|y|\to \infty$.

\begin{figure}[h]
\centering
\begin{subfigure}[]{0.45\textwidth}
\includegraphics[width=\textwidth]{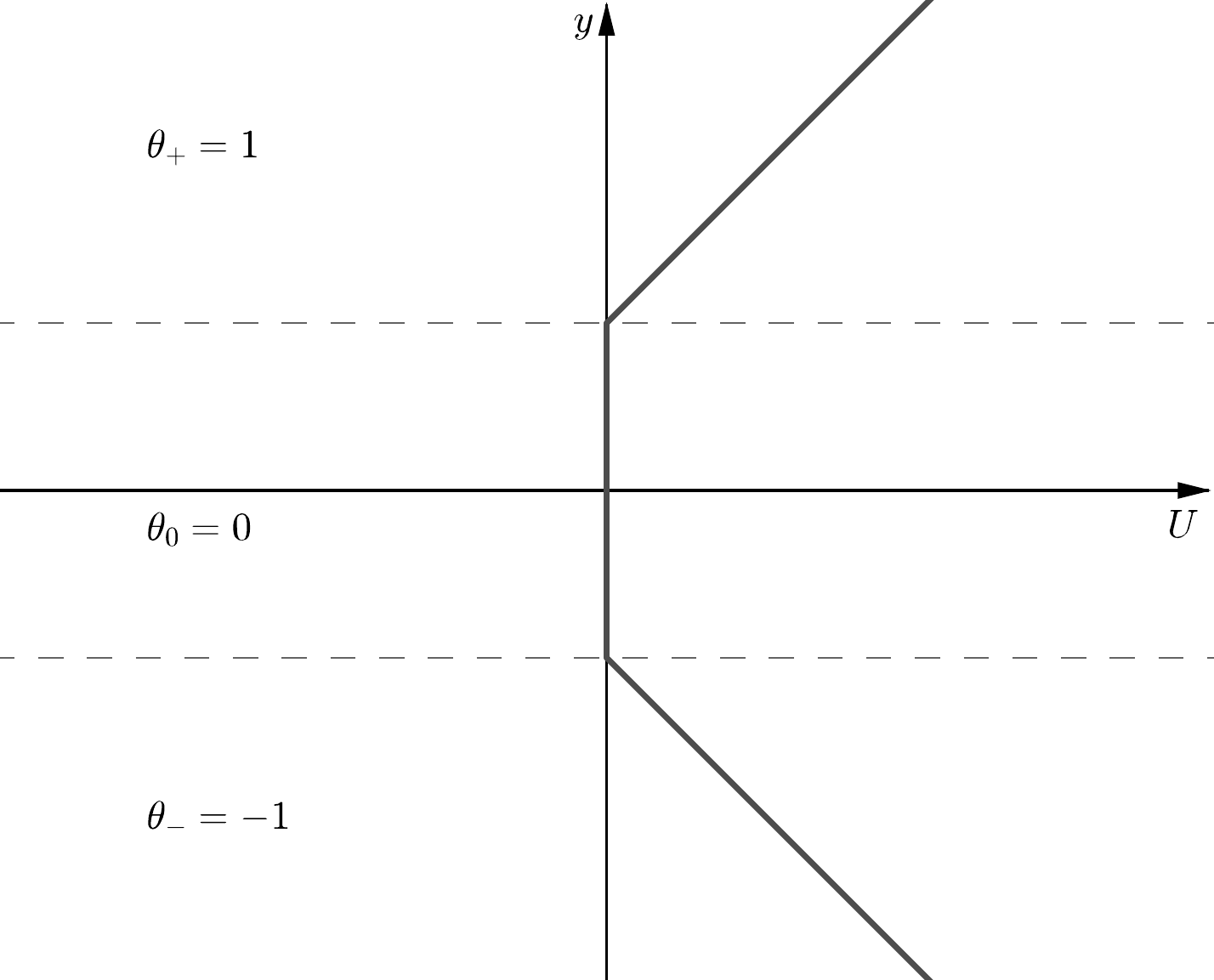}
\caption{Symmetric Euler shear flow.}
\end{subfigure}~
\begin{subfigure}[]{0.45\textwidth}
\includegraphics[width=\textwidth]{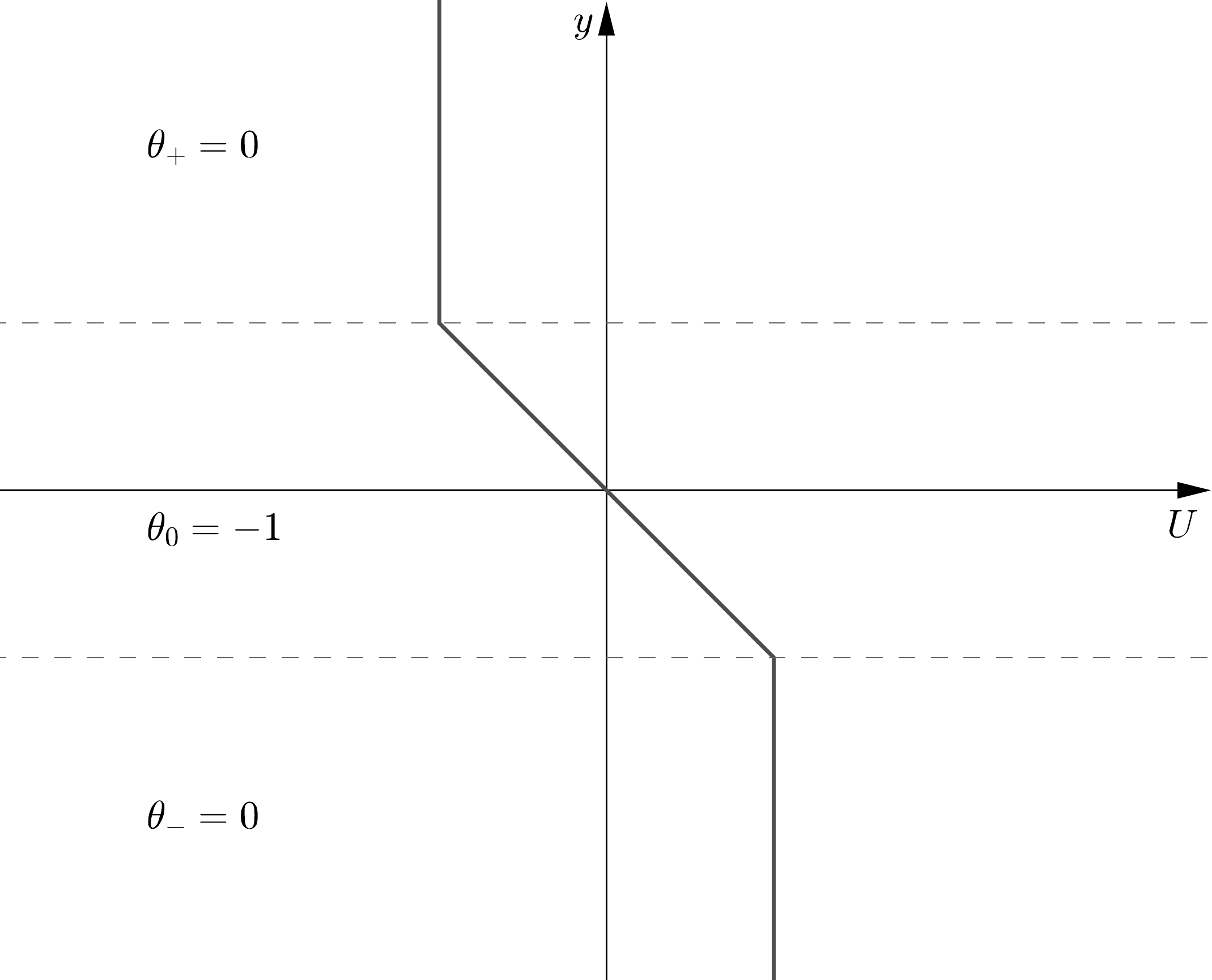}
\caption{Anti-symmetric Euler shear flow.}
\end{subfigure}\\
\begin{subfigure}[]{0.45\textwidth}
\includegraphics[width=\textwidth]{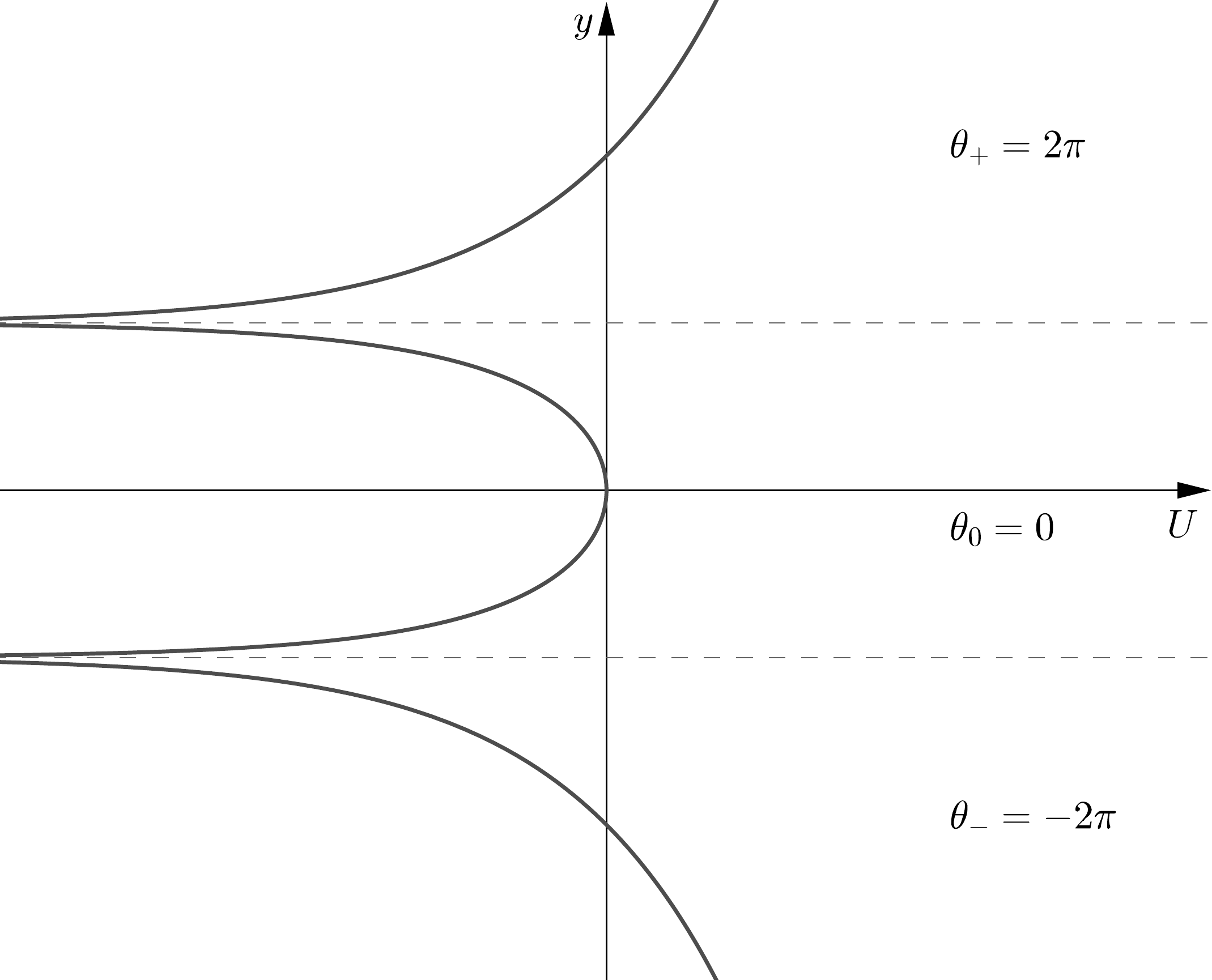}
\caption{Symmetric SQG shear flow.}
\end{subfigure}~
\begin{subfigure}[]{0.45\textwidth}
\includegraphics[width=\textwidth]{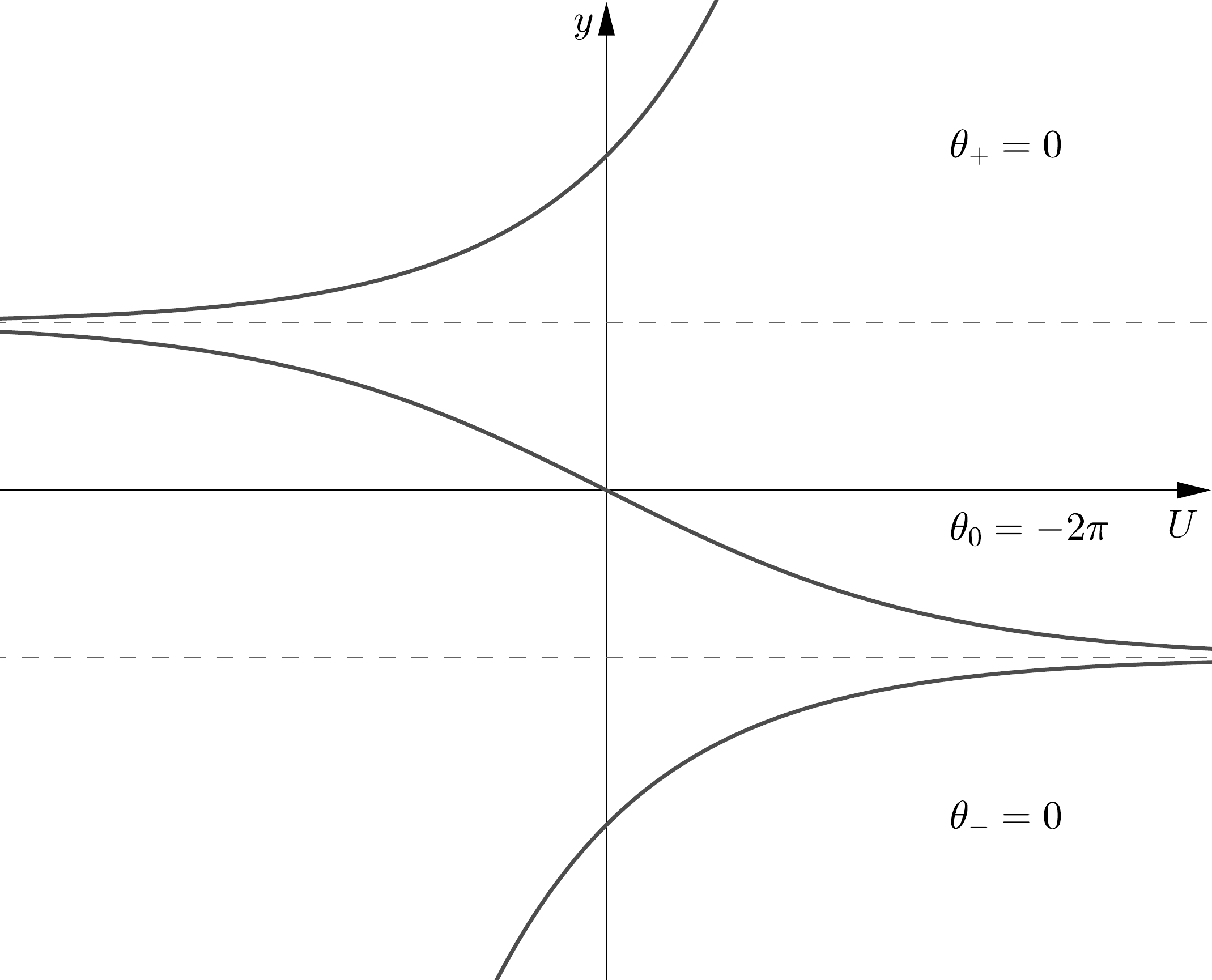}
\caption{Anti-symmetric SQG shear flow.}
\end{subfigure}
\caption{Euler and SQG shear flows. The symmetric flows have scaled jumps $\Theta_+ = \Theta_-=1$, and the anti-symmetric flows
have  $\Theta_+ = -\Theta_-=1$.}
\label{fig:SQGshear}
\end{figure}

There do not appear to be many studies of the stability of SQG and GSQG shear flows $\vec{u} = (U(y),0)$. However, as noted by Friedlander and Shvydkoy \cite{friedlander} for SQG shear flows, the classical necessary conditions for the linearized instability of Euler shear flows --- the Rayleigh and Fj\o rtoft criteria --- carry over directly to sufficiently smooth flows: If there are linear modes with exponential growth in time, then $\moddy^\alpha U$ must change sign,
and for any constant $U_*$, the function
$(U-U_*)\cdot \moddy^\alpha U$ must be strictly positive for some values of $y$.
Conversely, Friedlander and Shvydkoy \cite{friedlander} prove that the SQG shear flow with $U(y) = \sin y$ is linearly unstable.

To study the stability of the two-front GSQG shear flows \eqref{shear_flow} by contour dynamics, we linearize the system \eqref{reg-GSQG-nc}
about $\vp=\psi=0$ to get
\begin{align}
\label{linsys}
\begin{split}
& \vp_t + v \vp_x + \Theta_+ \L_1 \vp_x + \Theta_- \L_2 \psi_x = 0,\qquad
\psi_t - v \psi_x + \Theta_- \L_1 \psi_x + \Theta_+ \L_2 \vp_x  = 0.
\end{split}
\end{align}
Taking the Fourier transform of \eqref{linsys} with respect to $x$, we get the system
\begin{equation}
\partial_t \left(\begin{array}{ll}
\hat\vp\\
\hat\psi
\end{array}\right)=\left(
\begin{array}{cc}
-i\xi(v +\Theta_+b_1(\xi)) & -i\xi \Theta_-b_2(\xi)\\
-i\xi \Theta_+ b_2(\xi) & -i\xi(- v +\Theta_-b_1(\xi))
\end{array}
\right)\left(\begin{array}{ll}
\hat\vp\\
\hat\psi
\end{array}\right),
\label{matrix_stab}
\end{equation}
where the symbols $b_1$, $b_2$ of $\L_1$, $\L_2$ are defined in \eqref{L1Symb}--\eqref{L2Symb}.
The characteristic polynomial (in $\mu$) of the coefficient matrix in \eqref{matrix_stab} is
\begin{align*}
\mu^2+i \xi b_1(\xi) [\Theta_+ + \Theta_-] \mu - \xi^2\left[\Theta_-  b_1(\xi) - v\right]
\left[\Theta_+  b_1(\xi) +  v\right]+ \Theta_+ \Theta_- \xi^2 b_2^2(\xi),
\end{align*}
with roots
\begin{equation}
\mu_\pm(\xi) = \frac{1}{2} \left\{- i \xi b_1(\xi) \left(\Theta_+ + \Theta_- \right) \pm \sqrt{\triangle(\xi)}\right\},
\label{growthrate}
\end{equation}
where the discriminant $\triangle$ is given by
\begin{align*}
\begin{split}
\triangle(\xi)
&= -\biggl[|\xi| b_1(\xi) (\Theta_+ - \Theta_-) + 2v|\xi|\biggr]^2 - 4 \Theta_+ \Theta_- \xi^2 b_2^2(\xi).
\end{split}
\end{align*}
If $\Theta_+ \Theta_- > 0$, then $\triangle(\xi)\leq 0$ for all $\xi\in \R$, so the roots of the characteristic polynomial are imaginary and the
GSQG shear flow is linearly stable. In particular, the symmetric Euler and SQG shear flows shown in Figure~\ref{fig:SQGshear}(a) and Figure~\ref{fig:SQGshear}(c) are linearly stable.

On the other hand, if $\triangle(\xi)>0$ for some $\xi\in \R$, then there is a mode with positive growth rate, and the shear flow is linearly unstable.
For the anti-symmetric Euler shear flow ($\alpha = 2$)  shown in Figure \ref{fig:SQGshear}(b), with
$|\xi| b_1(\xi) = 1/2$, $|\xi|b_2(\xi) = e^{-2h|\xi|}/2$, and $v = -h$ from \eqref{defv}, we get that
\[
\triangle(\xi) = e^{- 4 h |\xi|} - \left(1 - 2 h |\xi|\right)^2,
\]
in agreement with the standard result obtained directly from the Euler equation \cite{vallis}, and there are unstable modes
for $0 < h|\xi| \lesssim 0.63923$.

For the anti-symmetric SQG shear flow ($\alpha = 1$) shown in Figure \ref{fig:SQGshear}(d), we find that
\begin{align}
\label{sqg_discriminant}
\triangle(\xi) = 16 \xi^2 K_0^2(2h |\xi|) - 16 \xi^2 \left[\log(h |\xi|) + \gamma\right]^2,
\end{align}
where $\gamma$ is the Euler-Mascheroni constant.
A numerical plot of the corresponding growth rates and wave speeds is shown in Figure~\ref{fig:sqg_instability}. These plots are qualitatively similar to the ones for the Euler equation.
In both cases, the instability results from an interaction between negative and positive energy waves on the fronts that leads to an exponential growth in time when the horizontal wavelengths of the waves are sufficiently large in comparison with the distance between the fronts.

\begin{figure}[h]
\centering
\begin{subfigure}[]{0.45\textwidth}
\includegraphics[width=\textwidth]{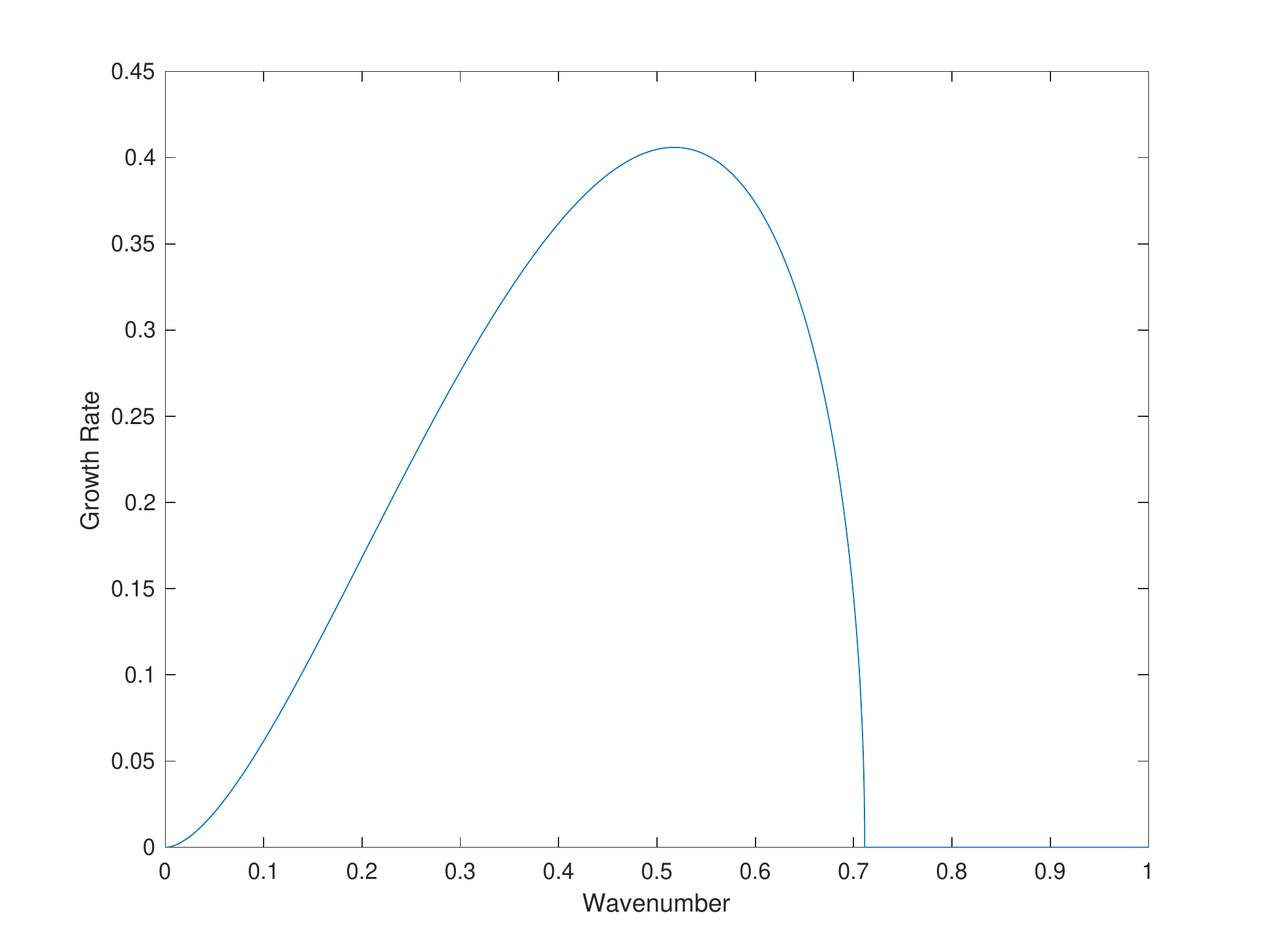}
\end{subfigure}~
\begin{subfigure}[]{0.45\textwidth}
\includegraphics[width=\textwidth]{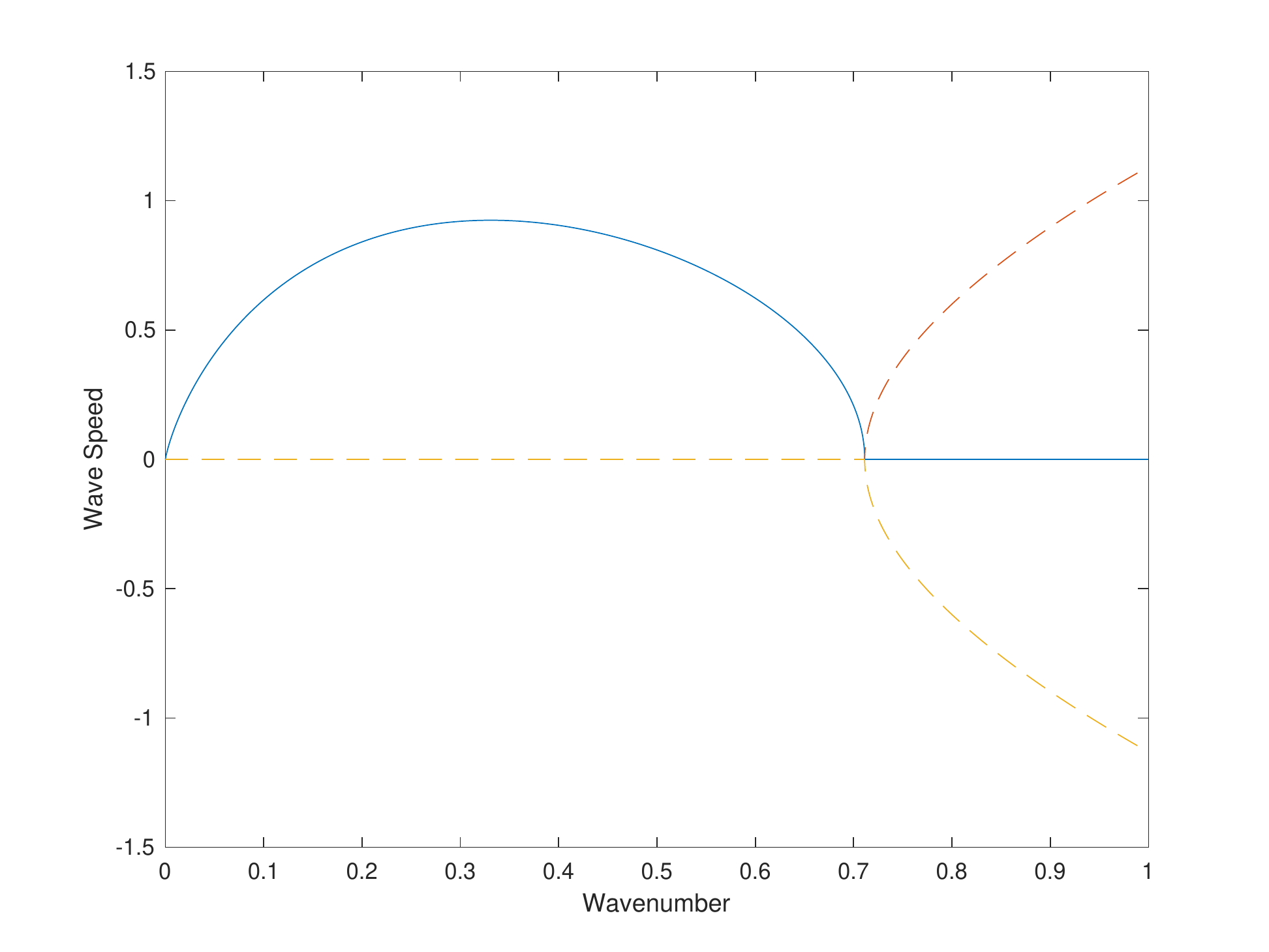}
\end{subfigure}
\caption{Left: Growth rate $\Im \mu$ for anti-symmetric SQG flow in Figure~\ref{fig:SQGshear} with
$\Theta_+ = 1$, $\Theta_- = - 1$, and $h=1$, calculated from \eqref{growthrate} and \eqref{sqg_discriminant}. Right: Real (dashed) and imaginary (solid) wave speeds $c = \mu/\xi$. The flow is unstable for $0 < h|\xi| \lesssim 0.71129$, with the maximum growth rate occurring at $h|\xi| \approx 0.51756$.}
\label{fig:sqg_instability}
\end{figure}

\section{A priori estimates for the two-front GSQG ($0 < \alpha < 1$) systems}\label{sec:lwp01}

\subsection{Para-differential reduction}

Throughout this section and the following ones, we fix $h > 0$ in \eqref{defh} and use $C(n, s)$ to denote a generic constant, which might change from line to line in the proof, depending only on $n$, $s$, and $h$ that grows no faster than exponentially in $n$.

First, we state an estimate which shows that we can distribute derivatives on each factor of $\vp$ in the multilinear terms in the front
equations. This estimate is not sharp, but it is sufficient for our needs below.
\begin{lemma}
\label{Tn-roughbdd01}
Let $\Tb_n$ be defined by \eqref{defTn} for $n \in \N$. Then
\[
|\Tb_n(\etab_n)| \le \frac{2^{1 + \alpha}}{\alpha} \prod_{j = 1}^{2n + 1} |\eta_j| + \frac{2^{1 + \alpha}}{3 - \alpha}
\qquad\text{for all}\ \etab_n \in \R^{2n + 1}.
\]
\end{lemma}

\begin{proof}
Splitting up the integral and using a Taylor expansion, we have
\begin{align*}
|\Tb_n(\etab_n)| &\le \int_\R \frac{\prod_{j = 1}^{2n + 1} |1 - e^{ i \eta_j \zeta}|}{|\zeta|^{2n + 2 - \alpha}} \diff{\zeta}
\le  \int_{|\zeta| < 2}\prod_{j = 1}^{2n+1} |\eta_j | \frac{\diff{\zeta}}{|\zeta|^{1 - \alpha}}
+  \int_{|\zeta| > 2}\frac{2^{2n + 1}}{|\zeta|^{2n + 2 - \alpha}} \diff{\zeta}
\le \frac{2^{1 + \alpha}}{\alpha} \prod_{j = 1}^{2n + 1} |\eta_j| + \frac{2^{1 + \alpha}}{3 - \alpha}.
\end{align*}
\end{proof}

We prove the following propositions which allows us to write \eqref{gsqgsys01-expd} in a form that is suitable for constructing weighted energies without a loss of derivatives.
\begin{proposition}
\label{varphinonlindecomp01}
For $0 < \alpha < 1$, suppose that $\varphi(\cdot, t), \psi(\cdot, t) \in H^s(\R)$ with $s \geq 5$ and $\|\varphi\|_{W^{4, \infty}} + \|\psi\|_{W^{4, \infty}}$ is sufficiently small. Then
\begin{align*}
& -  \sum_{n = 1}^\infty \frac{c_n}{2n + 1} \px \int_{\R^{2n + 1}} \Tb_n(\etab_n) \hat{\varphi}(\eta_1, t) \hat{\varphi}(\eta_2, t) \dotsm \hat{\varphi}(\eta_{2n + 1}, t) e^{i (\eta_1 + \eta_2 + \dotsb + \eta_{2n + 1}) x} \diff{\etab_n}\\
&= \px |\px|^{1 - \alpha} \Big\{T_{B^{1 - \alpha}[\varphi]} \varphi(x, t) + D^{-1} T_{B^{- \alpha}[\varphi]} \varphi(x, t)\Big\} + \px T_{B^0[\varphi]} \varphi(x, t) + \Rc_1,
\end{align*}
where the symbols $B^{1 - \alpha}[\varphi]$, $B^0[\varphi]$, and $B^{-\alpha}[\varphi]$ are defined by
\begin{align}
\label{defB01}
\begin{split}
B^{1 - \alpha}[\varphi](\cdot, \xi) &=  \sum_{n = 1}^\infty B^{1 - \alpha}_n[\varphi](\cdot, \xi), \quad B^0[\varphi](\cdot, \xi) =  \sum_{n = 1}^\infty B^0_n[\varphi](\cdot, \xi),
\\
B^{-\alpha}[\varphi](\cdot, \xi) &=  \sum_{n = 1}^\infty B^{-\alpha}_n[\varphi](\cdot, \xi) = \frac{1 - \alpha}{2 i} \px B^{1 - \alpha}[\vp](\cdot, \xi),
\end{split}
\end{align}
with
\begin{align*}
B^{1 - \alpha}_n[\varphi](\cdot, \xi) &= -2 c_n \Gamma(\alpha - 1) \sin\bigg(\frac{\pi \alpha}{2}\bigg) \F_{\zeta}^{-1}\Bigg[\int_{\R^{2n}} \delta\bigg(\zeta - \sum_{j = 1}^{2n} \eta_j\bigg) \prod_{j = 1}^{2n} \bigg(i \eta_j \hat{\varphi}(\eta_j) \chi\Big(\frac{(2n + 1) \eta_j}{\xi}\Big)\bigg) \diff{\hat{\etab}_n}\Bigg],\\
B^0_n[\varphi](\cdot, \xi) &= 2 c_n \Gamma(\alpha - 1) \sin\bigg(\frac{\pi \alpha}{2}\bigg) \F_{\zeta}^{-1}\Bigg[\int_{\R^{2n}} \delta\bigg(\zeta - \sum_{j = 1}^{2n} \eta_j\bigg) \prod_{j = 1}^{2n} \bigg(i \eta_j \hat{\varphi}(\eta_j) \chi\Big(\frac{(2n + 1) \eta_j}{\xi}\Big)\bigg)\\*
& \hspace{3in} \cdot \int_{[0, 1]^{2n}} \bigg|\sum_{j = 1}^{2n} \eta_j s_j\bigg|^{1 - \alpha} \diff{\hat{\s}_n} \diff{\hat{\etab}_n}\Bigg],\\
B^{-\alpha}_n[\varphi](\cdot, \xi) &= -2 c_n (1 - \alpha) \Gamma(\alpha - 1) \sin\bigg(\frac{\pi \alpha}{2}\bigg) \F_{\zeta}^{-1}\Bigg[\int_{\R^{2n}} \delta\bigg(\zeta - \sum_{j = 1}^{2n} \eta_j\bigg) \prod_{j = 1}^{2n} \bigg(i \eta_j \hat{\varphi}(\eta_j) \chi\Big(\frac{(2n + 1) \eta_j}{\xi}\Big)\bigg)\\
& \hspace{3in} \cdot \int_{[0, 1]^{2n}} \sum_{j = 1}^{2n} \eta_j s_j \diff{\hat{\s}_n} \diff{\hat{\etab}_n}\Bigg]\\
& = \frac{1 - \alpha}{2 i} \px B_n^{1 - \alpha}[\vp](\cdot, \xi).
\end{align*}
Here $\chi$ is the cutoff function in \eqref{weyldef}, $\hat{\etab}_n = (\eta_1, \eta_2, \dotsc, \eta_{2n})$ and $\hat{\s}_n = (s_1, s_2, \dotsc, s_{2n})$. The operators $T_{B^{1 - \alpha}[\varphi]}$ and $T_{B^0[\varphi]}$ are self-adjoint and they satisfy the estimates
\begin{align}
\label{B-est01}
\begin{split}
\|B^{1 - \alpha}[\vp]\|_{\M_{(3, 5)}} \lesssim \sum_{n = 1}^\infty C(n, s) |c_n| \|\vp\|_{W^{4, \infty}}^{2n},\\
\|B^0[\vp]\|_{\M_{(3, 5)}} \lesssim \sum_{n = 1}^\infty C(n, s) |c_n| \|\vp\|_{W^{4, \infty}}^{2n}.
\end{split}
\end{align}
while the remainder term $\Rc_1$ satisfies
\begin{align}
\label{R1est01}
\|\Rc_1\|_{H^s} \lesssim \|\vp\|_{H^s} \sum_{n = 1}^\infty C(n, s) |c_n| \|\vp\|_{W^{4, \infty}}^{2n}.
\end{align}
Similar conclusions hold for $\psi$.
\end{proposition}
\begin{proof}
In this proof, we suppress the time variable for simplicity.

By symmetry, we may assume that $|\eta_{2n + 1}|$ is the largest frequency. Then
\begin{align}
\label{intprod01}
\begin{split}
& - \frac{c_n}{2n + 1} \px \int_{\R^{2n + 1}} \Tb_n(\etab_n) \hat{\varphi}(\eta_1) \hat{\varphi}(\eta_2) \dotsm \hat{\varphi}(\eta_{2n + 1}) e^{i (\eta_1 + \eta_2 + \dotsb + \eta_{2n + 1}) x} \diff{\etab_n}\\
= ~& - c_n \px \int _\R \int\limits_{\substack{|\eta_j| \leq |\eta_{2n + 1}|\\ \text{for all}\ j = 1, 2, \dotsc, 2n}} \Tb_n(\etab_n) \hat{\varphi}(\eta_1) \hat{\varphi}(\eta_2) \dotsm \hat{\varphi}(\eta_{2n}) e^{i (\eta_1 + \eta_2 + \dotsb + \eta_{2n}) x} \diff{\hat{\etab}_n} \hat{\varphi}(\eta_{2n + 1}) e^{i \eta_{2n + 1} x} \diff{\eta_{2n + 1}}\\
= ~& - c_n \px \int_\R \int\limits_{\substack{|\eta_j| \leq |\eta_{2n + 1}|\\ \text{for all}\ j = 1, 2, \dotsc, 2n}} \Tb_n(\etab_n) \prod_{j = 1}^{2n} \bigg[\chi\bigg(\frac{(2n + 1) \eta_j}{\eta_{2n + 1}}\bigg) + 1 - \chi\bigg(\frac{(2n + 1) \eta_j}{\eta_{2n + 1}}\bigg)\bigg] \hat{\varphi}(\eta_j)\\
& \hspace{2.5in} \cdot e^{i (\eta_1 + \eta_2 + \dotsb + \eta_{2n}) x} \diff{\hat{\etab}_n} \hat{\varphi}(\eta_{2n + 1}) e^{i \eta_{2n + 1} x} \diff{\eta_{2n + 1}}.
\end{split}
\end{align}
We expand the product in \eqref{intprod01} into terms of the form $\chi^{2n-\ell} (1-\chi)^{\ell}$ and consider two cases.

{\bf Case \Rm{1}.} When we take only factors of $\chi$ in the expansion of the product, we  get the term
\begin{align}
\label{high-low01}
- c_n \px \int_\R \int\limits_{\substack{|\eta_j| \leq |\eta_{2n + 1}|\\ \text{for all}\ j = 1, 2, \dotsc, 2n}} \Tb_n(\etab_n) \prod_{j = 1}^{2n} \chi\bigg(\frac{(2n + 1) \eta_j}{\eta_{2n + 1}}\bigg) \hat{\varphi}(\eta_j) \cdot e^{i (\eta_1 + \eta_2 + \dotsb + \eta_{2n}) x} \diff{\hat{\etab}_n} \hat{\varphi}(\eta_{2n + 1}) e^{i \eta_{2n + 1} x} \diff{\eta_{2n + 1}}.
\end{align}
In the following, we write $\s_n = (s_1, s_2, \dotsc, s_{2n + 1})$.

By \eqref{defTn}, we can write
\begin{align*}
\Tb_n(\etab_n) &= - \int_\R |\zeta|^{\alpha - 1} \sgn{\zeta} \int_{[0, 1]^{2n + 1}} \prod_{j = 1}^{2n + 1} i \eta_j e^{i \eta_j s_j \zeta} \diff{\s_n} \diff{\zeta}\\
&= - 2 i \Gamma(\alpha) \sin\bigg(\frac{\pi \alpha}{2}\bigg) \prod_{j = 1}^{2n + 1} (i \eta_j) \int_{[0, 1]^{2n + 1}} \bigg|\sum_{j = 1}^{2n + 1} \eta_j s_j\bigg|^{- \alpha} \sgn\bigg(\sum_{j = 1}^{2n + 1} \eta_j s_j\bigg) \diff{\s_n}\\
&= -2 \Gamma(\alpha - 1) \sin\bigg(\frac{\pi \alpha}{2}\bigg) |\eta_{2n + 1}|^{1 - \alpha} \prod_{j = 1}^{2n} (i \eta_j) \int_{[0, 1]^{2n}} \bigg|1 + \sum_{j = 1}^{2n} \frac{\eta_j}{\eta_{2n + 1}} s_j\bigg|^{1 - \alpha} - \bigg|\sum_{j = 1}^{2n} \frac{\eta_j}{\eta_{2n + 1}} s_j\bigg|^{1 - \alpha} \diff{\hat{\s}_n}.
\end{align*}
Substitution of this expression into \eqref{high-low01} yields the following terms
\begin{align}
& - c_n \px \int_\R \int\limits_{\substack{|\eta_j| \leq |\eta_{2n + 1}|\\ \text{for all}\ j = 1, 2, \dotsc, 2n}} \Tb_n^{1 - \alpha}(\etab_n) \prod_{j = 1}^{2n} \chi\bigg(\frac{(2n + 1) \eta_j}{\eta_{2n + 1}}\bigg) \hat{\varphi}(\eta_j) \cdot e^{i (\eta_1 + \eta_2 + \dotsb + \eta_{2n}) x} \diff{\hat{\etab}_n} \hat{\varphi}(\eta_{2n + 1}) e^{i \eta_{2n + 1} x} \diff{\eta_{2n + 1}},\label{decompTn01-1}\\
& - c_n \px \int_\R \int\limits_{\substack{|\eta_j| \leq |\eta_{2n + 1}|\\ \text{for all}\ j = 1, 2, \dotsc, 2n}} \Tb_n^{0}(\etab_n) \prod_{j = 1}^{2n} \chi\bigg(\frac{(2n + 1) \eta_j}{\eta_{2n + 1}}\bigg) \hat{\varphi}(\eta_j) \cdot e^{i (\eta_1 + \eta_2 + \dotsb + \eta_{2n}) x} \diff{\hat{\etab}_n} \hat{\varphi}(\eta_{2n + 1}) e^{i \eta_{2n + 1} x} \diff{\eta_{2n + 1}},\label{decompTn01-2}\\
& - c_n \px \int_\R \int\limits_{\substack{|\eta_j| \leq |\eta_{2n + 1}|\\ \text{for all}\ j = 1, 2, \dotsc, 2n}} \Tb_n^{- \alpha}(\etab_n) \prod_{j = 1}^{2n} \chi\bigg(\frac{(2n + 1) \eta_j}{\eta_{2n + 1}}\bigg) \hat{\varphi}(\eta_j) \cdot e^{i (\eta_1 + \eta_2 + \dotsb + \eta_{2n}) x} \diff{\hat{\etab}_n} \hat{\varphi}(\eta_{2n + 1}) e^{i \eta_{2n + 1} x} \diff{\eta_{2n + 1}},\label{decompTn01-3}\\
& - c_n \px \int_\R \int\limits_{\substack{|\eta_j| \leq |\eta_{2n + 1}|\\ \text{for all}\ j = 1, 2, \dotsc, 2n}} \Tb_n^{\leq -1}(\etab_n) \prod_{j = 1}^{2n} \chi\bigg(\frac{(2n + 1) \eta_j}{\eta_{2n + 1}}\bigg) \hat{\varphi}(\eta_j) \cdot e^{i (\eta_1 + \eta_2 + \dotsb + \eta_{2n}) x} \diff{\hat{\etab}_n} \hat{\varphi}(\eta_{2n + 1}) e^{i \eta_{2n + 1} x} \diff{\eta_{2n + 1}},\label{decompTn01-4}
\end{align}
where
\begin{align*}
\Tb_n^{1 - \alpha}(\etab_n) &= -2 \Gamma(\alpha - 1) \sin\bigg(\frac{\pi \alpha}{2}\bigg) |\eta_{2n + 1}|^{1 - \alpha} \prod_{j = 1}^{2n} (i \eta_j),\\
\Tb_n^{0}(\etab_n) &= 2 \Gamma(\alpha - 1) \sin\bigg(\frac{\pi \alpha}{2}\bigg) \prod_{j = 1}^{2n} (i \eta_j) \int_{[0, 1]^{2n}} \bigg|\sum_{j = 1}^{2n} \eta_j s_j\bigg|^{1 - \alpha} \diff{\hat{\s}_n},\\
\Tb_n^{- \alpha}(\etab_n) &= -2 \Gamma(\alpha) \sin\bigg(\frac{\pi \alpha}{2}\bigg) |\eta_{2n + 1}|^{-\alpha} (\sgn{\eta_{2n + 1}}) \prod_{j = 1}^{2n} (i \eta_j) \int_{[0, 1]^{2n}} \sum_{j = 1}^{2n} \eta_j s_j \diff{\hat{\s}_n},\\
\Tb_n^{\leq -1}(\etab_n) &= - 2 \Gamma(\alpha - 1) \sin\bigg(\frac{\pi \alpha}{2}\bigg) |\eta_{2n + 1}|^{1 - \alpha} \prod_{j = 1}^{2n} (i \eta_j)\\
& \qquad \cdot \int_{[0, 1]^{2n}} \bigg\{\bigg|1 + \sum_{j = 1}^{2n} \frac{\eta_j}{\eta_{2n + 1}} s_j\bigg|^{1 - \alpha} - 1 - (1 - \alpha) \sum_{j = 1}^{2n} \frac{\eta_j}{\eta_{2n + 1}} s_j\bigg\} \diff{\hat{\s}_n}.
\end{align*}

We claim that \eqref{decompTn01-1}, \eqref{decompTn01-2}, and \eqref{decompTn01-3} can be written as
\begin{align}
\label{TBdecomp01}
\px |\px|^{1 - \alpha} T_{B^{1 - \alpha}[\varphi]} \varphi + \Rc_{1, 1}, \quad \px T_{B^0[\varphi]} \varphi + \Rc_{1, 2}, \quad \text{and} \quad \px T_{B^{-\alpha}[\varphi]} \varphi + \Rc_{1, 3},
\end{align}
where $\Rc_{1, 1}$, $\Rc_{1, 2}$, and $\Rc_{1, 3}$ satisfy the estimate \eqref{R1est01}. Indeed,
\begin{align*}
\F\left[\px T_{B^{1 - \alpha}_n[\varphi]} \big(|\px|^{1 - \alpha} \varphi\big)\right](\xi) &= - \frac{c_n}{\pi} \Gamma(\alpha - 1) \sin\bigg(\frac{\pi \alpha}{2}\bigg) i \xi \int_\R \chi\bigg(\frac{|\xi - \eta|}{|\xi + \eta|}\bigg) |\eta|^{1 - \alpha}\\
& \qquad \int_{\R^{2n}} \delta\bigg(\xi - \eta - \sum_{j = 1}^{2n} \eta_j\bigg) \prod_{j = 1}^{2n} \bigg(i \eta_j \hat{\varphi}(\eta_j) \chi\Big(\frac{2 (2n + 1) \eta_j}{\xi + \eta}\Big)\bigg) \diff{\hat{\etab}_n} \hat{\varphi}(\eta) \diff{\eta},
\end{align*}
while the Fourier transform of \eqref{decompTn01-1} is
\begin{align*}
& -\frac{c_n}{\pi} \Gamma(\alpha - 1) \sin\bigg(\frac{\pi \alpha}{2}\bigg) i \xi \int_\R \int\limits_{\substack{|\eta_j| \leq |\eta_{2n + 1}|\\ \text{for all}\ j = 1, 2, \dotsc, 2n}} \delta\bigg(\xi - \sum_{j = 1}^{2n + 1} \eta_j\bigg) |\eta_{2n + 1}|^{1 - \alpha}\\
& \hspace{1in} \cdot \prod_{j = 1}^{2n} \bigg(\chi\Big(\frac{(2n + 1) \eta_j}{\eta_{2n + 1}}\Big) (i \eta_j) \hat{\varphi}(\eta_j)\bigg) \diff{\hat{\etab}_n} \hat{\varphi}(\eta_{2n + 1}) \diff{\eta_{2n + 1}}.
\end{align*}
The difference of the above two integrals is
\begin{align}
\label{error01}
\begin{split}
- \frac{c_n}{\pi} & \Gamma(\alpha - 1) \sin\bigg(\frac{\pi \alpha}{2}\bigg) i \xi \int_{\R^{2n + 1}} \delta\bigg(\xi - \sum_{j = 1}^{2n + 1} \eta_j\bigg) |\eta_{2n + 1}|^{1 - \alpha} \\
& \cdot \Bigg[\chi\bigg(\frac{|\xi - \eta_{2n + 1}|}{|\xi + \eta_{2n + 1}|}\bigg) \prod_{j = 1}^{2n} \bigg(i \eta_j \hat{\varphi}(\eta_j) \chi\Big(\frac{2 (2n + 1) \eta_j}{\xi + \eta_{2n + 1}}\Big)\bigg)\\
& \quad - \mathbb{I}_n \prod_{j = 1}^{2n} \bigg(i \eta_j \hat{\varphi}(\eta_j) \chi\Big(\frac{(2n + 1) \eta_j}{\eta_{2n + 1}}\Big)\bigg)\Bigg] \diff{\hat{\etab}_n} \hat{\varphi}(\eta_{2n + 1}) \diff{\eta_{2n + 1}},
\end{split}
\end{align}
where $\mathbb{I}_n$ is the function which is equal to $1$ on $\{|\eta_j| \leq |\eta_{2n + 1}|,\ \text{for all}\ 1, \dotsc, 2n\}$ and equal to zero otherwise.

When $\etab_n$ satisfies
\begin{align}
\label{etancon01}
|\eta_j| \leq \frac{1}{40} \frac{1}{2 n + 1} |\eta_{2 n + 1}| \qquad \text{for all}\ j = 1, 2, \dotsc, 2n,
\end{align}
we have $\mathbb{I}_n = 1$ and $\chi \left(\frac{(2n + 1) \eta_j}{\eta_{2n + 1}}\right) = 1$. In addition, since $\xi = \sum_{j = 1}^{2n + 1} \eta_j$, we have
\begin{align*}
& \frac{|\xi - \eta_{2n + 1}|}{|\xi + \eta_{2n + 1}|} = \frac{\left|\sum_{j = 1}^{2n} \eta_j\right|}{\left|\sum_{j = 1}^{2n} \eta_j + 2 \eta_{2n + 1}\right|} \leq \frac{\frac{1}{40} |\eta_{2n + 1}|}{(2 - \frac{1}{40}) |\eta_{2n + 1}|} = \frac{1}{79} < \frac{3}{40},\\
& \frac{2 (2n + 1) |\eta_j|}{|\xi + \eta_{2n + 1}|} \leq \frac{\frac{1}{20} |\eta_{2n + 1}|}{(2 - \frac{1}{40}) |\eta_{2n + 1}|} = \frac{2}{79} < \frac{3}{40}.
\end{align*}

Therefore, the integrand of \eqref{error01} is supported outside of the set \eqref{etancon01}, and there exists $j_1 \in \{1, \dotsc, 2n\}$, such that $|\eta_{2n + 1}| \geq |\eta_{j_ 1}| > \frac{1}{40} \frac{1}{2n + 1} |\eta_{2n + 1}|$. It follows from this comparability of $|\eta_{j_1}|$ and $|\eta_{2n + 1}|$ that the $H^s$-norm of the error term \eqref{error01} can be bounded by
\[
\|\vp\|_{H^s} C(n, s) |c_n| \|\vp\|_{W^{4, \infty}}^{2n},
\]
so \eqref{decompTn01-1} can be written as in \eqref{TBdecomp01}. Similar calculations apply to \eqref{decompTn01-2} and \eqref{decompTn01-3}.

The symbols $B^{1 - \alpha}_n[\varphi]$ and $B^0_n[\varphi]$ are real, so that $T_{B^{1 - \alpha}_n[\varphi]}$ and $T_{B^0_n[\varphi]}$ are self-adjoint. On the contrary, the symbol $B^{- \alpha}[\varphi]$ is purely imaginary. Again, without loss of generality, we assume $|\eta_{2n}| = \max_{1 \leq j \leq 2n} |\eta_j|$ and observe that
\begin{align*}
\int_{[0, 1]^{2n}} \bigg|\sum_{j = 1}^{2n} \eta_j s_j\bigg|^{1 - \alpha} \diff{\hat{\s}_n} = |\eta_{2n}|^{1 - \alpha} + O(1).
\end{align*}
Thus, using Young's inequality, we obtain the symbol estimates \eqref{B-est01}.

To estimate \eqref{decompTn01-4}, we observe that on the support of the functions $\chi\left(\frac{(2n + 1) \eta_j}{\eta_{2n + 1}}\right)$, we have
\[
\frac{|\eta_j|}{|\eta_{2n + 1}|} \leq \frac{1}{10 (2n + 1)}.
\]
Since $s_j \in [0, 1]$, a Taylor expansion gives
\[
\left|\Tb_n^{\leq -1}(\etab_n)\right| \lesssim \frac{\left[\prod_{j = 1}^{2n} |\eta_j|\right] \left[\sum_{j = 1}^{2n} |\eta_j|\right]}{|\eta_{2n + 1}|^{1 + \alpha}}.
\]
Therefore, the $H^s$-norm of \eqref{decompTn01-4} is bounded by $C(n, s) |c_n| \|\vp\|_{H^s} \|\vp\|_{W^{4,\infty}}^{2n}$.

{\bf Case \Rm{2}.} When there is at least one factor of the form $1 - \chi$ in the expansion of the product in the integral \eqref{intprod01}, we get a term like
\begin{align}
\label{error'01}
\begin{split}
& c_n \px \int_\R \int\limits_{\substack{|\eta_j| \leq |\eta_{2n + 1}|\\ \text{for all}\ j = 1, 2, \dotsc, 2n}} \Tb_n(\etab_n) \prod_{k = 1}^\ell \bigg[1 - \chi\bigg(\frac{(2n + 1) \eta_{j_k}}{\eta_{2n + 1}}\bigg)\bigg] \prod_{k = \ell + 1}^{2n} \chi\bigg(\frac{(2n + 1) \eta_{j_k}}{\eta_{2n + 1}}\bigg)\\
& \hspace{2.5in} \cdot \prod_{j = 1}^{2n} \hat{\varphi}(\eta_j) e^{i (\eta_1 + \eta_2 + \dotsb + \eta_{2n}) x} \diff{\hat{\etab}_n} \hat{\varphi}(\eta_{2n + 1}) e^{i \eta_{2n + 1} x} \diff{\eta_{2n + 1}},
\end{split}
\end{align}
where $1 \leq \ell \leq 2n$ is an integer, and $\{j_k : k = 1, \dotsc, 2n\}$ is a permutation of $\{1, \dotsc, 2n\}$.

The function $1 - \chi\left(\frac{(2n + 1) \eta_{j_1}}{\eta_{2n + 1}}\right)$ is compactly supported on
\[
\frac{|\eta_{j_1}|}{|\eta_{2n + 1}|}\geq \frac{3}{40(2n + 1)}.
\]
By assumption, $\eta_{2n + 1}$ has the largest absolute value, so
\[
\frac{3}{40(2n + 1)} |\eta_{2n + 1}| \leq |\eta_{j_1}|\leq |\eta_{2n + 1}|,
\]
meaning that the frequencies $|\eta_{j_1}|$ and $|\eta_{2n + 1}|$ are comparable. Using Lemma \ref{Tn-roughbdd01}, we can bound the $H^s$-norm of \eqref{error'01} by
\[
\|\vp\|_{H^s} \bigg(\sum_{n = 1}^\infty  C(n, s) |c_n| \|\varphi\|_{W^{4, \infty}}^{2n}\bigg),
\]
and the proposition follows.

\end{proof}

\begin{proposition}
\label{varphipsinonlindecomp01}
For $0 < \alpha \leq 1$, suppose that $\varphi(\cdot, t), \psi(\cdot, t) \in H^s(\R)$ with $s \geq 5$ and $\|\varphi\|_{W^{4, \infty}} + \|\psi\|_{W^{4, \infty}}$ is sufficiently small. Then we can write
\begin{align*}
 \sum_{n = 1}^\infty \sum_{\ell = 0}^n \sum_{m = 1}^{2n - \ell + 1} d_{n, \ell, m, 1} \px\left\{\big(\varphi(x, t)\big)^{2n - \ell + 1 - m} |\px|^{n + \frac{1 - \alpha}{2}} K_{n + \frac{1 - \alpha}{2}}(2 h |\px|)\big(\psi(x, t)\big)^m\right\} = T_{\Bf_1[\vp, \psi]} \vp_x + \Rc_2,
\end{align*}
where
\begin{align}
\label{Bfdef01}
\begin{split}
\Bf_1[\vp, \psi] &= \sum_{n = 1}^\infty \sum_{\ell = 0}^n \sum_{m = 1}^{2n - \ell} d_{n, \ell, m, 1} (2n - \ell + 1 - m) \Bf_{1, n, \ell, m}[\vp, \psi],\\
\Bf_{1, n, \ell, m}[\vp, \psi] &= \varphi^{2n - \ell - m} |\px|^{n + \frac{1 - \alpha}{2}} K_{n + \frac{1 - \alpha}{2}}(2 h |\px|)\psi^m.
\end{split}
\end{align}
The symbol $\Bf_1[\vp, \psi]$ and remainder $\Rc_2$ satisfy symbol estimates
\begin{align}
\label{BfRc2est}
\begin{split}
\|\Bf_1[\vp, \psi]\|_{\M_{(3, 5)}} & \lesssim \sum_{n = 1}^\infty \sum_{\ell = 0}^n \sum_{m = 1}^{2n - \ell} C(n, s) h^{\ell - 2n - (1 - \alpha)} \|\vp\|_{W^{4, \infty}}^{2n - \ell - m} \|\psi\|_{W^{4, \infty}}^m,\\
\|\Rc_2\|_{H^s} &\lesssim \|\vp\|_{H^s} \sum_{n = 1}^\infty \sum_{\ell = 0}^n \sum_{m = 1}^{2n - \ell} C(n, s) h^{\ell - 2n - (1 - \alpha)} \|\psi\|_{W^{4, \infty}}^m \|\vp\|_{W^{4, \infty}}^{2n - \ell - m}\\
& \qquad + \|\psi\|_{H^s} \sum_{n = 1}^\infty \sum_{\ell = 0}^n \sum_{m = 1}^{2n - \ell} C(n, s) h^{\ell - 2n - (2 - \alpha)} \|\psi\|_{W^{4, \infty}}^{m - 1} \|\vp\|_{W^{4, \infty}}^{2n - \ell + 1 - m}\\
&\qquad + \|\psi\|_{H^s} \sum_{n = 1}^\infty \sum_{\ell = 0}^n C(n, s) h^{\ell - 2n + \alpha - \frac{5}{2}} \|\psi\|_{W^{4, \infty}}^{2n - \ell}.
\end{split}
\end{align}
A similar result holds with $\vp$ and $\psi$ exchanged.
\end{proposition}

\begin{proof}
We suppress the dependence of variables of $\varphi$ and $\psi$ for simplicity. By the product rule and Bony's decomposition, we see that for $m < 2n - \ell + 1$,
\begin{align*}
& \px\left\{\varphi^{2n - \ell + 1 - m} |\px|^{n + \frac{1 - \alpha}{2}} K_{n + \frac{1 - \alpha}{2}}(2 h |\px|)\psi^m\right\}\\
= ~ & \left[(2n - \ell + 1 - m) \varphi^{2n - \ell - m} |\px|^{n + \frac{1 - \alpha}{2}} K_{n + \frac{1 - \alpha}{2}}(2 h |\px|)\psi^m\right] \varphi_x(x, t)\\
& \qquad + \varphi^{2n - \ell + 1 - m} \px |\px|^{n + \frac{1 - \alpha}{2}} K_{n + \frac{1 - \alpha}{2}}(2 h |\px|)\psi^m\\
= ~ & (2n - \ell + 1 -m) T_{\Bf_{1, n, \ell, m}[\vp, \psi]} \vp_x + \Rc_{2, n, \ell, m},
\end{align*}
where $\Bf_{1, n, \ell, m}[\vp, \psi]$ is defined as in \eqref{Bfdef01}, and, by Lemma \ref{BesselProp},
\begin{align*}
\|\Rc_{2, n, \ell, m}\|_{H^s} & \leq C(n, s) \Gamma\left(n + \frac{1 - \alpha}{2}\right) h^{- n - \frac{1 - \alpha}{2}} \|\vp\|_{H^s} \|\psi\|_{W^{4, \infty}}^m \|\vp\|_{W^{4, \infty}}^{2n - \ell - m}\\
& \qquad + C(n, s) \Gamma\left(n + \frac{1 - \alpha}{2}\right) h^{- n - \frac{3 - \alpha}{2}} \|\psi\|_{H^s} \|\psi\|_{W^{4, \infty}}^{m - 1} \|\vp\|_{W^{4, \infty}}^{2n - \ell + 1 - m}.
\end{align*}
When $m = 2n - \ell + 1$, again, by Lemma \ref{BesselProp}, we have
\begin{align*}
& \Rc_{2, n, \ell, 2n - \ell + 1} = \px |\px|^{n + \frac{1 - \alpha}{2}} K_{n + \frac{1 - \alpha}{2}}(2 h |\px|)\psi^{2n - \ell + 1},\\
& \|\Rc_{2, n, \ell, 2n - \ell + 1}\|_{H^s} \leq C(n, s) \Gamma\left(n + \frac{1 - \alpha}{2}\right) h^{- n - \frac{3 - \alpha}{2}} \|\psi\|_{H^s} \|\psi\|_{W^{4, \infty}}^{2n - \ell}.
\end{align*}

The estimates \eqref{BfRc2est} for $\Bf_1[\vp, \psi]$ and
\[
\Rc_2 = \sum_{n = 1}^\infty \sum_{\ell = 0}^n \sum_{m = 1}^{2n - \ell + 1} \Rc_{2, n, \ell, m},
\]
then follow from the above estimates and Stirling's formula applied to the $\Gamma$-function coefficients.

\end{proof}

\begin{proposition}
\label{prop:paraeqn01}
The first equation of system \eqref{gsqgsys01-expd} can be written as
\begin{align}
\label{paraeqn01}
\begin{split}
\vp_t(x, t) & +
v \vp_x(x, t) + 2\Theta_+ \sin\bigg(\frac{\pi \alpha}{2}\bigg) \Gamma(\alpha - 1) |\px|^{1 - \alpha} \vp_x(x, t)\\
& \qquad+ \px T_{\Bf_+[\vp, \psi]} \vp(x, t) + \Theta_+\px |\px|^{1 - \alpha} \Big\{T_{B^{1 - \alpha}[\varphi]} \varphi(x, t) + D^{-1} T_{B^{- \alpha}[\varphi]} \varphi(x, t)\Big\} + \Rc = 0,
\end{split}
\end{align}
where $T_{\Bf_+[\vp, \psi]}$ is self-adjoint and its symbol and the remainder term $\Rc$ satisfy the estimates
\begin{align}
\begin{split}
 \Bf_+[\vp, \psi] =  \Theta_- \sum_{n = 1}^\infty \sum_{\ell = 0}^n (2n - \ell + 1) d_{n, \ell, 0, 1} \varphi^{2n - \ell} + \Theta_+ B^0[\vp] + \Theta_- \Bf_1[\vp, \psi],
\end{split}\notag\\
\begin{split}
& \|\Bf_+[\vp, \psi]\|_{\M_{(3, 5)}} \lesssim \sum_{n = 1}^\infty \sum_{\ell = 0}^n C(n, s) h^{\ell - n - 1 + \frac{\alpha}{2}} \|\vp\|_{W^{4, \infty}}^{2n - \ell} + \sum_{n = 1}^\infty C(n, s) |c_n| \|\vp\|_{W^{4, \infty}}^{2n}\\
& \hspace{1.5in} + \sum_{n = 1}^\infty \sum_{\ell = 0}^n \sum_{m = 1}^{2n - \ell} C(n, s) h^{\ell - 2n - (1 - \alpha)} \|\vp\|_{W^{4, \infty}}^{2n - \ell - m} \|\psi\|_{W^{4, \infty}}^m,
\end{split}\label{Bfest01}\\
\begin{split}
& \|\Rc\|_{H^s} \lesssim h^{- (1 - \alpha)} \|\psi\|_{H^s}
+ \|\vp\|_{H^s} \sum_{n = 1}^\infty C(n, s) |c_n| \|\vp\|_{W^{4, \infty}}^{2n}
\\
&\hspace{1.5in}+ \|\vp\|_{H^s} \sum_{n = 1}^\infty \sum_{\ell = 0}^n C(n, s) h^{\ell - n - 1 + \frac{\alpha}{2}} \|\vp\|_{W^{4, \infty}}^{2n - \ell}\\
& \hspace{1.5in} + \|\vp\|_{H^s} \sum_{n = 1}^\infty \sum_{\ell = 0}^n \sum_{m = 1}^{2n - \ell} C(n, s) h^{\ell - 2n - (1 - \alpha)} \|\vp\|_{W^{4, \infty}}^{2n - \ell + 1 - m} \|\psi\|_{W^{4, \infty}}^m\\
& \hspace{1.5in} + \|\psi\|_{H^s} \sum_{n = 1}^\infty \sum_{\ell = 0}^n \sum_{m = 1}^{2n - \ell} C(n, s) h^{\ell - 2n - (2 - \alpha)} \|\vp\|_{W^{4, \infty}}^{2n - \ell + 1 - m} \|\psi\|_{W^{4, \infty}}^{m - 1}\\
& \hspace{1.5in} + \|\psi\|_{H^s} \sum_{n = 1}^\infty \sum_{\ell = 0}^n C(n, s) h^{\ell - 2n + \alpha - \frac{5}{2}} \|\psi\|_{W^{4, \infty}}^{2n - \ell}.\label{paraRest01}
\end{split}
\end{align}
\end{proposition}

\begin{proof}
We first use Bony's decomposition to write the nonlinear terms as
\[
d_{n, \ell, 0, 1} \px \left\{\left(\varphi(x, t)\right)^{2n - \ell + 1}\right\} = d_{n, \ell, 0, 1} (2n - \ell + 1) T_{\left(\varphi(x, t)\right)^{2n - \ell}} \varphi_x(x, t) + \Rc_3,
\]
where
\[
\|\Rc_3\|_{H^s} \lesssim C(n, s) h^{\ell - n - 1 - \frac{\alpha}{2}} \|\vp\|_{H^s} \|\vp\|_{W^{1, \infty}}^{2n - \ell}.
\]
Using Lemma \ref{BesselProp}, we obtain
\[
\bigg\|\Theta_- \frac{2 \sqrt{\pi}}{\Gamma\left(1 - \frac{\alpha}{2}\right) (4 h)^{\frac{1 - \alpha}{2}}} |\px|^{\frac{1 - \alpha}{2}} K_{\frac{1 - \alpha}{2}}(2 h |\px|) \psi_x\bigg\|_{H^s} \lesssim h^{- (1 - \alpha)} \|\psi\|_{H^s}.
\]
Therefore, by Proposition \ref{varphinonlindecomp01}, Proposition \ref{varphipsinonlindecomp01}, and Kato-Ponce type commutator estimates, we obtain \eqref{paraeqn01} and the estimates \eqref{Bfest01} and \eqref{paraRest01}. The self-adjointness of $T_{\Bf[\vp, \psi]}$ follows from the fact that the symbol $\Bf[\vp, \psi]$ is real-valued.
\end{proof}

\subsection{Energy estimates and local existence}

In this subsection, we omit the dependence of $\vp$ or $\psi$ in the symbols $\beta[\vp]$, $B^{1 - \alpha}[\vp]$, $B^{- \alpha}[\vp]$, and $\Bf[\vp, \psi]$ when there is no ambiguity.

Writing
\begin{equation}
\vartheta = 2  \sin\Big(\frac{\pi \alpha}{2}\Big) \frac{\Gamma(\alpha)}{1 - \alpha}=-2 \sin\Big(\frac{\pi \alpha}{2}\Big) \Gamma(\alpha-1),
\label{def_vartheta}
\end{equation}
we define $\beta[\vp](x, \eta)$ as
\begin{align}
\label{betadef01}
\beta[\vp](x, \eta) = \left(1 - \frac{1}{\vartheta} B^{1 - \alpha}[\vp](x, \eta)\right)^{ \frac{1 - \alpha}{2 (2 - \alpha)}},
\end{align}
which, thanks to \eqref{defB01}, is a solution to the first order variable coefficient PDE
\[
\eta \px B^{1 - \alpha} \partial_\eta \beta + \left[\vartheta_+ (2 - \alpha) - \eta \partial_\eta B^{1 - \alpha} - (2 - \alpha) B^{1 - \alpha}\right] \px \beta + i B^{- \alpha} \beta = 0.
\]

Since $B^{1 - \alpha}[\vp] \in \M_{(3, 5)}$ with estimates \eqref{B-est01}, we derive that $\beta[\vp] \in \M_{(3, 5)}$ with estimates
\begin{align}
\label{beta-est}
\|\beta[\vp] - 1\|_{\M_{(3, 5)}} + \|\partial_t \beta[\vp]\|_{\M_{(1, 1)}} + \|\px \beta[\vp]\|_{\M_{(2, 5)}} \lesssim \sum_{n = 1}^\infty C(n, s) |c_n| \|\vp\|_{W^{4, \infty}}^{2n}.
\end{align}

According to the decomposition of the nonlinear terms in last subsection, we construct a weighted energy as follows.
\begin{align}
\label{weightedE01}
\begin{split}
E_\vp^{(j)}(t) & = \int_\R |D|^j T_{\beta[\vp]} \varphi(x, t) \cdot \left(\vartheta - T_{B^{1 - \alpha}[\vp]}\right)^{2j + 1}  |D|^j T_{\beta[\vp]} \vp(x, t) \diff{x},\\
\tilde{E}_\vp^{(s)}(t) & = \|\vp\|_{L^2(\R)}^2 + \sum_{j = 1}^s E^{(j)}(t).
\end{split}
\end{align}

Now we are ready to derive \emph{a priori} estimates for the system.

\begin{proposition}
\label{apriori01}
Let $s \geq 5$ be an integer and $(\vp, \psi)$ a smooth solution of \eqref{gsqgsys01-expd}, with $(\vp_0, \psi_0) \in H^s(\R) \times H^s(\R)$. There exists a constant $\tilde{C} > 1$, depending only on $s$, such that if $(\vp_0, \psi_0)$ satisfyies
\begin{align}
\label{initbdd01}
\begin{split}
& \left\|\vartheta - T_{B^{1 - \alpha}[\vp_0]}\right\|_{L^2 \to L^2} \ge m_0,\qquad \left\|\vartheta - T_{B^{1 - \alpha}[ \psi_0]}\right\|_{L^2 \to L^2} \ge m_0,\\
& \left\|T_{\beta[\vp_0]}\right\|_{L^2 \to L^2} \geq m_0', \qquad \left\|T_{\beta[\psi_0]}\right\|_{L^2 \to L^2} \geq m_0',\\
& \sum_{n = 0}^\infty \sum_{\ell = 0}^n \sum_{m = 0}^{2n - \ell} \tilde{C}^n \left(1 + h^{\ell - 2n + \alpha - \frac{5}{2}}\right) \|\vp_0\|_{W^{4, \infty}}^{2n - \ell - m} \|\psi_0\|_{W^{4, \infty}}^m < \infty,
\end{split}
\end{align}
for some constants $m_0, m_0' > 0$, then there exists a time $T > 0$ such that
\begin{align*}
& \left\|\vartheta - T_{B^{1 - \alpha}[\vp]}\right\|_{L^2 \to L^2} > \frac12m_0, \qquad \left\|\vartheta - T_{B^{1 - \alpha}[\psi]}\right\|_{L^2 \to L^2} > \frac12m_0,\\
& \left\|T_{\beta[\vp]}\right\|_{L^2 \to L^2} >  \frac12m_0', \qquad \left\|T_{\beta[\psi]}\right\|_{L^2 \to L^2} >  \frac12m_0',\\
& \sum_{n = 0}^\infty \sum_{\ell = 0}^n \sum_{m = 0}^{2n - \ell} \tilde{C}^n \left(1 + h^{\ell - 2n + \alpha - \frac{5}{2}}\right) \|\vp\|_{W^{4, \infty}}^{2n - \ell - m} \|\psi\|_{W^{4, \infty}}^m < \infty,
\end{align*}
for all $t \in [0, T]$, and
\begin{align}
\label{apest01}
\frac{\diff}{\diff{t}} \left\{\tilde{E}_\vp^{(s)}(t) + \tilde{E}_\psi^{(s)}(t)\right\} \leq \left(\|\vp\|_{H^s}^2 + \|\psi\|_{H^s}^2\right) \cdot F\left(\|\vp\|_{W^{4, \infty}}, \|\psi\|_{W^{4, \infty}}\right),
\end{align}
where $\tilde{E}_\vp^{(s)}(t)$ is defined in \eqref{weightedE01} and $\tilde{E}_\psi^{(s)}(t)$ can be defined analogously, and $F \colon \R_+ \times \R_+ \to \R$ is a continuous, real-valued function that is monotone-increasing in either variables, such that
\begin{align}
\label{defF01}
F\left(\|\vp\|_{W^{4, \infty}}, \|\psi\|_{W^{4, \infty}}\right) \approx \sum_{n = 0}^\infty \sum_{\ell = 0}^n \sum_{m = 0}^{2n - \ell} \tilde{C}^n \left(1 + h^{\ell - 2n + \alpha - \frac{5}{2}}\right) \|\vp\|_{W^{4, \infty}}^{2n - \ell - m} \|\psi\|_{W^{4, \infty}}^m.
\end{align}
\end{proposition}

Before proving this proposition, we first state a lemma, whose proof follows directly from definition of Weyl para-product and Kato-Ponce type commutator estimates, and is similar to the proof of Lemma 4.2 in \cite{HSZ18p}.
\begin{lemma}
\label{lem-dtcomm01}
For any positive integer $k$, if $(\vp, \psi)$ is a smooth solution of \eqref{gsqgsys01-expd} and $f \in C_t^1L_x^2$, then
\[
\partial_t (\vartheta - T_{B^{1 - \alpha[\vp]}})^{k} f = (\vartheta_+ - T_{B^{1 - \alpha[\vp]}})^k f_t - k (\vartheta_+ - T_{B^{1 - \alpha[\vp]}})^{k - 1} T_{\partial_t B^{1 - \alpha}[\vp]} f + \tilde{\Rc}_1(f),
\]
where the remainder term satisfies
\[
\left\|\tilde{\Rc}_1(f)\right\|_{H^1} \lesssim \|f\|_{L^2} F\left(\|\vp\|_{W^{4, \infty}}, \|\psi\|_{W^{4, \infty}}\right).
\]
\end{lemma}

\begin{proof}[Proof of Proposition \ref{apriori01}]
We first observe that $\|\vp\|_{L^2(\R)}$ is conserved by the system, so we only need to estimate the higher-order energies.
By assumption \eqref{initbdd01} on the  initial data and continuity in time, there exists $T > 0$ such that for all $0 \leq t \leq T$,
\[
\sum_{n = 1}^\infty \sum_{\ell = 0}^n \sum_{m = 0}^{2n - \ell} \tilde{C}^n \left(1 + h^{\ell - 2n + \alpha - \frac{5}{2}}\right) \|\vp\|_{W^{4, \infty}}^{2n - \ell - m} \|\psi\|_{W^{4, \infty}}^m < \infty,
\]
and
\[
\left\|T_{\beta[\vp]} \vp\right\|_{H^s} \approx \|\vp\|_{H^s}.
\]

We first apply $T_{\beta[\vp]}$ to equation \eqref{paraeqn01} to obtain
\begin{align}
\label{betaeqn01}
\begin{split}
\partial_t T_{\beta[\vp]} \vp & - \Theta_+\px |\px|^{1 - \alpha} \left(\vartheta - T_{B^{1 - \alpha}[\vp]}\right) T_{\beta[\vp]} \vp +
v \px T_{\beta[\vp]} \vp + \px T_{\Bf[\vp, \psi]} T_{\beta[\vp]} \vp\\
& \quad =\Theta_+ \bigg\{\vartheta \left[T_{\beta[\vp]}, \px |\px|^{1 - \alpha}\right] - \left[T_{\beta[\vp]}, \px |\px|^{1 - \alpha} T_{B^{1 - \alpha}[\vp]}\right] - \px D^{-1} |\px|^{1 - \alpha} T_{B^{-\alpha}[\vp]} T_{\beta[\vp]}\bigg\} \vp\\
& \quad - \left[T_{\beta[\vp]}, \px T_{\Bf[\vp, \psi]}\right] \vp -\Theta_+ \left[T_{\beta[\vp]}, \px D^{-1} |\px|^{1 - \alpha} T_{B^{-\alpha}[\vp]}\right] \vp - \left[T_{\beta[\vp]}, \partial_t\right] \vp\\
& \qquad +
v \left[T_{\beta[\vp]}, \px\right] \vp - T_{\beta[\vp]} \Rc.
\end{split}
\end{align}

We start with the first commutator term in the curly bracket on the right-hand-side of the equation. Using a Taylor expansion, we obtain that
\begin{align*}
&\F \Bigg[\vartheta \left[T_{\beta[\vp]}, \px |\px|^{1 - \alpha}\right] \varphi\Bigg](\xi)\\
 = ~& \vartheta \int_\R \chi\bigg(\frac{|\xi - \eta|}{|\xi + \eta|}\bigg) \tilde{\beta}\bigg(\xi - \eta, \frac{\xi + \eta}{2}\bigg) i \big(\eta |\eta|^{1 - \alpha} - \xi |\xi|^{1 - \alpha}\big) \hat{\vp}(\eta) \diff{\eta}\\
= ~& - \vartheta (2 - \alpha) \int_\R |\xi|^{1 - \alpha} \chi\bigg(\frac{|\xi - \eta|}{|\xi + \eta|}\bigg) i (\xi - \eta) \tilde{\beta}\bigg(\xi - \eta, \frac{\xi + \eta}{2}\bigg) \hat{\vp}(\eta) \diff{\eta} + \Rc_4,
\end{align*}
where $\tilde{a}$ denotes the partial Fourier transform of a symbol $a$ 
defined in \eqref{atilde}, and $\Rc_4$ is a remainder term satisfying
\[
\|\Rc_4\|_{H^s} \lesssim \|\vp\|_{H^s} \sum_{n = 1}^\infty C(n, s) |c_n| \|\vp\|_{W^{4, \infty}}^{2n}.
\]

For the second commutator term in the curly bracket in \eqref{betaeqn01}, we use a Taylor expansion and Lemma \ref{weyl-comm} to obtain
\begin{align*}
& \F\bigg[\left[T_{\beta[\vp]}, \px |\px|^{1 - \alpha} T_{B^{1 - \alpha}[\vp]}\right]\vp\bigg](\xi)\\
= ~& \int_{\R^2} \chi\bigg(\frac{|\xi - \eta|}{|\xi + \eta|}\bigg) \chi\bigg(\frac{|\eta - \zeta|}{|\eta + \zeta|}\bigg) \cdot \bigg[i \eta |\eta|^{1 - \alpha} \tilde{\beta}\bigg(\xi - \eta, \frac{\xi + \eta}{2}\bigg) \widetilde{B^{1 - \alpha}}\bigg(\eta - \zeta, \frac{\eta + \zeta}{2}\bigg)\\
& \hspace{2in} - i \xi |\xi|^{1 - \alpha} \tilde{\beta}\bigg(\eta - \zeta, \frac{\eta + \zeta}{2}\bigg) \widetilde{B^{1 - \alpha}}\bigg(\xi - \eta, \frac{\xi + \eta}{2}\bigg)\bigg] \hat{\vp}(\zeta) \diff{\eta}\diff{\zeta}\\
= ~& \int_{\R^2} i \xi |\xi|^{1 - \alpha} \chi\bigg(\frac{|\xi - \eta|}{|\xi + \eta|}\bigg) \chi\bigg(\frac{|\eta - \zeta|}{|\eta + \zeta|}\bigg) \cdot \bigg[\tilde{\beta}\bigg(\xi - \eta, \frac{\xi + \eta}{2}\bigg) \widetilde{B^{1 - \alpha}}\bigg(\eta - \zeta, \frac{\eta + \zeta}{2}\bigg)\\
& \hspace{2.5in} - \tilde{\beta}\bigg(\eta - \zeta, \frac{\eta + \zeta}{2}\bigg) \widetilde{B^{1 - \alpha}}\bigg(\xi - \eta, \frac{\xi + \eta}{2}\bigg)\bigg] \hat{\vp}(\zeta) \diff{\eta}\diff{\zeta}\\
& \quad - (2 - \alpha) \int_{\R^2} \chi\bigg(\frac{|\xi - \eta|}{|\xi + \eta|}\bigg) \chi\bigg(\frac{|\eta - \zeta|}{|\eta + \zeta|}\bigg) |\xi|^{1 - \alpha} i (\xi - \eta) \tilde{\beta}\bigg(\xi - \eta, \frac{\xi + \eta}{2}\bigg) \widetilde{B^{1 - \alpha}}\bigg(\eta - \zeta, \frac{\eta + \zeta}{2}\bigg) \hat{\vp}(\zeta) \diff{\eta}\diff{\zeta}\\
& \quad + \Rc_5,\\
= ~& \int_\R |\xi|^{1 - \alpha} \chi\bigg(\frac{|\xi - \eta|}{|\xi + \eta|}\bigg) \bigg[\frac{\xi + \eta}{2} \cdot \reallywidetilde{\left\{\beta, B^{1 - \alpha}\right\}}\bigg(\xi - \eta, \frac{\xi + \eta}{2}\bigg)  - (2 - \alpha) \reallywidetilde{(\px \beta) B^{1 - \alpha}}\bigg(\xi - \eta, \frac{\xi + \eta}{2}\bigg)\bigg] \hat{\vp}(\eta) \diff{\eta}\\
& \quad + \Rc_6,
\end{align*}
where
\[
\|\Rc_5\|_{H^s} + \|\Rc_6\|_{H^s} \lesssim \|\vp\|_{H^s} \bigg(\sum_{n = 1}^\infty C(n, s) |c_n| \|\vp\|_{W^{4, \infty}}^{2n}\bigg)^2.
\]

For last term in the curly bracket in \eqref{betaeqn01}, we use Lemma \ref{weyl-comm} again to get
\begin{align*}
& \F\left[\px D^{-1} |\px|^{1 - \alpha} T_{B^{- \alpha}[\vp]} T_{\beta[\vp]} \vp\right](\xi)\\
&\qquad = \int_\R |\xi|^{1 - \alpha} \chi\bigg(\frac{|\xi - \eta|}{|\xi + \eta|}\bigg) i \widetilde{\beta B^{- \alpha}}\bigg(\xi - \eta, \frac{\xi + \eta}{2}\bigg) \hat{\vp}(\eta) \diff{\eta} + \Rc_7,
\end{align*}
where
\[
\|\Rc_7\|_{H^s} \lesssim \|\vp\|_{H^s} \bigg(\sum_{n = 1}^\infty C(n, s) |c_n| \|\vp\|_{W^{4, \infty}}^{2n}\bigg)^2.
\]

Therefore, by invoking the choice of $\beta[\vp]$, we conclude that the operator in the curly bracket in \eqref{betaeqn01} is of order $0$ and the whole term is an error term that satisfies
\begin{align*}
& \bigg\|\bigg\{\vartheta_+ \left[T_{\beta[\vp]}, \px |\px|^{1 - \alpha}\right] - \left[T_{\beta[\vp]}, \px |\px|^{1 - \alpha} T_{B^{1 - \alpha}[\vp]}\right] - \px D^{-1} |\px|^{1 - \alpha} T_{B^{-\alpha}[\vp]} T_{\beta[\vp]}\bigg\} \vp\bigg\|_{H^s}\\
&\qquad \lesssim\|\vp\|_{H^s} \bigg(\sum_{n = 1}^\infty C(n, s) |c_n| \|\vp\|_{W^{4, \infty}}^{2n}\bigg)^2.
\end{align*}
By Kato-Ponce type estimates, \eqref{B-est01}, \eqref{beta-est}, \eqref{Bfest01}, and \eqref{paraRest01}, we also find that
\begin{align*}
& \bigg\|\left[T_{\beta[\vp]}, \px T_{\Bf[\vp, \psi]}\right] \vp\bigg\|_{H^s} + \bigg\|\left[T_{\beta[\vp]}, \px D^{-1} |\px|^{1 - \alpha} T_{B^{-\alpha}[\vp]}\right] \vp\bigg\|_{H^s}\\
& \qquad + \bigg\|\left[T_{\beta[\vp]}, \partial_t\right] \vp\bigg\|_{H^s} + \bigg\|\left[T_{\beta[\vp]}, \px\right] \vp\bigg\|_{H^s} + \left\|T_{\beta[\vp]} \Rc\right\|_{H^s}\\
& \hspace{1in} \lesssim \left(\|\vp\|_{H^s} + \|\psi\|_{H^s}\right) \cdot F(\|\vp\|_{W^{4, \infty}}, \|\psi\|_{W^{4, \infty}}).
\end{align*}

We conclude that equation \eqref{betaeqn01} can be rewritten as
\begin{align*}
\partial_t T_{\beta[\vp]} \vp & - \Theta_+\px |\px|^{1 - \alpha} \left(\vartheta - T_{B^{1 - \alpha}[\vp]}\right) T_{\beta[\vp]} \vp +
v \px T_{\beta[\vp]} \vp + \px T_{\Bf[\vp, \psi]} T_{\beta[\vp]} \vp = \Rc_7,
\end{align*}
where
\[
\|\Rc_7\|_{H^s} \lesssim \left(\|\vp\|_{H^s} + \|\psi\|_{H^s}\right) \cdot F(\|\vp\|_{W^{4, \infty}}, \|\psi\|_{W^{4, \infty}}).
\]

Applying the operator $|D|^s$ to this equation and using Lemma \ref{lem-DsT}, we get that
\begin{align}
\label{Dsbetaeqn01}
\begin{split}
&|D|^s \partial_t T_{\beta[\vp]} \vp - \Theta_+\px |\px|^{1 - \alpha} \left(\vartheta - T_{B^{1 - \alpha}[\vp]}\right) |D|^s T_{\beta[\vp]} \vp +
v \px |D|^s T_{\beta[\vp]} \vp + \px |D|^s T_{\Bf[\vp, \psi]} T_{\beta[\vp]} \vp = \Rc_8,
\end{split}
\end{align}
where
\[
\|\Rc_8\|_{L^2} \lesssim \left(\|\vp\|_{H^s} + \|\psi\|_{H^s}\right) \cdot F(\|\vp\|_{W^{4, \infty}}, \|\psi\|_{W^{4, \infty}}).
\]

Applying $(\vartheta - T_{B^{1 - \alpha}[\vp]})^s$ to \eqref{Dsbetaeqn01} and commuting $(\vartheta - T_{B^{1 - \alpha}[\vp]})^s$ with $\px |\px|^{1 - \alpha}$  up to remainder terms, we obtain that
\begin{align}
\label{Dsbetaeqn01b}
\begin{split}
& \left(\vartheta - T_{B^{1 - \alpha}[\vp]}\right)^s |D|^s \partial_t T_{\beta[\vp]} \vp - \Theta_+\px |\px|^{1 - \alpha} \left(\vartheta - T_{B^{1 - \alpha}[\vp]}\right)^{s + 1} |D|^s T_{\beta[\vp]} \vp\\
& \hspace{.5in} +
v \left(\vartheta - T_{B^{1 - \alpha}[\vp]}\right)^{s + 1} \px |D|^s T_{\beta[\vp]} \vp + \left(\vartheta - T_{B^{1 - \alpha}[\vp]}\right)^s \px |D|^s T_{\Bf[\vp, \psi]} T_{\beta[\vp]} \vp = \Rc_9,
\end{split}
\end{align}
where
\begin{align}
\label{R9est01}
\|\Rc_9\|_{L^2} \lesssim \left(\|\vp\|_{H^s} + \|\psi\|_{H^s}\right) \cdot F(\|\vp\|_{W^{4, \infty}}, \|\psi\|_{W^{4, \infty}}).
\end{align}

By Lemma \ref{lem-dtcomm01}, with $k=2s+1$ and $f = |D|^s T_{\beta[\vp]} \vp$, the time derivative of $E_\vp^{(s)}(t)$ in \eqref{weightedE01} is
\begin{align}
\label{dtE01}
\begin{split}
\frac{\diff}{\diff{t}} E_\vp^{(s)}(t) & = - (2s + 1) \int_\R |D|^s T_{\beta[\vp]} \vp \left(\vartheta - T_{B^{1 - \alpha}[\vp]}\right)^{2s} T_{\partial_t B^{1 - \alpha}[\vp]} |D|^s T_{\beta[\vp]} \vp \diff{x}\\
& \qquad + 2 \int_\R |D|^s T_{\beta[\vp]} \vp \cdot \left(\vartheta - T_{B^{1 - \alpha}[\vp]}\right)^{2s + 1} |D|^s \partial_t \left(T_{\beta[\vp]} \vp\right) \diff{x}\\
& \qquad + \int_\R \tilde{\Rc}_1\left(|D|^s T_{\beta[\vp]} \vp\right) \cdot |D|^s T_{\beta[\vp]} \vp \diff{x}.
\end{split}
\end{align}

Equation \eqref{gsqgsys01-expd} implies that
\[
\|\partial_t \vp_x\|_{L^\infty} \lesssim F\left(\|\vp\|_{W^{4, \infty}}, \|\psi\|_{W^{4, \infty}}\right),
\]
so the first term on the right-hand side of \eqref{dtE01} can be estimated by
\begin{align*}
& \bigg|\int_\R |D|^s T_{\beta[\vp]} \vp \left(\vartheta - T_{B^{1 - \alpha}[\vp]}\right)^{2s} T_{\partial_t B^{1 - \alpha}[\vp]} |D|^s T_{\beta[\vp]} \vp \diff{x}\bigg|\\
& \qquad \lesssim \|\vp\|_{H^s}^2 F\left(\|\vp\|_{W^{4, \infty}}, \|\psi\|_{W^{4, \infty}}\right).
\end{align*}

Using Lemma \ref{lem-dtcomm01}, we can estimate the third term  on the right-hand side of \eqref{dtE01} by
\[
\bigg|\int_\R \tilde{\Rc}_1\left(|D|^s T_{\beta[\vp]} \vp\right) \cdot |D|^s T_{\beta[\vp]} \vp \diff{x}\bigg| \lesssim \left(\|\vp\|_{H^s}^2 + \|\vp\|_{H^s} \|\psi\|_{H^s}\right) F\left(\|\vp\|_{W^{4, \infty}}, \|\psi\|_{W^{4, \infty}}\right).
\]

To estimate the second term on the right-hand side of \eqref{dtE01}, we multiply \eqref{Dsbetaeqn01b} by
\[
\left(\vartheta - T_{B^{1 - \alpha}[\vp]}\right)^{s + 1} |D|^s T_{\beta[\vp]} \vp,
\]
integrate the result with respect to $x$, and use the self-adjointness of $\left(\vartheta_+ - T_{B^{1 - \alpha}[\vp]}\right)^{s + 1}$ to obtain
\[
\int_\R |D|^s T_{\beta[\vp]} \vp \cdot \left(\vartheta - T_{B^{1 - \alpha}[\vp]}\right)^{2s + 1} |D|^{s} \partial_t \left(T_{\beta[\vp]} \vp\right) \diff{x} = \Rm{1} + \Rm{2} + \Rm{3} + \Rm{4},
\]
where
\begin{align*}
\Rm{1} &= - \int_\R |D|^s T_{\beta[\vp]} \vp \cdot \left(\vartheta - T_{B^{1 - \alpha}[\vp]}\right)^{2s + 1} |D|^s \px T_{\Bf[\vp, \psi]} T_{\beta[\vp]} \vp \diff{x},\\
\Rm{2} &= \int_\R \left(\vartheta - T_{B^{1 - \alpha}[\vp]}\right)^{s + 1} |D|^s T_{\beta[\vp]} \vp \cdot \px |\px|^{1 - \alpha} \left(\vartheta - T_{B^{1 - \alpha}[\vp]}\right)^{s + 1} |D|^s T_{\beta[\vp]} \vp \diff{x},\\
\Rm{3} &= \int_\R \left(\vartheta - T_{B^{1 - \alpha}[\vp]}\right)^{s + 1} |D|^s T_{\beta[\vp]} \vp \cdot \Rc_9 \diff{x},\\
\Rm{4} &= -
v \int_\R |D|^s  T_{\beta[\vp]} \vp \cdot \left(\vartheta - T_{B^{1 - \alpha}[\vp]}\right)^{2s + 1} |D|^s \px T_{\beta[\vp]} \vp \diff{x}.
\end{align*}
The last three terms are straightforward to estimate, but the first term requires more work.

Since $\px |\px|^{1 - \alpha}$ is skew-adjoint, we have
\[
\Rm{2} = 0.
\]
By \eqref{R9est01} and the boundedness of
\[
\left(\vartheta_+ - T_{B^{1 - \alpha}[\vp]}\right)^{s + 1}
\]
on $L^2$, we have that
\[
|\Rm{3}| \lesssim \left(\|\vp\|_{H^s}^2 + \|\vp\|_{H^s} \|\psi\|_{H^s}\right)\cdot F\left(\|\vp\|_{W^{4, \infty}}, \|\psi\|_{W^{4, \infty}}\right).
\]
Since $\left(\vartheta_+ - T_{B^{1 - \alpha}[\vp]}\right)^{2s + 1}$ is self-adjoint, we can integrate by parts to obtain
\begin{align*}
\Rm{4} &= - \Rm{4} -
v \int_\R |D|^s  T_{\beta[\vp]} \vp \cdot \left[\left(\vartheta_+ - T_{B^{1 - \alpha}[\vp]}\right)^{2s + 1}, \px\right] |D|^s T_{\beta[\vp]} \vp \diff{x}.
\end{align*}
Using commutator estimates, we have
\[
\bigg|\int_\R |D|^s  T_{\beta[\vp]} \vp \cdot \left[\left(\vartheta_+ - T_{B^{1 - \alpha}[\vp]}\right)^{2s + 1}, \px\right] |D|^s T_{\beta[\vp]} \vp \diff{x}\bigg| \lesssim \|\vp\|_{H^s}^2 F\left(\|\vp\|_{W^{4, \infty}}, \|\psi\|_{W^{4, \infty}}\right),
\]
and we conclude that
\[
|\Rm{4}| \lesssim \|\vp\|_{H^s}^2 F\left(\|\vp\|_{W^{4, \infty}}, \|\psi\|_{W^{4, \infty}}\right).
\]

\noindent {\bf Term} $\Rm{1}$ {\bf estimate.} We write  $\Rm{1} = -\Rm{1}_a + \Rm{1}_b$, where
\begin{align*}
\Rm{1}_a & = \int_{\R} |D|^s T_{\beta[\vp]} \vp \cdot \left(\vartheta_+ - T_{B^{1 - \alpha}[\vp]}\right)^{2s + 1} \partial_x  T_{\Bf[\vp, \psi]} |D|^s T_{\beta[\vp]} \vp \diff{x},\\
 \Rm{1}_b&= \int_{\R} |D|^s T_{\beta[\vp]} \vp \cdot \left(\vartheta_+ - T_{B^{1 - \alpha}[\vp]}\right)^{2s + 1} \partial_x \left[T_{\Bf[\vp, \psi]}, |D|^s\right] T_{\beta[\vp]} \vp \diff{x}.
\end{align*}
By commutator estimates and \eqref{Bfest01}, the second integral satisfies
\[
|\Rm{1}_b| \lesssim \|\vp\|_{H^s}^2 \cdot F\left(\|\vp\|_{W^{4, \infty}}, \|\psi\|_{W^{4, \infty}}\right).
\]
To estimate the integral $\Rm{1}_a$, we write it as
\begin{align*}
\Rm{1}_a &= \Rm{1}_{a_1} - \Rm{1}_{a_2},
\end{align*}
where
\begin{align*}
\Rm{1}_{a_1}& = \int_{\R} |D|^s T_{\beta[\vp]} \vp \cdot \left[\left(\vartheta_+ - T_{B^{1 - \alpha}[\vp]}\right)^{2s + 1}, \partial_x\right] \left(T_{\Bf[\vp, \psi]} |D|^s T_{\beta[\vp]} \vp\right) \diff{x},
\\
\Rm{1}_{a_2}  &= \int_{\R} \px |D|^s T_{\beta[\vp]} \vp \cdot \left(\vartheta_+ - T_{B^{1 - \alpha}[\vp]}\right)^{2s + 1} \left(T_{\Bf[\vp, \psi]} |D|^s T_{\beta[\vp]} \vp\right) \diff{x}.
\end{align*}

\noindent {\bf Term} $\Rm{1}_{a_1}$ {\bf estimate.} A Kato-Ponce commutator estimate and \eqref{Bfest01} gives
\[
|\Rm{1}_{a_1}| \lesssim \|\vp\|_{H^s}^2 \cdot F\left(\|\vp\|_{W^{4, \infty}}, \|\psi\|_{W^{4, \infty}}\right).
\]

\noindent {\bf Term} $\Rm{1}_{a_2}$ {\bf estimate.} We have
\begin{align}
\label{eqBs01}
\begin{split}
\Rm{1}_{a_2} &= \int_\R \left(T_{\Bf[\vp, \psi]} |D|^s T_{\beta[\vp]} \vp\right) \cdot \left(\vartheta_+ - T_{B^{1 - \alpha}[\vp]}\right)^{2s + 1} \px |D|^s T_{\beta[\vp]} \vp \diff{x}\\
&= - \int_\R \partial_x \left(T_{\Bf[\vp, \psi]} |D|^s T_{\beta[\vp]} \vp\right) \cdot \left(\vartheta_+ - T_{B^{1 - \alpha}[\vp]}\right)^{2s + 1} |D|^s T_{\beta[\vp]} \vp \diff{x}\\
& \qquad - \int_\R \left(T_{\Bf[\vp, \psi]} |D|^s T_{\beta[\vp]} \vp\right) \cdot \left[\partial_x, \left(\vartheta_+ - T_{B^{1 - \alpha}[\vp]}\right)^{2s + 1}\right] |D|^s T_{\beta[\vp]} \vp \diff{x}\\
&= - \int_\R \left(T_{\Bf[\vp, \psi]} \px |D|^s T_{\beta[\vp]} \vp + \left[\partial_x, T_{\Bf[\vp, \psi]}\right] |D|^s T_{\beta[\vp]} \vp\right) \cdot \left(\vartheta_+ - T_{B^{1 - \alpha}[\vp]}\right)^{2s + 1} |D|^s T_{\beta[\vp]} \vp \diff{x}\\
& \qquad - \int_\R \left(T_{\Bf[\vp, \psi]} |D|^s T_{\beta[\vp]} \vp\right) \cdot \left[\partial_x, \left(\vartheta_+ - T_{B^{1 - \alpha}[\vp]}\right)^{2s + 1}\right] |D|^s T_{\beta[\vp]} \vp \diff{x}.
\end{split}
\end{align}
Using commutator estimates and \eqref{B-est01}, \eqref{Bfest01}, and \eqref{beta-est}, we get that
\begin{align*}
& \left\|\left[\partial_x, (\vartheta_+ - T_{B^{1 - \alpha}[\vp]})^{2s + 1}\right] |D|^s T_{\beta[\vp]} \vp\right\|_{L^2} + \left\|\left[\partial_x, T_{\Bf[\vp, \psi]}\right] |D|^s T_{\beta[\vp]} \vp\right\|_{L^2}
\lesssim \|\vp\|_{H^s} \cdot F\left(\|\vp\|_{W^{4, \infty}}, \|\psi\|_{W^{4, \infty}}\right),\\
& \left\|\partial_x \left[(\vartheta_+ - T_{B^{1 - \alpha}[\vp]})^{2s + 1}, T_{\Bf[\vp, \psi]}\right] |D|^s T_{\beta[\vp]} \vp\right\|_{L^2} \lesssim \|\vp\|_{H^s} \cdot F\left(\|\vp\|_{W^{4, \infty}}, \|\psi\|_{W^{4, \infty}}\right).
\end{align*}
Since $T_{\Bf[\vp, \psi]}$ is self-adjoint, we can rewrite \eqref{eqBs01} as
\[
\begin{aligned}
\Rm{1}_{a_2}
& = - \Rm{1}_{a_2} + \Rc_{10},
\end{aligned}\]
with
\[
|\Rc_{10}| \lesssim \|\vp\|_{H^s}^2 \cdot F\left(\|\vp\|_{W^{4, \infty}}, \|\psi\|_{W^{4, \infty}}\right),
\]
and we conclude that
\[
|\Rm{1}_{a_2}| \lesssim \|\vp\|_{H^s}^2 \cdot F\left(\|\vp\|_{W^{4, \infty}}, \|\psi\|_{W^{4, \infty}}\right).
\]

This completes the estimate of the terms on the right hand side of \eqref{dtE01}. Collecting the above estimates and using the interpolation inequalities, we obtain that
\begin{align*}
\tilde{E}_\vp^{(s)}(t) \le \tilde{E}_{\vp}^{(s)}(0) + \int_0^t \left(\|\vp\|_{H^s}^2 + \|\vp\|_{H^s} \|\psi\|_{H^s}\right) \cdot F\left(\|\vp\|_{W^{4,\infty}}, \|\psi\|_{W^{4,\infty}}\right) \diff{t'},
\end{align*}
with
\begin{align*}
F\left(\|\vp\|_{W^{4, \infty}}, \|\psi\|_{W^{4, \infty}}\right) \approx \sum_{n = 0}^\infty \sum_{\ell = 0}^n \sum_{m = 0}^{2n - \ell} C(n, s) \left(1 + h^{\ell - 2n + \alpha - \frac{5}{2}}\right) \|\vp\|_{W^{4, \infty}}^{2n - \ell - m} \|\psi\|_{W^{4, \infty}}^m.
\end{align*}

We observe that there exists a constant $\tilde{C}(s) > 0$ such that $C(n, s) \lesssim \tilde{C}(s)^n$. The series in \eqref{defF01} then converges whenever $\|\varphi\|_{W^{4, \infty}} + \|\psi\|_{W^{4, \infty}}$ is sufficiently small, and we can choose $F$ to be an increasing, continuous, real-valued function that satisfies \eqref{defF01}.

Finally, since $\left\|\vartheta_+ - T_{B^{1 - \alpha}[\vp_0]}\right\|_{L^2 \to L^2} \ge m_0$ and
\[
\left\|B^{1 - \alpha}[\vp](\cdot, t)\right\|_{\M_{(3, 5)}},\quad \|T_{\beta[\vp]}\|_{L^2 \to L^2}, \quad F\left(\|\varphi\|_{W^{4, \infty}} + \|\psi\|_{W^{4, \infty}}\right)
\]
are continuous in time, there exist $\Time > 0$ and $m_1, m_2 > 0$, depending only on the initial data, such that
\[
m_1 \|\vp\|_{H^s}^2\leq \tilde{E}_\vp^{(s)} \leq m_2 \|\vp\|_{H^s}^2.
\]
Similar estimates for the second equation of the system with an analogously defined energy $\tilde{E}_\psi^{(s)}$, then give \eqref{apest01}.

\end{proof}

\section{A priori estimates for the two-front SQG ($\alpha = 1)$ systems}
\label{sec:lwp02}
\subsection{Para-differential reduction}
The following result, from Proposition 3.2 in \cite{HSZ18p}, enables us to simplify the system. We recall that $L = \log|\px|$ denotes
the Fourier multiplier with symbol $\log|\xi|$, and we use $C(n,s)$ to denote a generic constant depending only on $n$, $s$, and $h$ that grows at most exponentially in $n$.

\begin{proposition}\label{varphinonlindecomp1}
For $\alpha = 1$, suppose that $\varphi(\cdot, t), \psi(\cdot, t) \in H^s(\R)$ with $s \geq 5$ and $\|\varphi\|_{W^{4, \infty}} + \|L \varphi\|_{W^{4, \infty}}$ and $\|\psi\|_{W^{4, \infty}} + \|L \psi\|_{W^{4, \infty}}$ are sufficiently small. Then we can write
\begin{align*}
& -\sum_{n = 1}^\infty \frac{c_n}{2n + 1} \px \int_{\R^{2n + 1}} \Tb_n(\etab_n) \hat{\varphi}(\eta_1, t) \hat{\varphi}(\eta_2, t) \dotsm \hat{\varphi}(\eta_{2n + 1}, t) e^{i (\eta_1 + \eta_2 + \dotsb + \eta_{2n + 1}) x} \diff{\etab_n}\\
= ~& \px \log|\px| \Big[T_{B^{\log}[\varphi]} \varphi(x, t)\Big] + \px \Big[T_{B^0[\varphi]} \varphi(x, t)\Big] + \Rc_1,
\end{align*}
\begin{align*}
& -\sum_{n = 1}^\infty \frac{c_n}{2n + 1} \px \int_{\R^{2n + 1}} \Tb_n(\etab_n) \hat{\psi}(\eta_1, t) \hat{\psi}(\eta_2, t) \dotsm \hat{\psi}(\eta_{2n + 1}, t) e^{i (\eta_1 + \eta_2 + \dotsb + \eta_{2n + 1}) x} \diff{\etab_n}\\
= ~& \px \log|\px| \Big[T_{B^{\log}[\psi]} \psi(x, t)\Big] + \px \Big[T_{B^0[\psi]} \psi(x, t)\Big] + \Rc_2,
\end{align*}
where the symbols $B^{\log}[f]$ and $B^0[f]$ are defined by
\begin{align}
\label{defBsqg}
\begin{split}
B^{\log}[f](\cdot, \xi) &= \sum_{n = 1}^\infty B^{\log}_n[f](\cdot, \xi), \quad B^0[f](\cdot, \xi) = \sum_{n = 1}^\infty B^0_n[f](\cdot, \xi),\\
B^{\log}_n[f](\cdot, \xi) &= -2 c_n \F_{\zeta}^{-1}\bigg[\int_{\R^{2n}} \delta\bigg(\zeta - \sum_{j = 1}^{2n} \eta_j\bigg) \prod_{j = 1}^{2n} \bigg(i \eta_j \hat{f}(\eta_j) \chi\Big(\frac{(2n + 1) \eta_j}{\xi}\Big)\bigg) \diff{\hat{\etab}_n}\bigg],\\
B^0_n[f](\cdot, \xi) &= 2 c_n \F_{\zeta}^{-1}\bigg[\int_{\R^{2n}} \delta\bigg(\zeta - \sum_{j = 1}^{2n} \eta_j\bigg) \prod_{j = 1}^{2n} \bigg(i \eta_j \hat{f}(\eta_j) \chi\Big(\frac{(2n + 1) \eta_j}{\xi}\Big)\bigg)\\
& \hspace{3in} \cdot \int_{[0, 1]^{2n}} \log\bigg|\sum_{j = 1}^{2n} \eta_j s_j\bigg| \diff{\hat{\s}_n} \diff{\hat{\etab}_n}\bigg].
\end{split}
\end{align}
Here $\chi$ is the cutoff function in the Weyl para-product \eqref{weyldef}, $\hat{\etab}_n = (\eta_1, \eta_2, \dotsc, \eta_{2n})$, and $\hat{\s}_n = (s_1, s_2, \dotsc, s_{2n})$. The operators $T_{B^{\log}[f]}$ and $T_{B^0[f]}$ are self-adjoint and their symbols satisfy the estimates
\begin{equation}
\label{B-est}
\begin{aligned}
\|B^{\log}[f]\|_{\M_{(3, 3)}} &\lesssim \sum_{n = 1}^\infty C(n,s) |c_n| \|f\|_{W^{4, \infty}}^{2n},\\
\|B^{0}[f]\|_{\M_{(3, 3)}} &\lesssim \sum_{n = 1}^\infty  C(n,s) |c_n| \Big(\|Lf\|_{W^{4, \infty}}^{2n} + \|f\|_{W^{4, \infty}}^{2n}\Big),
\end{aligned}
\end{equation}
while the remainder term $\Rc_1$ satisfies
\begin{equation*}
\begin{aligned}
\|\Rc_1\|_{H^s} \lesssim \|\vp\|_{H^s} \left\{\sum_{n = 1}^\infty  C(n, s) |c_n| \Big(\|\varphi\|_{W^{4, \infty}}^{2n} + \|L \varphi\|_{W^{4, \infty}}^{2n}\Big)\right\},\\
\|\Rc_2\|_{H^s} \lesssim \|\psi\|_{H^s} \left\{\sum_{n = 1}^\infty  C(n, s) |c_n| \Big(\|\psi\|_{W^{4, \infty}}^{2n} + \|L \psi\|_{W^{4, \infty}}^{2n}\Big)\right\},
\end{aligned}
\end{equation*}
where the constants $C(n,s)$ have at most exponential growth in $n$.
\end{proposition}

By Proposition \ref{varphipsinonlindecomp01} and Proposition \ref{varphinonlindecomp1}, we can write \eqref{sqgsys-expd1} in the following paralinearized form
\begin{equation}\label{fsqgeq}
\begin{aligned}
&\vp_t - (\Theta_+ - \Theta_-) (\gamma + \log{h}) \vp_x + T_{\Bff_\vp} \vp_x+ \Rc_1 + 2 \Theta_- K_0(2 h |\px|) \psi_x = \Theta_+ L \big[(2 - T_{B^{\log}[\vp]}) \vp\big]_x,\\
&\psi_t + (\Theta_+ - \Theta_-) (\gamma + \log{h}) \psi_x + T_{\Bff_\psi} \psi_x+ \Rc_2 + 2 \Theta_+ K_0(2 h |\px|) \vp_x = \Theta_- L \big[(2 - T_{B^{\log}[\psi]}) \psi\big]_x,\\
\end{aligned}
\end{equation}
where
\begin{align*}
\Bff_\vp= \Theta_- \sum_{n = 1}^\infty \sum_{\ell = 0}^n (2n - \ell + 1) d_{n, \ell, 0, 1} \varphi^{2n - \ell} +\Theta_-\Bff_1[\vp, \psi]+\Theta_+B^0[\vp],\\
\Bff_\psi= \Theta_+ \sum_{n = 1}^\infty \sum_{\ell = 0}^n (2n - \ell + 1) d_{n, \ell, 0, 1} \psi^{2n - \ell} +\Theta_+\Bff_1[\psi, \vp]+\Theta_-B^0[\psi],
\end{align*}
and $\Rc_1$ and $\Rc_2$ are bounded by
\begin{align}\label{remterm}
\|\Rc_i\|_{H^s}\lesssim (\|\vp\|_{H^s}+\|\psi\|_{H^s}) F(\|\vp\|_{W^{4,\infty}}+\|L\vp\|_{W^{4,\infty}}+\|\psi\|_{W^{4,\infty}}+\|L\psi\|_{W^{4,\infty}}), \quad i=1,2,
\end{align}
where $F$ is a positive polynomial.

\subsection{Energy estimates and local existence}
We can therefore define homogeneous and nonhomogeneous weighted energies that are equivalent to the $H^s$-energies by
\begin{align*}
&E^{(j)}(t) = \int_\R |\Theta_+| |D|^j \varphi(x, t) \cdot \Big(2 - T_{B^{\log}[\vp]}\Big)^{2j + 1} |D|^j \vp(x, t) + |\Theta_-| |D|^j \psi(x, t) \cdot \Big(2 - T_{B^{\log}[\psi]}\Big)^{2j + 1} |D|^j \psi(x, t) \diff{x},\\
&\tilde{E}^{(s)}(t) = \|\vp\|_{L^2(\R)}^2 + \|\psi\|_{L^2(\R)}^2+ \sum_{j=1}^s E^{(j)}(t).
\end{align*}
For simplicity, we consider only integer norms with $s\in \N$.

We now are ready to prove the following \emph{a priori} estimates.
\begin{proposition}\label{apriori}
Let $s \geq 5$ be an integer and $\vp$, $\psi$ a smooth solution of \eqref{fsqgeq} with $\vp_0, \psi_0 \in {H}^s(\R)$.
There exists a constant $\tilde{C} > 1$, depending only on $s$, such that if
$\vp_0, \psi_0$ satisfies
\[
\big\|T_{B^{\log}[\vp_0]}\|_{L^2 \to L^2} \leq C,\qquad
\sum_{n = 1}^\infty \tilde{C}^n |c_n| \Big(\|\varphi_0\|_{W^{4, \infty}}^{2n} + \|L \varphi_0\|_{W^{4, \infty}}^{2n}\Big) < \infty,
\]
\[
\big\|T_{B^{\log}[\psi_0]}\|_{L^2 \to L^2} \leq C,\qquad
\sum_{n = 1}^\infty \tilde{C}^n |c_n| \Big(\|\psi_0\|_{W^{4, \infty}}^{2n} + \|L \psi_0\|_{W^{4, \infty}}^{2n}\Big) < \infty,
\]
for some constant $0 < C < 2$, then there exists a time $\Time > 0$ such that
\[
\big\|T_{B^{\log}[\vp(t)]}\|_{L^2 \to L^2} < 2,\qquad
\sum_{n = 1}^\infty \tilde C^n |c_n| \Big(\|\varphi(t)\|_{W^{4, \infty}}^{2n} + \|L \varphi(t)\|_{W^{4, \infty}}^{2n}\Big) < \infty,
\]
\[
\big\|T_{B^{\log}[\psi(t)]}\|_{L^2 \to L^2} < 2,\qquad
\sum_{n = 1}^\infty \tilde C^n |c_n| \Big(\|\psi(t)\|_{W^{4, \infty}}^{2n} + \|L \psi(t)\|_{W^{4, \infty}}^{2n}\Big) < \infty,
\]
for all $t\in [0,T]$, and
\begin{equation}
\frac{\diff}{\diff{t}} \tilde{E}^{(s)}(t) \le  F\left(\|\vp\|_{W^{4,\infty}} + \|L \vp\|_{W^{4,\infty}}+\|\psi\|_{W^{4,\infty}} + \|L \psi\|_{W^{4,\infty}}\right) \tilde{E}^{(s)}(t),
\label{apest}
\end{equation}
where $F(\cdot)$ is an increasing, continuous, real-valued function.
\end{proposition}

\begin{proof}
Observe that $\|\vp\|_{L^2(\R)}^2 + \|\psi\|_{L^2(\R)}^2$ is conserved by the system. So we only need to estimate the higher-order energy.
By direct calculation, for $f=\vp$ or $\psi$,
\begin{equation}\label{weighted12}
\pt (2 - T_{B^{\log}[f]})^s f = (2 - T_{B^{\log}[f]})^s f_t - s (2 - T_{B^{\log}[f]})^{s - 1} T_{\pt B^{\log}[f]} \psi + \Rc(f),
\end{equation}
where the remainder term $\Rc$ is bounded by \eqref{remterm}.

By continuity in time, there exists $T>0$ such that
\[
\sum_{n = 1}^\infty \tilde C^n |c_n| \Big(\|\varphi(t)\|_{W^{4, \infty}}^{2n} + \|L \varphi(t)\|_{W^{4, \infty}}^{2n}+\|\psi(t)\|_{W^{4, \infty}}^{2n} + \|L \psi(t)\|_{W^{4, \infty}}^{2n}\Big) < \infty
\qquad\text{for all $0\le t \le T$}.
\]

We apply the operator $|D|^s$ to the first equation of \eqref{fsqgeq} to get
\begin{equation}\label{Dseqn}
\begin{aligned}
|D|^s \vp_t & - (\Theta_+ - \Theta_-) (\gamma + \log{h}) |D|^s\vp_x+
|D|^s T_{\Bff_\vp} \vp_x \\
& +|D|^s \Rc_1+ 2 \Theta_- |D|^s K_0(2 h |\px|) \psi_x(x, t) = |D|^s \px L  \big[(2 - T_{B^{\log}[\vp]}) \vp\big].
\end{aligned}\end{equation}
Using Lemma \ref{lem-DsT}, we find that
\[
\begin{aligned}
|D|^s\left[(2 - T_{B^{\log}[\vp]}) \vp\right]&=2|D|^s\vp-|D|^s(T_{B^{\log}[\vp]} \vp)\\
&=2|D|^s\vp-T_{B^{\log}[\vp]} |D|^s\vp + sT_{\px B^{\log}[\vp]} |D|^{s-2} \vp_x + \Rc_{3},
\end{aligned}
\]
where
\[
\|\px \Rc_{3}\|_{L^2} \lesssim \bigg(\sum_{n = 1}^\infty C(n,s) |c_n| \|\varphi\|_{W^{4, \infty}}^{2n}\bigg) \|\vp\|_{H^{s-1}} .
\]
Thus, we can write the right-hand side of \eqref{Dseqn} as
\[
\begin{aligned}
\px L |D|^s & \left[(2 - T_{B^{\log}[\vp]}) \vp\right]\\
& = \px L \left[(2 - T_{B^{\log}[\vp]}) |D|^s\vp + s T_{\px B^{\log}[\vp]} |D|^{s-2} \vp_x\right]+\Rc_{4}\\
& =L\big\{(2 - T_{B^{\log}[\vp]}) |D|^s\vp_x - T_{\px B^{\log}[\vp]} |D|^s\vp - s T_{\px B^{\log}[\vp]} |D|^{s} \vp \big\} + \Rc_{4}\\
&= L\left\{(2 - T_{B^{\log}[\vp]}) |D|^s\vp_x - (s + 1) T_{\px B^{\log}[\vp]} |D|^s \vp \right\} + \Rc_{4},
\end{aligned}\]
where
\[
\|\Rc_4\|_{L^2}\lesssim \bigg(\sum_{n = 1}^\infty C(n,s) |c_n| \Big(\|\varphi\|_{W^{4, \infty}}^{2n} + \|L \varphi\|_{W^{4, \infty}}^{2n}\Big)\bigg) \|\vp\|_{H^s}.
\]

Applying $(2 - T_{B^{\log}[\vp]})^s$ to \eqref{Dseqn}, and
commuting $(2 - T_{B^{\log}[\vp]})^s$ with $L$ up to remainder
terms, we obtain that
\begin{equation}
\label{JDseqn}
\begin{aligned}
(2 - T_{B^{\log}[\vp]})^s & |D|^s \vp_t - (\Theta_+ - \Theta_-) (\gamma + \log{h}) (2 - T_{B^{\log}[\vp]})^s|D|^s\vp_x\\
&+ (2 - T_{B^{\log}[\vp]})^s \px |D|^s T_{\Bff_\vp} \vp+ 2 \Theta_-  (2 - T_{B^{\log}[\vp]})^s |D|^s K_0(2 h |\px|) \psi_x(x, t) \\
&\qquad = L\left\{(2 - T_{B^{\log}[\vp]})^{s + 1} |D|^s\vp_x - (s + 1) (2 - T_{B^{\log}[\vp]})^s T_{\px B^{\log}[\vp]} |D|^s \vp \right\} + \Rc_{5}\\
&\qquad=  \px L\left\{(2 - T_{B^{\log}[\vp]})^{s + 1} |D|^s \vp\right\} + \Rc_{5},
\end{aligned}
\end{equation}
where $\|\Rc_{5}\|_{L^2}$ is bounded by the right-hand-side of \eqref{remterm}.

By \eqref{weighted12}, the time derivative of $E^{(s)}(t)$ is
\bel{dtE}\begin{aligned}
\frac{\diff}{\diff{t}} E^{(s)}(t) & = - \int_{\R} (2s+1) |D|^s\vp\cdot (2 - T_{B^{\log}[\vp]})^{2s} T_{\pt B^{\log}[\vp]} |D|^s\vp\diff{x}\\
& + 2\int_{\R}|D|^s\vp \cdot (2 - T_{B^{\log}[\vp]})^{2s+1} |D|^s\vp_t\diff{x} + \int_{\R}\Rc\left(|D|^s\vp\right) |D|^s\vp \diff{x}.
\end{aligned}\eeq
We will estimate each of the terms on the right-hand side of \eqref{dtE}.

Equation \eqref{fsqgeq} implies that
\[
\|\vp_{xt}\|_{L^\infty} \lesssim \sum_{n = 1}^\infty C(n,s) |c_n| \Big(\|\varphi\|_{W^{4, \infty}}^{2n} + \|L \varphi\|_{W^{4, \infty}}^{2n}+\|\psi\|_{W^{4, \infty}}^{2n} + \|L \psi\|_{W^{4, \infty}}^{2n}\Big),\]
so the first term on the right-hand side of \eqref{dtE} can be estimated by
\[
\begin{aligned}
&\left| \int_{\R} (2s+1)|D|^s\vp \cdot (2 - T_{B^{\log}[\vp]})^{2s} T_{\pt B^{\log}[\vp]} |D|^s\vp\diff{x} \right|\\
&\qquad\lesssim  \bigg(\sum_{n = 1}^\infty C(n,s) |c_n| \Big(\|\varphi\|_{W^{4, \infty}}^{2n} + \|L \varphi\|_{W^{4, \infty}}^{2n}+\|\psi\|_{W^{4, \infty}}^{2n} + \|L \psi\|_{W^{4, \infty}}^{2n}\Big)\bigg) \|\vp\|_{H^s}^2.
\end{aligned}
\]
We can estimate the third term  on the right-hand side of \eqref{dtE} by
\begin{align*}
 \int_{\R}\Rc \left(|D|^s\vp\right) |D|^s\vp \diff{x} \lesssim  \bigg(\sum_{n = 1}^\infty C(n,s) |c_n| \Big(\|\varphi\|_{W^{4, \infty}}^{2n} + \|L \varphi\|_{W^{4, \infty}}^{2n}+\|\psi\|_{W^{4, \infty}}^{2n} + \|L \psi\|_{W^{4, \infty}}^{2n}\Big)\bigg) \|\vp\|_{H^s} \|\vp\|_{H^{s - 1}}.
 \end{align*}

To estimate the second term  on the right-hand side \eqref{dtE}, we multiply \eqref{JDseqn} by $(2 - T_{B^{\log}[\vp]})^{s+1} |D|^s \vp$, integrate the result with respect to $x$, and use the self-adjointness of
$(2 - T_{B^{\log}[\vp]})^{s + 1}$, which gives
\[
\int_{\R} |D|^s \vp \cdot (2 - T_{B^{\log}[\vp]})^{2s + 1} |D|^s \vp_t \diff{x}
= \Rm{1} + \Rm{2} + \Rm{3}+\Rm{4},
\]
where
\begin{align*}
\Rm{1} &= - \int_{\R} |D|^s\vp\cdot (2 - T_{B^{\log}[\vp]})^{2s+1} |D|^s \partial_x T_{\Bff_\vp} \vp \diff{x},
\\
\Rm{2} &= \int_\R (2 - T_{B^{\log}[\vp]})^{s+1} |D|^s \vp \cdot \px L (2 - T_{B^{\log}[\vp]})^{s+1} |D|^s \vp \diff{x},
\\
\Rm{3}&=\int _\R(2 - T_{B^{\log}[\vp]})^{s+1} |D|^s \vp \cdot \left(\Rc_{1}+2\Theta_-K_0(2h|\px|)\psi_x \right)\diff{x},\\
\Rm{4}&= - \int_{\R} |D|^s\vp \cdot (\Theta_+ - \Theta_-) (\gamma + \log{h}) (2 - T_{B^{\log}[\vp]})^{2s+1}|D|^{s}\vp_x.
\end{align*}
We have $\Rm{2}=0$, since $\partial_x L$ is skew-symmetric, and
\[
|\Rm{3}| \lesssim  \bigg(\sum_{n = 0}^\infty C(n,s) |c_n| \Big(\|\varphi\|_{W^{4, \infty}}^{2n} + \|L \varphi\|_{W^{4, \infty}}^{2n}+\|\psi\|_{W^{4, \infty}}^{2n} + \|L \psi\|_{W^{4, \infty}}^{2n}\Big)\bigg) (\|\vp\|_{H^s}^2+\|\psi\|_{H^s}^2).
\]

Because $(2-T_{B^{\log}[\vp]})$ is self-adjoint,
\begin{align*}
\Rm{4}&= - (\Theta_+ - \Theta_-) (\gamma + \log{h}) \int_{\R} (2 - T_{B^{\log}[\vp]})^{2s+1} |D|^s\vp\cdot  |D|^{s}\vp_x\diff x\\
&= (\Theta_+ - \Theta_-) (\gamma + \log{h}) \int_{\R} \partial_x(2 - T_{B^{\log}[\vp]})^{2s+1} |D|^s\vp\cdot  |D|^{s}\vp\diff x\\
&=-\Rm{4} + (\Theta_+ - \Theta_-) (\gamma + \log{h}) \int_{\R} [\partial_x, (2 - T_{B^{\log}[\vp]})^{2s+1}] |D|^s\vp\cdot  |D|^{s}\vp\diff x.
\end{align*}
By a commutator estimate,
\[
\left|\int_{\R} [\partial_x, (2 - T_{B^{\log}[\vp]})^{2s+1}] |D|^s\vp\cdot  |D|^{s}\vp\diff x\right|\lesssim \|\vp\|_{H^s}^2 F(\|\vp\|_{W^{4,\infty}} + \|L \vp\|_{W^{4,\infty}}).
\]
Therefore
\[
|\Rm{4}|\lesssim \|\vp\|_{H^s}^2 F(\|\vp\|_{W^{4,\infty}} + \|L \vp\|_{W^{4,\infty}}).
\]

\noindent {\bf Term} $\Rm{1}$ {\bf estimate.} We write  $\Rm{1} = -\Rm{1}_a + \Rm{1}_b$, where
\begin{align*}
\Rm{1}_a & =\int_{\R}|D|^s\vp\cdot (2 - T_{B^{\log}[\vp]})^{2s+1} \partial_x  T_{\Bff_\vp} |D|^s \vp \diff{x},\\
 \Rm{1}_b&= \int_{\R} |D|^s\vp\cdot (2 - T_{B^{\log}[\vp]})^{2s+1} \partial_x [T_{\Bff_\vp}, |D|^s] \vp \diff{x}.
\end{align*}
By a commutator estimate and \eqref{B-est}, the second integral satisfies
\[
|\Rm{1}_b| \lesssim  \bigg(\sum_{n = 1}^\infty C(n,s) |c_n| \Big(\|\varphi\|_{W^{4, \infty}}^{2n} + \|L \varphi\|_{W^{4, \infty}}^{2n}+\|\psi\|_{W^{4, \infty}}^{2n} + \|L \psi\|_{W^{4, \infty}}^{2n}\Big)\bigg) \|\vp\|_{H^s}^2.
\]
To estimate the first integral, we write it as
\begin{align*}
\Rm{1}_a &= \Rm{1}_{a_1} - \Rm{1}_{a_2},
\end{align*}
where
\begin{align*}
\Rm{1}_{a_1}& = \int_{\R} |D|^s \vp \cdot [(2 - T_{B^{\log}[\vp]})^{2s + 1}, \partial_x] \left(T_{\Bff_\vp} |D|^s \vp\right) \diff{x},
\\
\Rm{1}_{a_2}  &= \int_{\R} |D|^s\vp_x \cdot(2 - T_{B^{\log}[\vp]})^{2s+1} \left(T_{\Bff_\vp} |D|^s\vp\right) \diff{x}.
\end{align*}

\noindent {\bf Term} $\Rm{1}_{a_1}$ {\bf estimate.} A Kato-Ponce commutator estimate and \eqref{B-est} gives
\[
|\Rm{1}_{a_1}| \lesssim  \bigg(\sum_{n = 1}^\infty C(n,s) |c_n| \Big(\|\varphi\|_{W^{4, \infty}}^{2n} + \|L \varphi\|_{W^{4, \infty}}^{2n}+\|\psi\|_{W^{4, \infty}}^{2n} + \|L \psi\|_{W^{4, \infty}}^{2n}\Big)\bigg) \|\vp\|_{H^s}^2.
\]

\noindent {\bf Term} $\Rm{1}_{a_2}$ {\bf estimate.} We have
\bel{eqBs}\begin{aligned}
\Rm{1}_{a_2}
&=\int_{\R} \left(T_{B^{0}[\vp]} |D|^s \vp\right) \cdot (2 - T_{B^{\log}[\vp]})^{2s + 1} |D|^s \vp_x \diff{x}\\
&=\int_{\R} \left(T_{B^0[\vp]} |D|^s \vp\right) \cdot \left\{\partial_x\left((2 - T_{B^{\log}[\vp]})^{2s + 1} |D|^s \vp\right)-\left[\partial_x, (2 - T_{B^{\log}[\vp]})^{2s + 1}\right] |D|^s \vp \right\}\diff{x}\\
&=- \int_{\R} \partial_x \left(T_{B^0[\vp]} |D|^s \vp\right) \cdot (2 - T_{B^{\log}[\vp]})^{2s + 1} |D|^s \vp \diff{x}\\
& \qquad - \int_\R \left(T_{B^0[\vp]} |D|^s \vp\right) \cdot \left[\partial_x, (2 - T_{B^{\log}[\vp]})^{2s + 1}\right] |D|^s \vp \diff{x}\\
&=- \int_{\R} \left(T_{B^0[\vp]} |D|^s \vp_x + \left[\partial_x, T_{B^0[\vp]}\right] |D|^s \vp\right) \cdot (2 - T_{B^{\log}[\vp]})^{2s + 1} |D|^s \vp \diff{x}\\
& \qquad - \int_\R \left(T_{B^0[\vp]} |D|^s \vp\right) \cdot \left[\partial_x, (2 - T_{B^{\log}[\vp]})^{2s + 1}\right] |D|^s \vp \diff{x}.
\end{aligned}\eeq
Using commutator estimates and \eqref{B-est}, we get that
\[
\begin{aligned}
\left\|\left[\partial_x,T_{B^0[\vp]}\right] |D|^s\vp\right\|_{L^2} &\lesssim \bigg(\sum_{n = 1}^\infty C(n,s) |c_n| \left(\|\vp\|_{W^{4,\infty}}^2 + \|L\vp\|_{W^{4,\infty}}^2\right)\bigg) \|\vp\|_{H^s},\\
\left\|\left[\partial_x,(2 - T_{B^{\log}[\vp]})^{2s+1}\right] |D|^s\vp\right\|_{L^2} &\lesssim  \bigg(\sum_{n = 1}^\infty C(n,s) |c_n| \Big(\|\varphi\|_{W^{4, \infty}}^{2n} + \|L \varphi\|_{W^{4, \infty}}^{2n}\Big)\bigg) \|\vp\|_{H^s},\\
\left\|\partial_x \left[(2 - T_{B^{\log}[\vp]})^{2s+1}, T_{B^0[\vp]}\right]|D|^s\vp \right\|_{L^2} &\lesssim  \bigg(\sum_{n = 1}^\infty C(n,s) |c_n| \Big(\|\varphi\|_{W^{4, \infty}}^{2n} + \|L \varphi\|_{W^{4, \infty}}^{2n}\Big)\bigg) \|\vp\|_{H^s}^2.
\end{aligned}
\]
Since $T_{B^0[\vp]}$ is self-adjoint, we can rewrite \eqref{eqBs} as
\[
\begin{aligned}
\Rm{1}_{a_2}
& = - \Rm{1}_{a_2} + \Rc_{6},
\end{aligned}\]
with
\[
|\Rc_{6}| \lesssim  \bigg(\sum_{n = 1}^\infty C(n,s) |c_n| \Big(\|\varphi\|_{W^{4, \infty}}^{2n} + \|L \varphi\|_{W^{4, \infty}}^{2n}+\|\psi\|_{W^{4, \infty}}^{2n} + \|L \psi\|_{W^{4, \infty}}^{2n}\Big)\bigg) \|\vp\|_{H^s}^2,
\]
and we conclude that
\[
|\Rm{1}_{a_2}| \lesssim  \bigg(\sum_{n = 1}^\infty C(n,s) |c_n| \Big(\|\varphi\|_{W^{4, \infty}}^{2n} + \|L \varphi\|_{W^{4, \infty}}^{2n}+\|\psi\|_{W^{4, \infty}}^{2n} + \|L \psi\|_{W^{4, \infty}}^{2n}\Big)\bigg) \|\vp\|_{H^s}^2.
\]

By a similar procedure, we can obtain the estimate for $\psi$. This completes the estimate of the terms on the right hand side of \eqref{dtE}. Collecting the above estimates and using the interpolation inequalities, we obtain that
\begin{equation}
\tilde{E}^{(s)}(t) \le \tilde{E}^{(s)}(0) + \int_0^t F\left(\|\vp\|_{W^{4,\infty}} + \|L \vp\|_{W^{4,\infty}}+\|\psi\|_{W^{4, \infty}} + \|L \psi\|_{W^{4, \infty}}\right) \|\vp\|_{H^s}^2 \diff{t'},
\label{apest1}
\end{equation}
where $F$ is a positive, increasing, continuous, real-valued function.

We observe that there exists a constant $\tilde C(s) > 0$ such that $C(n,s)\lesssim \tilde C(s)^n$. The series in $F$ then converges whenever $\|\varphi\|_{W^{4, \infty}} + \|L\varphi\|_{W^{4, \infty}}+\|\psi\|_{W^{4, \infty}}^{2n} + \|L \psi\|_{W^{4, \infty}}^{2n}$ is sufficiently small, and we can choose $F$ to be an increasing, continuous, real-valued function.

Finally, since $\|2 - T_{B^{\log}[\vp_0]}\|_{L^2 \to L^2} \ge 2-C$,  and $\|B^{\log}[\vp](\cdot, t)\|_{\M_{(3,3)}}$
and $F\left(\|\varphi\|_{W^{4, \infty}} + \|L\varphi\|_{W^{4, \infty}}\right)$ are continuous in time, there exist $\Time>0$ and $m>0$, depending only on the initial data, such that
\[
\|2-T_{B^{\log}[\vp(t)]}\|_{L^2 \to L^2} \geq m \qquad \text{for $0\le t\le\Time$}.
\]
We therefore obtain that
\[
m^{2s+1}(\|\vp\|_{H^s}^2+\|\vp\|_{H^s}^2)\leq \tilde{E}^{(s)}\leq 2^{2s+1} (\|\vp\|_{H^s}^2+\|\vp\|_{H^s}^2),
\]
so \eqref{apest1} implies \eqref{apest}.
\end{proof}

\section{A priori estimates for the two-front GSQG ($1 < \alpha \leq 2$) systems}\label{sec:lwp12}
When $1 < \alpha \leq 2$, we write the two-front GSQG systems \eqref{eulersys} and \eqref{gsqgsys12} in the form \eqref{reg-GSQG-nc}, and define the energy
\[
\tilde{E}^{(s)}(t) = \|\vp(t)\|_{H^s}^2+\|\psi(t)\|_{H^s}^2.
\]

\begin{proposition}
Let $s \geq 3$ be an integer. Suppose $\vp_0, \psi_0\in H^s(\R)$ and
\[
2 h - \psi_0(x) + \vp_0(x) > 0\ \text{for all}\ x \in \R.
\]
Then for the smooth solutions to the initial value problem \eqref{reg-GSQG-nc},
\[
\frac{\diff}{\diff t}E^{(s)}(t)\leq P(E^{(s)}(t)),
\]
where $P$ is a positive polynomial.
\end{proposition}
\begin{proof}
It is clear that under the assumption $\vp_0, \psi_0\in H^s(\R)$, the pointwise condition $2 h - \psi_0(x) + \vp_0(x) > 0$ implies that there exists some $c > 0$ and $\ve > 0$ such that for any $x \in \R$ and $-c < \zeta < c$, we have $|2 h - \psi_0(x + \zeta) + \vp(x)| > \ve$. Without loss of generality, we fix $c = 1$ in the rest of the proof.

Observe that $\tilde{E}^{(0)}(t)$ is conserved by the system. In the following, we fix $1 \leq k \leq s$. We directly estimate the $\dot H^k$ norm of $\varphi$ and $\psi$. We apply $\partial_x^k$ to \eqref{reg-GSQG-nc},
\begin{align*}
\begin{split}
& \partial_x^k\varphi_t(x, t) + v \partial_x^k \varphi_x(x, t) + \Theta_+ \partial_x^k\L_1 \varphi_x(x, t) + \Theta_- \partial_x^k\L_2 \psi_x(x, t)\\
& \quad + \Theta_+ \px^{k} \int_\R \left[G(\sqrt{\zeta^2+(\varphi(x + \zeta, t) - \varphi(x, t))^2})-G(|\zeta|)\right] \left[\varphi_x(x + \zeta, t) - \varphi_x(x, t)\right] \diff{\zeta}\\
& \quad+ \Theta_- \px^{k} \int_\R  \left[G(\sqrt{\zeta^2+(-2h+\psi(x+\zeta,t)-\vp(x,t))^2})-G(\sqrt{\zeta^2+(2h)^2})\right] \left[\psi_x(x + \zeta, t) - \varphi_x(x, t)\right] \diff{\zeta} = 0,\\
& \partial_x^k\psi_t(x, t) - v \partial_x^k\psi_x(x, t) + \Theta_- \partial_x^k\L_1 \psi_x(x, t) + \Theta_+\partial_x^k \L_2 \varphi_x(x, t)\\
& \quad + \Theta_- \px^{k} \int_\R \left[G(\sqrt{\zeta^2+(\psi(x + \zeta, t) - \psi(x, t))^2})-G(|\zeta|)\right] \left[\psi_x(x + \zeta, t) - \psi_x(x, t)\right] \diff{\zeta}\\
& \quad+ \Theta_+ \px^{k} \int_\R  \left[G(\sqrt{\zeta^2+(2h+\vp(x+\zeta,t)-\psi(x,t))^2})-G(\sqrt{\zeta^2+(2h)^2})\right] \left[\vp_x(x + \zeta, t) - \psi_x(x, t)\right] \diff{\zeta} = 0,
\end{split}
\end{align*}
multiply the first equation by $\partial_x^k\varphi$ and the second equation by $\partial_x^k\psi$, take the sum and integrate
with respect to $x$. The terms involving
\begin{align*}
& v \partial_x^k\varphi_x(x, t) + \Theta_+ \partial_x^k\L_1 \varphi_x(x, t),
\\
& - v \partial_x^k\psi_x(x, t) + \Theta_- \partial_x^k\L_1 \psi_x(x, t),
\end{align*}
vanish. Therefore, letting $C_k^i$ denote the binomial coefficients, we obtain
\begin{eqnarray}\nonumber
&&\frac{\diff}{\diff t} \int_{\R} \frac{1}2 |\partial_x^k\varphi|^2+\frac{1}2 |\partial_x^k\psi|^2\\\nonumber
&=&-\int_{\R} \Theta_- \L_2 \partial_x^{k+1}\psi(x, t) \partial_x^k\varphi(x,t)  + \Theta_+ \L_2 \partial_x^{k+1}\varphi(x, t) \partial_x^k\psi(x,t) \diff x\\\nonumber
&&  -\sum\limits_{i=0}^k C_k^i\Theta_+ \iint_{\R^2}\partial_x^k\varphi(x,t) \big[\partial_x^i\varphi_x(x + \zeta, t) - \partial_x^i\varphi_x(x, t)\big] \partial_x^{k-i}\left[G(\sqrt{\zeta^2+(\vp(x+\zeta, t)-\vp(x,t))^2})-G(\zeta)\right]\diff{\zeta}\diff x\\\nonumber
&& - \sum\limits_{i=0}^k C_k^i \Theta_- \iint_{\R^2}\partial_x^k\varphi(x,t) \big[\partial_x^i\psi_x(x + \zeta, t) - \partial_x^i\varphi_x(x, t)\big] \\\nonumber
&&\qquad\cdot\partial_x^{k-i}\left[G(\sqrt{\zeta^2+(-2h+\psi(x+\zeta, t)-\vp(x,t))^2})-G(\sqrt{\zeta^2+(2h)^2})\right] \diff{\zeta}\diff x\\\nonumber
&&  - \sum\limits_{i=0}^k C_k^i \Theta_- \iint_{\R^2}\partial_x^k\psi(x,t) \big[\partial_x^i\psi_x(x + \zeta, t) - \partial_x^i\psi_x(x, t)\big]\partial_x^{k-i} \left[G(\sqrt{\zeta^2+(\psi(x + \zeta, t) - \psi(x, t))^2})-G(|\zeta|)\right] \diff{\zeta}\diff x\\\nonumber
&&  - \sum\limits_{i=0}^k C_k^i \Theta_+ \iint_{\R^2}\partial_x^k\psi(x,t) \big[\partial_x^i\varphi_x(x + \zeta, t) - \partial_x^i\psi_x(x, t)\big]\\\label{kthenergy}
&&\qquad\cdot \partial_x^{k-i} \left[G(\sqrt{\zeta^2 + (2h + \vp(x + \zeta, t) - \psi(x, t))^2}) - G(\sqrt{\zeta^2 + (2h)^2})\right] \diff{\zeta}\diff x.
\end{eqnarray}
Then we estimate each term on the right-hand-side.

{\bf 1.} By Lemma \ref{BesselProp}, the first integral on the right-hand-side is bounded by
\[
\bigg|\int_{\R} \Theta_- \L_2 \partial_x^{k+1}\psi(x, t) \partial_x^k\varphi(x,t)  + \Theta_+ \L_2 \partial_x^{k+1}\varphi(x, t) \partial_x^k\psi(x,t) \diff x\bigg| \lesssim (|\Theta_+|+|\Theta_-|) \|\partial_x^k\psi\|_{L^2}\|\partial_x^k\vp\|_{L^2}.
\]

{\bf 2. } A term in sum of the second integral on the right-hand-side of \eqref{kthenergy} is
\begin{align*}
\begin{split}
&\iint_{\R^2}\partial_x^k\varphi(x,t) \big[\partial_x^i\varphi_x(x + \zeta, t) - \partial_x^i\varphi_x(x, t)\big] \partial_x^{k-i}\left[G(\sqrt{\zeta^2+(\vp(x+\zeta, t)-\vp(x,t))^2})-G(\zeta)\right] \diff{\zeta}\diff x\\
=~&\iint_{\R^2}\partial_x^k\varphi(x,t) \big[\partial_x^i\varphi_x(x + \zeta, t) - \partial_x^i\varphi_x(x, t)\big]\\
&\cdot \partial_x^{k-i-1}\left[G'(\sqrt{\zeta^2+(\vp(x+\zeta, t)-\vp(x,t))^2})\frac{\vp(x+\zeta, t)-\vp(x,t)}{\sqrt{\zeta^2+(\vp(x+\zeta, t)-\vp(x,t))^2}}(\varphi_x(x + \zeta, t) - \varphi_x(x, t))\right] \diff{\zeta}\diff x.
\end{split}
\end{align*}

{\bf (a)} For $i=0, \dotsm, k-1$, we have
\begin{align*}
\begin{split}
&\left|\iint_{\R^2}\partial_x^k\varphi(x,t) \big[\partial_x^i\varphi_x(x + \zeta, t) - \partial_x^i\varphi_x(x, t)\big] \partial_x^{k-i} \left[G(\sqrt{\zeta^2+(\vp(x+\zeta, t)-\vp(x,t))^2})-G(\zeta)\right] \diff x\diff{\zeta}\right|\\
\lesssim & \|\partial_x^k\varphi(t)\|_{L^2}\cdot \int_{\R}\bigg(\int_{\R} \bigg\{\big[\partial_x^i\varphi_x(x + \zeta, t) - \partial_x^i\varphi_x(x, t)\big]\\
&\cdot \partial_x^{k-i-1}\bigg[G'(\sqrt{\zeta^2+(\vp(x+\zeta, t)-\vp(x,t))^2})\frac{\vp(x+\zeta, t)-\vp(x,t)}{\sqrt{\zeta^2+(\vp(x+\zeta, t)-\vp(x,t))^2}}(\partial_x\varphi(x + \zeta, t) -\partial_x \varphi(x, t))\bigg] \bigg\}^2\diff x\bigg)^{1/2}\diff{\zeta}.
\end{split}
\end{align*}
We divide the above integral into two parts: $|\zeta|>1$ and $|\zeta|\leq1$, and write
\begin{align}\label{Gprime}
G'(x)=\frac{C_\alpha}{|x|^{2-\alpha}x},
\end{align}
where $C_2=-1/{2\pi}$ and $C_\alpha=\alpha-2$ when $\alpha\in(1,2)$.

For $|\zeta|>1$, we observe that the term
\begin{multline*}
\big[\partial_x^i\varphi_x(x + \zeta, t) - \partial_x^i\varphi_x(x, t)\big]\\
\cdot \partial_x^{k-i-1}\left[G'(\sqrt{\zeta^2+(\vp(x+\zeta, t)-\vp(x,t))^2})\frac{\vp(x+\zeta, t)-\vp(x,t)}{\sqrt{\zeta^2+(\vp(x+\zeta, t)-\vp(x,t))^2}}(\partial_x\varphi(x + \zeta, t) -\partial_x \varphi(x, t))\right]
\end{multline*}
contains at most $k$-th order derivatives, and
\begin{align}\label{Gprime2}
\frac{G'(\sqrt{\zeta^2+(\vp(x+\zeta, t)-\vp(x,t))^2})}{\sqrt{\zeta^2+(\vp(x+\zeta, t)-\vp(x,t))^2}}=\frac{C_\alpha}{[\zeta^2+(\vp(x+\zeta, t)-\vp(x,t))^2]^{\frac{4-\alpha}2}}.
\end{align}
Using H\"older's inequality, with $L^2$-norms for the highest derivatives and $L^\infty$-norms for the other terms, we get that
\begin{align*}
\begin{split}
&\left|\int_{|\zeta|>1}\int_{\R}\partial_x^k\varphi(x,t) \big[\partial_x^i\varphi_x(x + \zeta, t) - \partial_x^i\varphi_x(x, t)\big] \partial_x^{k-i} \left[G(\sqrt{\zeta^2+(\vp(x+\zeta, t)-\vp(x,t))^2})-G(\zeta)\right] \diff x\diff{\zeta}\right|\\
\lesssim~ & \|\varphi(t)\|_{H^k}\|\varphi(t)\|_{H^k}\cdot  P\left(\|\vp\|_{W^{\left[\frac{k}2\right]+1,\infty}}\right) \cdot \int_{|\zeta|>1}\frac1{|\zeta|^2}\diff{\zeta}.
\end{split}
\end{align*}
where $P$ is a positive polynomial.

For $|\zeta|<1$, letting $c_{i_1 i_2 i_3}$ denote multilinear coefficients, we can expand the higher-order derivatives as
\begin{align*}
 &\int_{|\zeta|<1}\bigg(\int_{\R} \bigg\{\big[\partial_x^i\varphi_x(x + \zeta, t) - \partial_x^i\varphi_x(x, t)\big]\\
&\quad\cdot \partial_x^{k-i-1}\bigg[G'(\sqrt{\zeta^2+(\vp(x+\zeta, t)-\vp(x,t))^2})\frac{\vp(x+\zeta, t)-\vp(x,t)}{\sqrt{\zeta^2+(\vp(x+\zeta, t)-\vp(x,t))^2}}(\partial_x\varphi(x + \zeta, t) -\partial_x \varphi(x, t))\bigg] \bigg\}^2\diff x\bigg)^{1/2}\diff{\zeta}\\
=&\int_{|\zeta|<1}\bigg(\int_{\R} \bigg\{\big[\partial_x^i \varphi_x(x + \zeta, t) - \partial_x^i\varphi_x(x, t)\big]\cdot \sum\limits_{i_1+i_2+i_3=k-i-1} \\
&c_{i_1i_2i_3}\bigg[\partial_x^{i_1} \frac{G'(\sqrt{\zeta^2+(\vp(x+\zeta, t)-\vp(x,t))^2})}{\sqrt{\zeta^2+(\vp(x+\zeta, t)-\vp(x,t))^2}} \cdot \partial_x^{i_2}(\vp(x+\zeta, t)-\vp(x,t))\cdot \partial_x^{i_3}(\partial_x\varphi(x + \zeta, t) -\partial_x \varphi(x, t))\bigg] \bigg\}^2\diff x\bigg)^{1/2}\diff{\zeta}.
\end{align*}
When $i_3\neq k-1$, the above integral is bounded approximately by $\|\varphi(t)\|_{H^k}\cdot  P\left(\|\vp\|_{W^{\left\lfloor\frac{k}2\right\rfloor+1,\infty}}\right)$, where $P$ is a positive polynomial, and we get $i_3=k-1$ only when $i=i_1=i_2=0$. By using \eqref{Gprime2}, the integral is
\begin{align*}
&\int_{|\zeta|<1}\bigg(\int_{\R} \bigg\{\big[\varphi_x(x + \zeta, t) - \varphi_x(x, t)\big]\cdot\\
&\bigg[\frac{G'(\sqrt{\zeta^2+(\vp(x+\zeta, t)-\vp(x,t))^2})}{\sqrt{\zeta^2+(\vp(x+\zeta, t)-\vp(x,t))^2}} \cdot (\vp(x+\zeta, t)-\vp(x,t))\cdot \partial_x^{k-1}(\partial_x\varphi(x + \zeta, t) -\partial_x \varphi(x, t))\bigg] \bigg\}^2\diff x\bigg)^{1/2}\diff{\zeta}\\
=~&\int_{|\zeta|<1}\bigg(\int_{\R} \bigg\{\frac{\varphi_x(x + \zeta, t) - \varphi_x(x, t)}{\zeta}\cdot\\
&\bigg[\frac{C_\alpha}{\zeta^{2-\alpha}[1+(\frac{\vp(x+\zeta, t)-\vp(x,t)}{\zeta})^2]^{\frac{4-\alpha}2}} \cdot \frac{\varphi(x + \zeta, t) - \varphi(x, t)}{\zeta}\cdot \partial_x^{k-1}[\partial_x\varphi(x + \zeta, t) -\partial_x \varphi(x, t)]\bigg] \bigg\}^2\diff x\bigg)^{1/2}\diff{\zeta}.
\end{align*}
Since $1/{\zeta^{2-\alpha}}$ is integrable on $|\zeta|<1$, the above integral is bounded by $\|\varphi(t)\|_{H^k}\cdot  P\left(\|\vp\|_{W^{\left\lfloor\frac{k}2\right\rfloor+1,\infty}}\right)$, where $P$ is a positive polynomial.

{\bf (b)} For $i= k$, after integrating by parts, we can put one derivative on $G$,
\begin{eqnarray}\label{pkenergy}\nonumber
&&\iint_{\R^2}\partial_x^k\varphi(x,t) \big[\partial_x^k\varphi_x(x + \zeta, t) - \partial_x^k\varphi_x(x, t)\big] \left[G(\sqrt{\zeta^2+(\vp(x+\zeta, t)-\vp(x,t))^2})-G(\zeta)\right] \diff{\zeta}\diff x \\\nonumber
&=~&\iint_{\R^2}\partial_x^k\varphi(x,t) \partial_\zeta\partial_x^k\varphi(x + \zeta, t)  \left[G(\sqrt{\zeta^2+(\vp(x+\zeta, t)-\vp(x,t))^2})-G(\zeta)\right] \diff{\zeta}\diff x   \\ \nonumber
&&-  \iint_{\R^2}\partial_x^k\varphi(x,t)\partial_x^k\varphi_x(x, t) \left[G(\sqrt{\zeta^2+(\vp(x+\zeta, t)-\vp(x,t))^2})-G(\zeta)\right] \diff{\zeta}\diff x \\\nonumber
&=~&-\iint_{\R^2}\partial_x^k\varphi(x,t) \partial_x^k\varphi(x + \zeta, t)   \partial_\zeta \left[G(\sqrt{\zeta^2+(\vp(x+\zeta, t)-\vp(x,t))^2})-G(\zeta)\right] \diff{\zeta}\diff x   \\
&&+\frac12  \iint_{\R^2}\partial_x^k\varphi(x,t)\partial_x^k\varphi(x, t)  \partial_x \left[G(\sqrt{\zeta^2+(\vp(x+\zeta, t)-\vp(x,t))^2})-G(\zeta)\right] \diff{\zeta}\diff x \\\nonumber
&=~&\iint_{\R^2}\partial_x^k\varphi(x,t) \partial_x^k\varphi(x +\zeta, t)\\\nonumber
&&\qquad \cdot \bigg[G'(\sqrt{\zeta^2+(\vp(x+\zeta, t)-\vp(x,t))^2})\frac{\zeta+[\vp(x+\zeta, t)-\vp(x,t)][\partial_\zeta\vp(x+\zeta, t)]}{\sqrt{\zeta^2+(\vp(x+\zeta, t)-\vp(x,t))^2}}-G'(\zeta)\bigg]   \diff{\zeta}\diff x   \\ \nonumber
&&+\frac12  \iint_{\R^2}\partial_x^k\varphi(x,t)\partial_x^k\varphi(x, t)  G'(\sqrt{\zeta^2+(\vp(x+\zeta, t)-\vp(x,t))^2})\\\nonumber
&&\qquad\cdot\frac{\vp(x+\zeta, t)-\vp(x,t)}{\sqrt{\zeta^2+(\vp(x+\zeta, t)-\vp(x,t))^2}}\left(\partial_x\varphi(x + \zeta, t) -\partial_x \varphi(x, t)\right)  \diff{\zeta}\diff x.
\end{eqnarray}
By using \eqref{Gprime} again, we get
\begin{multline*}
G'(\sqrt{\zeta^2+(\vp(x+\zeta, t)-\vp(x,t))^2})\frac{\zeta+[\vp(x+\zeta, t)-\vp(x,t)][\partial_\zeta\vp(x+\zeta, t)]}{\sqrt{\zeta^2+(\vp(x+\zeta, t)-\vp(x,t))^2}}-G'(\zeta)\\
=C_\alpha\frac{\zeta+[\vp(x+\zeta, t)-\vp(x,t)][\partial_\zeta\vp(x+\zeta, t)]}{[\zeta^2+(\vp(x+\zeta, t)-\vp(x,t))^2]^{\frac{4-\alpha}{2}}}-C_\alpha\frac{1}{\zeta|\zeta|^{2-\alpha}}.
\end{multline*}

We divide the integral into $|\zeta|>1$ and $|\zeta|<1$. For $|\zeta|>1$,
\begin{align*}
\begin{split}
&\left|\int_{\R}\int_{|\zeta|>1}\partial_x^k\varphi(x,t) \big[\partial_x^k\varphi_x(x + \zeta, t) - \partial_x^k\varphi_x(x, t)\big] \left[G(\sqrt{\zeta^2+(\vp(x+\zeta, t)-\vp(x,t))^2})-G(\zeta) \right] \diff{\zeta}\diff x\right|\\
\lesssim~& \|\partial_x^k\vp\|_{L^2}^2\|\vp\|_{L^\infty}\|\vp\|_{W^{1,\infty}}.
\end{split}
\end{align*}

For $|\zeta|<1$, by \eqref{pkenergy}
\begin{align*}
\begin{split}
&\left|\int_{\R}\int_{|\zeta|<1}\partial_x^k\varphi(x,t) \big[\partial_x^k\varphi_x(x + \zeta, t) - \partial_x^k\varphi_x(x, t)\big] \left[G(\sqrt{\zeta^2+(\vp(x+\zeta, t)-\vp(x,t))^2})-G(\zeta)\right] \diff{\zeta}\diff x\right|\\
\lesssim~& \|\partial_x^k\vp\|_{L^2}^2\|\vp\|_{W^{2,\infty}}\|\vp\|_{W^{1,\infty}}.
\end{split}
\end{align*}

{\bf 3.} To estimate the third integral on the right-hand-side of \eqref{kthenergy}, it suffices to estimate
\begin{align}\label{kthint}
\begin{split}
 &\iint_{\R^2}\partial_x^k\varphi(x,t) \big[\partial_x^i\psi_x(x + \zeta, t) - \partial_x^i\varphi_x(x, t)\big]
 \\
 &\qquad \cdot\partial_x^{k-i}\left[G(\sqrt{\zeta^2+(-2h+\psi(x+\zeta, t)-\vp(x,t))^2})-G(\sqrt{\zeta^2+(2h)^2})\right] \diff{\zeta}\diff x.
 \end{split}
\end{align}
Since $h>0$, the above integrand does not have singularity at $\zeta=0$, and we only need to take care of large $|\zeta| \gg 1$.

{\bf (a)} When $i = k$, direct calculation yields
\begin{align*}
 &\iint_{\R^2}\partial_x^k\varphi(x,t) \big[\partial_x^k\psi_x(x + \zeta, t) - \partial_x^k\varphi_x(x, t)\big]\cdot\left[G(\sqrt{\zeta^2+(-2h+\psi(x+\zeta, t)-\vp(x,t))^2})-G(\sqrt{\zeta^2+(2h)^2})\right] \diff{\zeta}\diff x\\
 =~&\iint_{\R^2}\bigg[\partial_x^k\varphi(x,t)\partial_\zeta\partial_x^k\psi(x + \zeta, t) - \frac12\partial_x(\partial_x^k\varphi(x,t)^2)\bigg]\cdot\left[G(\sqrt{\zeta^2+(-2h+\psi(x+\zeta, t)-\vp(x,t))^2})-G(\sqrt{\zeta^2+(2h)^2})\right] \diff{\zeta}\diff x\\
 =~&\iint_{\R^2}- \partial_x^k\varphi(x,t)\partial_x^k\psi(x + \zeta, t) \cdot \partial_\zeta\left[G(\sqrt{\zeta^2+(-2h+\psi(x+\zeta, t)-\vp(x,t))^2})-G(\sqrt{\zeta^2+(2h)^2})\right]\\
 & + \frac12(\partial_x^k\varphi(x,t)^2) \partial_x\left[G(\sqrt{\zeta^2+(-2h+\psi(x+\zeta, t)-\vp(x,t))^2})-G(\sqrt{\zeta^2+(2h)^2})\right] \diff{\zeta}\diff x.
\end{align*}
By \eqref{Gprime}, we have
\begin{align*}
&\partial_\zeta\left[G(\sqrt{\zeta^2+(-2h+\psi(x+\zeta, t)-\vp(x,t))^2})-G(\sqrt{\zeta^2+(2h)^2})\right]\\
=~&G'(\sqrt{\zeta^2+(-2h+\psi(x+\zeta, t)-\vp(x,t))^2})\frac{\zeta+(-2h+\psi(x+\zeta, t)-\vp(x,t))\partial_\zeta\psi(x+\zeta, t)}{\sqrt{\zeta^2+(-2h+\psi(x+\zeta, t)-\vp(x,t))^2}}\\
&-G'(\sqrt{\zeta^2+(2h)^2})\frac{\zeta}{\sqrt{\zeta^2+(2h)^2}}\\
=~&C_\alpha\frac{\zeta+(-2h+\psi(x+\zeta, t)-\vp(x,t))\partial_\zeta\psi(x+\zeta, t)}{[\zeta^2+(-2h+\psi(x+\zeta, t)-\vp(x,t))^2]^{\frac{4-\alpha}2}}-C_\alpha\frac{\zeta}{(\zeta^2+(2h)^2)^{\frac{4-\alpha}2}}.
\end{align*}
When $|\zeta|>1$, the above fractions are bounded by
\begin{align*}
&\left|\partial_\zeta\left[G(\sqrt{\zeta^2+(-2h+\psi(x+\zeta, t)-\vp(x,t))^2})-G(\sqrt{\zeta^2+(2h)^2})\right]\right|\\
\lesssim ~&(1+\|\vp\|_{W^{1,\infty}}+\|\psi\|_{W^{1,\infty}})^2[\zeta^2+(-2h+\psi(x+\zeta, t)-\vp(x,t))^2]^{-\frac{2-\alpha}2},
\end{align*}
and when $|\zeta|<1$ by
\begin{align*}
&\left|\partial_\zeta\left[G(\sqrt{\zeta^2+(-2h+\psi(x+\zeta, t)-\vp(x,t))^2})-G(\sqrt{\zeta^2+(2h)^2})\right]\right|\\
\lesssim ~&\frac{1+(h+\|\vp\|_{L^\infty}+\|\psi\|_{L^\infty})(\|\vp\|_{W^{1,\infty}}+\|\psi\|_{W^{1,\infty}})}{[\zeta^2 + (2h- (\psi(x + \zeta, t) - \varphi(x, t)))^2]^{\frac{4-\alpha}2}}.
\end{align*}

Similarly, when $|\zeta|>1$
\begin{align*}
&\left|\partial_x\left[G(\sqrt{\zeta^2+(-2h+\psi(x+\zeta, t)-\vp(x,t))^2})-G(\sqrt{\zeta^2+(2h)^2})\right]\right|\\
=~&\left|G'(\sqrt{\zeta^2+(-2h+\psi(x+\zeta, t)-\vp(x,t))^2})\frac{(-2h+\psi(x+\zeta, t)-\vp(x,t))(\psi_x(x+\zeta, t)-\vp_x(x,t))}{[\zeta^2+(-2h+\psi(x+\zeta, t)-\vp(x,t))^2]^{1/2}}\right|\\
=~& \left|\frac{C_\alpha (-2h+\psi(x+\zeta, t)-\vp(x,t))(\psi_x(x+\zeta, t)-\vp_x(x,t))}{(\zeta^2+(-2h+\psi(x+\zeta, t)-\vp(x,t))^2)^{\frac{4-\alpha}{2}}} \right|\\
\lesssim~&(h+\|\vp\|_{L^\infty}+\|\psi\|_{L^\infty})\|\vp\|_{W^{1,\infty}} \left(\zeta^2+(-2h+\psi(x+\zeta, t)-\vp(x,t))^2\right)^{-\frac{4-\alpha}{2}},
\end{align*}
and when $|\zeta|<1$
\begin{align*}
&\left|\partial_x\left[G(\sqrt{\zeta^2+(-2h+\psi(x+\zeta, t)-\vp(x,t))^2})-G(\sqrt{\zeta^2+(2h)^2})\right]\right|\\
\lesssim~&\frac{(1+h+\|\vp\|_{L^\infty}+\|\psi\|_{L^\infty})\|\vp\|_{W^{1,\infty}}}{[\zeta^2 + (2h- (\psi(x + \zeta, t) - \varphi(x, t)))^2]^{\frac{4-\alpha}2}}.
\end{align*}

By the above calculations, the decay rate for large $|\zeta|$ is at least $|\zeta|^{-2}$, which is integrable. Then
\begin{align*}
 &\left|\iint_{\R^2}\partial_x^k\varphi(x,t) \big[\partial_x^k\psi_x(x + \zeta, t) - \partial_x^k\varphi_x(x, t)\big]\cdot\left[G(\sqrt{\zeta^2+(-2h+\psi(x+\zeta, t)-\vp(x,t))^2})-G(\sqrt{\zeta^2+(2h)^2})\right] \diff{\zeta}\diff x\right|\\
\lesssim~& \|\vp\|_{H^k}(\|\psi\|_{H^k}+\|\vp\|_{H^k})(1+\|\vp\|_{W^{1,\infty}}+\|\psi\|_{W^{1,\infty}})^2(1+\|[(2h- (\psi(x + \zeta, t) - \varphi(x, t)))^2]^{-\frac{4-\alpha}2}\|_{L^\infty_x L^\infty_{|\zeta|< 1}}).
\end{align*}

{\bf (b)} When $i=0,\dotsm, k-1$, we can expand the higher order derivatives in \eqref{kthint}, and then
\begin{align*}
|\eqref{kthint}|& \lesssim \left|\iint_{\R^2}\partial_x^k\varphi(x,t) \big[\partial_x^i\psi_x(x + \zeta, t) - \partial_x^i\varphi_x(x, t)\big]\cdot\partial_x^{k-i-1}\frac{ 2[2h- (\psi(x + \zeta, t) - \varphi(x, t))][\partial_x\vp(x , t) ]}{(\zeta^2 + [2h- (\psi(x + \zeta, t) - \varphi(x, t))]^2)^{\frac{4 - \alpha}{2}}} \diff{\zeta}\diff x\right|\\
&\lesssim \|\varphi(t)\|_{H^k}\|\psi(t)\|_{H^k}\cdot  P\Big(\|\vp\|_{W^{\left\lfloor\frac{k}2\right\rfloor+1,\infty}}+\|\psi\|_{W^{\left\lfloor\frac{k}2\right\rfloor+1,\infty}}\Big)(1+\|[(2h- (\psi(x + \zeta, t) - \varphi(x, t)))^2]^{-\frac{4-\alpha}2}\|_{L^\infty_x L^\infty_{|\zeta|< 1}}),
\end{align*}
where $P$ is a positive polynomial.

{\bf 4.} The last two integrals in \eqref{kthenergy} can be estimated in a similar way.
Since
\[
1+\left\|[(2h- (\psi(x + \zeta, t) - \varphi(x, t)))^2]^{-\frac{4-\alpha}2}\right\|_{L^\infty_x L^\infty_{|\zeta|< 1}} < M
\]
at the initial time for some constant $M$, by a continuity argument there is a time $T_1>0$, such that
\[
1 + \left\|[(2h- (\psi(x + \zeta, t) - \varphi(x, t)))^2]^{-\frac{4-\alpha}2}\right\|_{L^\infty_x L^\infty_{|\zeta|< 1}} < 2M
\]
for $t\in[0,T_1)$,

Therefore, we obtain that
\[
\frac{\diff}{\diff t} \left(\|\vp(t)\|_{H^s}^2+\|\psi(t)\|_{H^s}^2\right) \lesssim MP\Big(\|\vp\|_{W^{\left\lfloor\frac{s}2\right\rfloor+1,\infty}}+\|\psi\|_{W^{\left\lfloor\frac{s}2\right\rfloor+1,\infty}}\Big)  \left(\|\vp(t)\|_{H^s}^2+\|\psi(t)\|_{H^s}^2\right),
\]
and by Sobolev embedding, we get that
\[
\frac{\diff}{\diff t} \left(\|\vp(t)\|_{H^s}^2+\|\psi(t)\|_{H^s}^2\right) \lesssim MP\left(\|\vp(t)\|_{H^s}+\|\psi(t)\|_{H^s}\right).
\]
Therefore, we have a local-in-time energy estimate: there exist $T_2>0$ such that $\|\vp(t)\|_{H^s}^2+\|\psi(t)\|_{H^s}^2<+\infty$ for $t\in[0,T_2)$. Taking $T=\min\{T_1, T_2\}$, we obtain the \emph{a priori} energy estimate in the time interval $[0,T)$.
\end{proof}

\end{document}